\documentclass[12pt, a4paper]{amsart}
\usepackage{auxhook}
\usepackage{docmute}
\usepackage[utf8]{inputenc}
\usepackage[T1]{fontenc} 
\usepackage[colorlinks=true, linkcolor=blue, citecolor=green, urlcolor=red]{hyperref}
\hypersetup{breaklinks=true}
\usepackage{diagbox} 
\usepackage{amsmath,amssymb,amsthm,booktabs,longtable}
\usepackage{multirow}
\usepackage{enumitem}
\usepackage{comment}
\usepackage{graphicx}
\usepackage{accsupp}
\usepackage[all]{xy}
\usepackage{caption}
\usepackage{color}
\usepackage{mathdots}
\usepackage{here}
\usepackage{url}
\usepackage{lscape}
\usepackage{longtable}
\usepackage{array}
\usepackage{tikz}
\usepackage{booktabs}
\usetikzlibrary{patterns,decorations.pathmorphing,decorations.pathreplacing}
\theoremstyle{plain}


\newtheorem{theorem}{Theorem}[section]
\newtheorem{lemma}[theorem]{Lemma}
\newtheorem{corollary}[theorem]{Corollary}
\newtheorem{proposition}[theorem]{Proposition}
\newtheorem{fact}[theorem]{Fact}

\newtheorem*{claim}{Claim}

\theoremstyle{definition}
\newtheorem{setting}[theorem]{Setting}

\newtheorem{definition}[theorem]{Definition}

\newtheorem{notation}[theorem]{Notation}

\newtheorem{problem}[theorem]{Problem}

\newtheorem{remark}[theorem]{Remark}
\newtheorem{question}[theorem]{Question}
\newtheorem{example}[theorem]{Example}
\newtheorem{condition}[theorem]{Condition}
\usepackage{cleveref}
\theoremstyle{plain}
\newtheorem{maintheorem}{Theorem}

\crefname{maintheorem}{Theorem}{Theorem}
\Crefname{maintheorem}{Theorem}{Theorem}

\crefname{main}{theorem}{Theorem}
\Crefname{main}{Theorem}{Theorem}
\newcommand{\R}{\mathbb{R}}
\newcommand{\F}{\mathbb{F}}
\newcommand{\D}{\mathbb{D}}
\newcommand{\Z}{\mathbb{Z}}
\newcommand{\Q}{\mathbb{Q}}
\newcommand{\C}{\mathbb{C}}
\newcommand{\HA}{\mathbb{H}}
\newcommand{\N}{\mathbb{N}}
\newcommand{\Hom}{\mathop{\mathrm{Hom}}\nolimits}
\newcommand{\End}{\mathop{\mathrm{End}}\nolimits}

\newcommand{\Int}{\mathop{\mathrm{Int}}\nolimits}

\newcommand{\irr}{\operatorname{Irr}}
\newcommand{\pairing}[2]{\left\langle #1, #2 \right\rangle}
\newcommand{\ord}{\operatorname{ord}_{2}}
\newcommand{\ind}{\operatorname{index}}
\newcommand{\Hai}{\mathbf{i}}
\newcommand{\Haj}{\mathbf{j}}
\newcommand{\Hak}{\mathbf{k}}

\newcommand{\E}{{}^\exists}
\newcommand{\Ad}{\operatorname{Ad}}
\newcommand{\ad}{\operatorname{ad}}
\newcommand{\id}{\operatorname{id}}
\newcommand{\trans}{{}^t\!}
\newcommand{\im}{\operatorname{Im}}
\numberwithin{equation}{section}

\newcommand{\VE}{\text{Property (VE)}}
\newcommand{\PE}{\text{Property (E)}}
\newcommand{\VET}{\text{Property ($\widetilde{\text{VE}}$)}}
\newcommand{\Rep}{\operatorname{Rep}}
\newcommand{\diag}{\operatorname{diag}}

\makeatletter
    
    \@addtoreset{equation}{section}

\def\Hline{
  \noalign{\ifnum0=`}\fi\hrule \@height 4.\arrayrulewidth \futurelet
   \reserved@a\@xhline} 
  \makeatother
\let\oldtocsection=\tocsection
\let\oldtocsubsection=\tocsubsection
\let\oldtocsubsubsection=\tocsubsubsection
\renewcommand{\tocsection}[2]{\hspace{0em}\oldtocsection{#1}{#2}}
\renewcommand{\tocsubsection}[2]{\hspace{1em}\oldtocsubsection{#1}{#2}}
\renewcommand{\tocsubsubsection}[2]{\hspace{2em}\oldtocsubsubsection{#1}{#2}}
\makeatother
\title[Hurwitz--Radon numbers and proper actions]{Hurwitz--Radon numbers and proper actions of semisimple Lie groups}
\author{Kazuki Kannaka, Koichi Tojo}
\date{}


\address[Kazuki KANNAKA]{%
Faculty of Mathematics and Physics, Institute of Science and Engineering, Kanazawa University, Kakumamachi, Kanazawa, Ishikawa, 920-1192, JAPAN;
RIKEN Interdisciplinary Theoretical and Mathematical Sciences (\lowercase{i}THEMS), 
Wako, Saitama 351-0198, Japan.
}
\email{kannaka@se.kanazawa-u.ac.jp}

\address[Koichi TOJO]{%
Department of Mathematical Sciences, Tokai University, 4-1-1 Kitakaname, Hiratsuka-shi, Kanagawa 259-1292, Japan;
RIKEN Center for Advanced Intelligence Project, Nihonbashi 1-chome Mitsui Building, 15th floor, 1-4-1 Nihonbashi, Chuo-ku, Tokyo 103-0027, Japan.}
\email{koichi.tojo@tokai.ac.jp}

\begin{document}
\begin{abstract}

We study proper isometric actions of non-compact semisimple Lie groups on pseudo-Riemannian symmetric spaces. Motivated by Okuda's classification of semisimple symmetric spaces admitting proper $SL(2,\mathbb{R})$-actions [J.\ Differential Geom., 2013], we focus on symmetric spaces lying on the ``boundary'' of the existence of proper $SL(2,\mathbb{R})$-actions. As a rigidity result, we show that any connected non-compact semisimple Lie group acting properly on these symmetric spaces must be globally isomorphic to $Spin(n,1)$ up to compact factors. Moreover, the Hurwitz--Radon number arises as the largest value of $n$ for the existence of $Spin(n,1)$-proper actions. Our symmetric spaces include the pseudo-Riemannian hyperbolic space $\mathbf{H}_{+}^{N,N-1}$ of 
signature $(N,N-1)$.
\end{abstract}
\maketitle

\tableofcontents

\section{Introduction}
\label{section:introduction}
Let $ N $ be a positive integer expressed as $ N = 2^{4a + b}(2c+1) $ with $ a, c \in \mathbb{N} $ and $ 0 \leq b \leq 3 $. The integer
\[
\rho(N) := 8a + 2^b,
\]
called the \textit{Hurwitz–Radon number} of $N$, appears in several areas of mathematics: 
for example, 
composition of quadratic forms (Hurwitz~\cite{Hurwitz22}, Radon~\cite{Radon22}), vector fields on spheres (Adams~\cite{Adams_sphere}), non-singular bilinear maps (Lam~\cite{Lam67}), and the existence problem of compact quotients of tangential homogeneous spaces (Kobayashi–Yoshino~\cite[Thm.~1.1.6]{KobayashiYoshino05}). 

In this paper, we report that the Hurwitz–Radon number governs a rigidity phenomenon in pseudo-Riemannian geometry.

\subsection{Illustrative example of main results}
We begin with an example illustrating our main results 
(Theorems~\ref{theorem:classification_of_symmetric_space_with_very_even_condition} to \ref{thm:proper-hurwitz-radon} in Section~\ref{section:main-result}), 
namely, the symmetric space 
\begin{align*}
\mathbf{H}_{+}^{p,q} &= \{x\in \R^{p+q+1}\mid 
x_{1}^2+\cdots +x_{p}^2 - x_{p+1}^2-\cdots -x_{p+q+1}^2=-1\}\\
&\simeq O(p, q+1) / O(p, q),
\end{align*}
which is a pseudo-Riemannian manifold of signature $(p,q)$ with constant sectional curvature $-1$. 
The Lie group $O(p,q+1)$ coincides with the full isometry group of $\mathbf{H}_{+}^{p,q}$.
The following examples are contained in
$\mathbf{H}_{+}^{p,q}$: the sphere $\mathbf{S}^{q}$ $(p=0)$, 
(the disjoint union of two copies of) the real hyperbolic space $\mathbf{H}^{p}$ $(q=0)$, the de Sitter space $\mathbf{dS}^{q+1}$ $(p=1)$ and the anti-de Sitter space $\mathbf{AdS}^{1+p}$ $(q=1)$.
Note that $\mathbf{H}_{+}^{p,q}$ has signature $(q,p)$ with sectional curvature $+1$,
if we replace the pseudo-Riemannian metric $g$ by $-g$.

We focus on the special case where $p-q=1$ and prove the following:
\begin{theorem}[See Section~\ref{section:intro-classification-proper} for the proof]
\label{theorem:X(p,q)}
  Let $L$ be a connected semisimple Lie group with no compact factors. Then  
  $\mathbf{H}_{+}^{N,N-1}$  admits a proper isometric action of $L$ if and only if $L$ is globally isomorphic to $Spin(n,1)$ with $2\leq n\leq \rho(N)$.  
\end{theorem}

Here, when $n\geq 2$, 
$Spin(n,1)$ is the non-trivial double-covering group of the identity component of 
the orthogonal group $SO(n,1)$. 
Although the isometry group $O(N, N)$ of $\mathbf{H}_{+}^{N,N-1}$ contains a vast number of non-compact semisimple Lie subgroups, the isometric actions of such subgroups are never proper unless they are isomorphic to $Spin(n, 1)$.

In this paper, we explain the position of
$\mathbf{H}_{+}^{N,N-1}$ among $\mathbf{H}_{+}^{p,q}$, and clarify conceptually how the Hurwitz–Radon number $\rho(N)$
naturally arises in connection with $\mathbf{H}_{+}^{N,N-1}$ in Theorem~\ref{theorem:X(p,q)}.
We then find an interesting family of homogeneous spaces containing
$\mathbf{H}_{+}^{N,N-1}$ 
and extend Theorem~\ref{theorem:X(p,q)} to this family (Theorem~\ref{thm:classification-L}).

\subsection{Background and motivation}

Let us recall some basic notions.
A discrete subgroup $\Gamma$ of a reductive Lie group $G$ is a \emph{discontinuous group for} $X=G/H$ if the $\Gamma$-action on $X$ is properly discontinuous and free.
Then the quotient space $X_{\Gamma}:=\Gamma\backslash X$ 
is a manifold, and 
the quotient map $X\rightarrow X_{\Gamma}$
is a covering map. 
Thus $X_{\Gamma}$
becomes a $(G,X)$-manifold and is called a \emph{Clifford--Klein 
form} of $X$.
In the definition of a discontinuous group, 
the condition of proper discontinuity is more essential.
Indeed, if the freeness is dropped, 
then $X_{\Gamma}$ is no longer a manifold, 
but it carries a natural structure of an orbifold (or a $V$-manifold).
On the other hand, if the proper discontinuity condition is removed, 
the quotient space $X_{\Gamma}$ may fail to be Hausdorff.

In line with the general theme of how strongly local conditions capture global information of manifolds, such as their fundamental groups, our study is motivated by the following problem on Clifford–Klein forms:
\begin{problem}[\cite{Kobayashi-unlimit}]
\label{problem:discrete-version}
    Let $G$ be a reductive Lie group and 
    $H$ a closed reductive subgroup.
    What types of discrete groups can appear as discontinuous groups for the homogeneous space $X=G/H$?
\end{problem}

We are interested in discontinuous groups for manifolds endowed with
indefinite metrics, such as $\mathbf{H}_{+}^{p,q}$ with $pq\neq 0$. 
In the setting of Problem~\ref{problem:discrete-version}, $X$ admits a $G$-invariant pseudo-Riemannian metric and the situation where the metric is indefinite corresponds to the case where the isotropy subgroup $H$ is non-compact.
In contrast to the case where $H$ is compact, it should be emphasized that
the proper discontinuity of an action on $X$ imposes restrictions
on discrete subgroups of $G$. Such restrictions give rise to interesting phenomena in the study of discontinuous groups in pseudo-Riemannian geometry, including the following Calabi--Markus phenomenon (\cite{CalabiMarkus1962,Kobayashi89}): 
every discontinuous group for $X$ is a finite group if and only if 
the real ranks of $G$ and $H$ are equal. 
 
In this paper, we investigate the following problem as a continuous group analogue of Problem~\ref{problem:discrete-version}. 
\begin{problem}
\label{problem:classify-continuous-analog}
Under the setting of Problem~\ref{problem:discrete-version}, 
classify the isomorphism classes of connected semisimple Lie groups $L$ without compact factors such that $X$ admits a proper action of $L$.
\end{problem}

Here we say that $X=G/H$ admits 
a proper action of a Lie group $L$ if there exists a Lie group homomorphism 
$\varphi\colon L\rightarrow G$ such that 
the $L$-action map 
$L\times X\rightarrow X$
defined by 
$l\cdot gH:=\varphi(l)gH$ 
is proper.
Theorem~\ref{theorem:X(p,q)} provides a complete solution to 
Problem~\ref{problem:classify-continuous-analog} in the case where 
$X=O(N,N)/O(N,N-1)$. 

Studying Problem~\ref{problem:classify-continuous-analog} provides a natural approach to Problem~\ref{problem:discrete-version}. 
In fact, if $X$ admits a proper $L$-action, then every torsion-free discrete subgroup $\Gamma$ of $L$ gives a discontinuous group for $X$. Discontinuous groups obtained in this way are called \textit{standard} (\cite[Def.~1.4]{KasselKobayashi16}). 
Kobayashi~\cite{Kobayashi89} established a fundamental criterion (Fact~\ref{fact:properness_criterion}) to determine whether the action of a reductive Lie subgroup of $G$ on $X$ is proper, which gave rise to the Benoist--Kobayashi criterion 
for the proper discontinuity of discrete group actions (\cite{Benoist96,Kobayashi96}). 
Based on Kobayashi's criterion, standard discontinuous groups have been investigated (e.g., \cite{KobayashiYoshino05, Teduka2008, Oku13, BochenskiOkuda16,Tojo2019Classification, BochenskiTralle24}). 
Our approach to Problem~\ref{problem:classify-continuous-analog} also relies on 
his criterion. 

The role that Problem~\ref{problem:classify-continuous-analog} plays in 
Problem~\ref{problem:discrete-version} is not limited to this.
One can consider small deformations of standard discontinuous groups (e.g., \cite{Goldmannonstandard,Ghys95,Kobayashi98,Kassel12,
KannakaOkudaTojo24}).
Remarkably, under certain conditions, 
\emph{any} small deformation of 
a discontinuous group $\Gamma$ inside $G$ preserves 
the proper discontinuity of the $\Gamma$-action on $X$.
A quantitative study of 
this theorem was initiated by 
Kobayashi~\cite{Kobayashi98}
and has recently been developed by several authors including F.\ Kassel (e.g., \cite{Kassel12,GueritaudGuichardKasselWienhard2017anosov,kannaka24,KasselTholozan}).
By combining their theorems with our main theorems, the recent work \cite{KannakaKobayashi25}  
by the first named author and Kobayashi constructed Zariski-dense discontinuous groups  
of high cohomological dimension. 
In particular, it is shown that a compact seven-dimensional pseudo-Riemannian manifold of signature $(4,3)$  
with negative constant sectional curvature, obtained as a compact quotient of $\mathbf{H}_{+}^{4,3}$,  
admits rich deformations 
(see \cite[Thm.~1.3 and Sect.~5.5]{KannakaKobayashi25} for details).

We briefly mention some results related to 
Problem~\ref{problem:classify-continuous-analog}.  
Teduka~\cite{Teduka2008PJA} classified the irreducible complex semisimple
symmetric spaces admitting a proper $SL(2,\R)$-action, and
Okuda~\cite{Oku13} gave the corresponding classification in the real case.
There has also been progress on the existence and non-existence of proper
$SL(2,\R)$-actions on non-symmetric homogeneous spaces (e.g., \cite{Teduka2008,BochenskiOkuda16}).  
However, to the best of the authors'
knowledge, 
there is no work in the literature that classifies all semisimple Lie groups
acting properly and isometrically on a given pseudo-Riemannian homogeneous space,  
as in Theorem~\ref{theorem:X(p,q)}.  
If one imposes cocompactness of the action, the situation is different
(\cite{Zeghib98,Tojo2019Classification,Bochenski-Tralle-transformation-group24,BochenskiTralle24}).

\subsection{Overview of the main theorems}
\label{section:abstract}
In Section~\ref{section:main-result} we shall state our main theorems (Theorems~\ref{theorem:classification_of_symmetric_space_with_very_even_condition} to \ref{thm:proper-hurwitz-radon}).
In this subsection, we clarify the position of
$\mathbf{H}^{N,N-1}_{+}$ in Theorem~\ref{theorem:X(p,q)} 
among the spaces
$\mathbf{H}^{p,q}_{+}$,
and then give a summary of the results.

Okuda~\cite{Oku13} classified irreducible semisimple symmetric spaces
admitting a proper $SL(2,\R)$-action.
Moreover, combined with the result of Benoist~\cite{Benoist96},
he showed that a semisimple symmetric space admits a proper
$SL(2,\R)$-action if and only if it admits a properly discontinuous action
of a non-virtually abelian  group.
Motivated by this research, we focus on symmetric spaces in the ``boundary case''
where these equivalent conditions hold. 
We expect interesting rigidity phenomena for proper actions of discrete or continuous
groups on such special spaces.
Theorem~\ref{theorem:X(p,q)} provides initial evidence in this direction.

Let us explain what we mean by the ``boundary case.''
For example, 
$\mathbf{H}_{+}^{p,q}$
admits a proper $SL(2,\R)$-action if and only if either
(i) $p-q>1$, or
(ii) $p-q=1$ and $p$ is even.
As illustrated in Figure~\ref{tab:X(p,q)-cm},
$\mathbf{H}_{+}^{N,N-1}$ in
Theorem~\ref{theorem:X(p,q)} lies precisely on the ``boundary case''
for the existence of proper $SL(2,\R)$-actions.
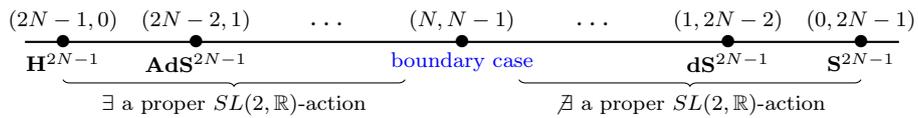
\begin{figure}[b]
\centering 
\BeginAccSupp{ActualText={A schematic diagram of the pseudo-Riemannian hyperbolic spaces Hp,q arranged along a line indexed by (p,q) with p+q=2N-1, ranging from the 2N-1 dimensional hyperbolic space H-plus-p,q at 2N-1,0 to the 2N-1 dimensional sphere at (0,2N-1). The diagram indicates that spaces corresponding to (2N-1,0) through (N+1,N-2) admit proper SL(2,R)-actions, while those corresponding to (N-1,N) through (0,2N-1) do not. The space at (N,N-1) is shown as lying on the boundary between existence and non-existence of proper SL(2,R)-actions.}}
\begin{tikzpicture}
    \path [thick, draw] (-5, 0) -- (6.5, 0);
    \coordinate (A) at (-4.5,0);
    \coordinate (B) at (-2.75,0);
    \coordinate (C) at (-1,0);
    \coordinate (D) at (0.75,0);
    \coordinate (F) at (2.5,0);
    \coordinate (G) at (4.25,0);
    \coordinate (H) at (6,0);
    
    \node at (A) {$\bullet$};
    \node [above] at (A) {\tiny $(2N-1,0)$};
    \node [below] at (A) {\tiny $\mathbf{H}^{2N-1}$};

    \node at (B) {$\bullet$};
    \node [above] at (B) {\tiny $(2N-2,1)$};
    \node [below] at (B) {\tiny $\mathbf{AdS}^{2N-1}$};

    \node [above] at (C) {\small $\cdots$};

    \node at (D) {$\bullet$};
    \node [above] at (D) {\tiny $(N,N-1)$};
    \node [below] at (D) {\tiny \textcolor{blue}{boundary case}};

    \node [above] at (F) {\small $\cdots$};

    \node at (G) {$\bullet$};
    \node [above] at (G) {\tiny $(1,2N-2)$};
    \node [below] at (G) {\tiny $\mathbf{dS}^{2N-1}$};
    
    \node at (H) {$\bullet$};
    \node [above] at (H) {\tiny $(0,2N-1)$};
    \node [below] at (H) {\tiny $\mathbf{S}^{2N-1}$};

    \draw [decorate, decoration={brace,mirror,raise=0.5cm}] (A) -- (0,0) node [pos=0.5,anchor=north,yshift=-0.55cm] {\tiny 
    $\exists$ a proper $SL(2,\mathbb{R})$-action}; 

    \draw [decoration={brace,mirror, raise=0.5cm}, decorate] (1.5,0) -- (H) node [pos=0.5,anchor=north,yshift=-0.55cm] {\tiny $\not\exists$
    a proper $SL(2,\mathbb{R})$-action }; 
\end{tikzpicture}
\EndAccSupp{}
\caption{$\mathbf{H}_{+}^{p,q}$ and proper $SL(2,\R)$-actions.} 
\label{tab:X(p,q)-cm} 
\end{figure}

In Section~\ref{section:intro-def-ve}, as one formulation of the ``boundary case'',
we introduce a condition for homogeneous spaces, called $\VE$.
Furthermore, we give an infinitesimal classification of irreducible semisimple
symmetric spaces satisfying $\VE$ and admitting a proper $SL(2,\R)$-action
(Theorem~\ref{theorem:classification_of_symmetric_space_with_very_even_condition}).
For example, 
$\mathbf{H}_{+}^{p,q}$ has
$\VE$ and admits a proper $SL(2,\R)$-action if and only if the above condition (ii) is satisfied.

In Section~\ref{section:intro-HR}, for the proof of Theorem~\ref{theorem:X(p,q)} and its generalization, 
we reformulate the Hurwitz--Radon number $\rho(N)$.  
We first observe, through the classical results of Hurwitz--Radon--Eckmann--Adams~\cite{Hurwitz22,Radon22,Eckmann43,Adams_sphere} 
(see \eqref{eq:classical-HR}), that the Hurwitz--Radon number $\rho(N)$ can be interpreted in terms of the Lie algebra 
$\mathfrak{so}(N,N)$ of the isometry group of $\mathbf{H}_{+}^{N,N-1}$ (Lemma~\ref{lemma:HR-our-interpretation}).  
Based on this observation, for a pair $(\mathfrak{g},\iota)$ consisting of a reductive Lie algebra $\mathfrak{g}$ and its representation 
$\iota\colon \mathfrak{g}\to \mathfrak{gl}(N,\C)$, we introduce variants of the Hurwitz--Radon number, denoted 
by $\rho^{(i)}(\mathfrak{g},\iota)$ ($i=1,2$).  
Moreover, as preparation for Section~\ref{section:intro-classification-proper},
we show that when $\mathfrak{g}$ is  classical and $\iota$ is its standard representation, 
the numbers $\rho^{(i)}(\mathfrak{g},\iota)$ can be computed in terms of the Hurwitz--Radon number $\rho(N)$ and the $2$-adic valuation $\ord(N)$ 
(Theorem~\ref{theorem:Hurwitz-radon}). 

In Section~\ref{section:intro-classification-proper}, we give a generalization of Theorem~\ref{theorem:X(p,q)} 
for $\mathbf{H}_{+}^{N,N-1}$ to homogeneous spaces with $\VE$.  
Recall that Theorem~\ref{theorem:X(p,q)} consists of two assertions:  
\begin{itemize}
    \item A semisimple Lie group not isomorphic to $Spin(n,1)$ does not act isometrically and properly on $\mathbf{H}_{+}^{N,N-1}$.  
    \item The Hurwitz--Radon number $\rho(N)$ arises as the maximal integer $n$ for which $Spin(n,1)$ can act isometrically and properly on $\mathbf{H}_{+}^{N,N-1}$.  
\end{itemize}
Theorems~\ref{thm:proper-imply-spin} and~\ref{thm:proper-hurwitz-radon} provide generalizations of these two statements.  
In particular, for a homogeneous space $G/H$ with $\VE$, we see that the classification result of Problem~\ref{problem:classify-continuous-analog} 
is naturally connected with the variants $\rho^{(1)}(\mathfrak{g},\iota)$ of the Hurwitz--Radon number introduced in this paper.

In Section~\ref{subsection:rigidity}, we return to Problem~\ref{problem:discrete-version}, which was our original motivation.  
In light of Theorems~\ref{thm:proper-imply-spin} and~\ref{thm:proper-hurwitz-radon}, when $G/H$ has $\VE$, it is natural to ask for strong restrictions on discrete groups that act properly discontinuously on $G/H$.  
From Theorems~\ref{thm:proper-imply-spin} and~\ref{thm:proper-hurwitz-radon}, 
we deduce one consequence for Problem~\ref{problem:discrete-version} (Corollary~\ref{cor:discrete-version}), 
and we also raise a related question (Question~\ref{problem:rigidity}). 

\subsection{Organization of the paper}
In Section~\ref{section:main-result}, we present our main theorems 
(Theorems~\ref{theorem:classification_of_symmetric_space_with_very_even_condition} to \ref{thm:proper-hurwitz-radon}) 
together with the ideas of their proofs.
A summary of the main results was already given in
Section~\ref{section:abstract}.
Theorem~\ref{theorem:classification_of_symmetric_space_with_very_even_condition} 
is proved in Section~\ref{section:classification_VE}, 
Theorem~\ref{theorem:Hurwitz-radon} in Section~\ref{section:proof-hurwitz-radon}, 
Theorem~\ref{thm:proper-imply-spin}
in Section~\ref{section:proper->spin}
and Theorem~\ref{thm:proper-hurwitz-radon} in Section~\ref{section:proof_proper-HR}.
Sections~\ref{section:proper->spin} to \ref{section:classification_VE} 
can be read independently of each other.

Section~\ref{section:proof_proper-HR} relies heavily on the results in 
Appendices~\ref{appendix:embedding} and~\ref{appendix:topocs-clifford}, and 
Section~\ref{section:proper->spin} uses 
Lemma~\ref{lemma:spin-center} in 
Appendix~\ref{section:clifford-spin}. 
In Appendix~\ref{appendix:embedding}, we prove a criterion 
(Proposition~\ref{prop:embedding-criterion}) 
for determining when a representation of a Lie group factors through 
$O(p,q)$, $O^{*}(2N)$, $Sp(N,\mathbb{R})$, or $Sp(p,q)$.  
Appendix~\ref{section:clifford-spin} collects the material on Clifford algebras, 
spin groups, and their basic structure theory associated with indefinite 
quadratic forms that is needed for this paper.  
Appendix~\ref{appendix:spin-representation} summarizes the real, quaternionic, 
orthogonal, and symplectic representation-theoretic properties of $Spin(n,1)$ 
required in the paper, including their counterparts for general signatures $(p,q)$.

\subsection{Notation and conventions} 
When referring to algebras over fields except for Lie algebras, we always assume that they are unital and associative. 
For an algebra $A$, the symbol $A^{\times}$ denotes the group of invertible elements. 

We use the symbols $\N_{+}$, $\N$, $\Z$, $\Q$, $\R$, $\C$, and $\HA = \R + \R\Hai + \R\Haj + \R\Hak$ for the positive integers, 
the non-negative integers,
the integers, and the rational, real, complex, and quaternionic numbers, respectively. 
We denote by $\ord\colon \Q^{\times}\rightarrow \Z$ the $2$-adic valuation.
Namely, if $t=2^a(\text{odd})/(\text{odd})$ $(a\in \Z)$, then $\ord t = a$.

For a set $S$, we write $\id_{S}$ for 
the identity map of $S$. 
We denote by $I_{N}$ the $N\times N$-identity matrix, and use the notation 
\[
I_{p,q}=\begin{pmatrix}
I_{p} & 0 \\
0 & -I_{q}
\end{pmatrix}.
\]

The Lie algebras of Lie groups $G,H,L,\ldots$ are denoted by the corresponding Gothic letters $\mathfrak{g},\mathfrak{h},\mathfrak{l},\ldots$, respectively.
For a Lie group homomorphism $f\colon G\to H$, its differential homomorphism $\mathfrak{g}\to\mathfrak{h}$ is denoted by $df$.
A complex Lie group $G_{\C}$ is 
    a \emph{complexification} of a Lie group $G$ if 
    there exists a Lie group embedding 
    $f\colon G\hookrightarrow G_{\C}$
    such that the $\C$-linear extension of $df\colon \mathfrak{g}\rightarrow \mathfrak{g}_{\C}$ is an isomorphism. 
    
    We refer to $\mathfrak{gl}(N,\mathbb{C})$ and the following subalgebras of  
$\mathfrak{gl}(N,\mathbb{C})$ as \emph{classical Lie algebras}:
\begin{align*}
\mathfrak{gl}(N,\mathbb{R}),\ \mathfrak{gl}(N/2,\mathbb{H}),\ \mathfrak{sl}(N,\mathbb{R}),\ \mathfrak{sl}(N,\mathbb{C}),\ \mathfrak{sl}(N/2,\mathbb{H}),\ \mathfrak{(s)u}(p,q),\\
\mathfrak{so}(p,q),\ \mathfrak{so}(N,\mathbb{C}),\ \mathfrak{so}^{*}(2(N/2)),\ \mathfrak{sp}(N/2,\mathbb{R}),\ \mathfrak{sp}(N/2,\mathbb{C}),\ \mathfrak{sp}(r,s),
\end{align*}
where $p+q = N$ and $r + s = N/2$, and 
we assume that $N/2$ is an integer whenever it appears.
The inclusion $\mathfrak{g} \hookrightarrow \mathfrak{gl}(N,\mathbb{C})$  
for a classical Lie algebra $\mathfrak{g}$ is called the \emph{standard representation}.
To treat $\mathfrak{so}(p,q)$, $\mathfrak{su}(p,q)$, and $\mathfrak{sp}(p,q)$ in a unified manner, 
it is convenient to use the notation
\[
\mathfrak{su}(p,q;\mathbb{D})
:= \{\, X \in \mathfrak{sl}(p+q,\mathbb{D}) \mid X^{*} I_{p,q}+I_{p,q}X=0 \,\},
\quad (\mathbb{D} = \mathbb{R},\, \mathbb{C},\, \mathbb{H}).
\] 

For a representation $\tau$, we denote by $[\tau]$ the equivalence class of $\tau$. 
We denote by $[\tau:\pi]$ the multiplicity
    of an irreducible representation $\pi$
    in a completely reducible representation $\tau$.

\section{Main results and ideas of the proofs}
\label{section:main-result}
In this section, we describe our main results together with outlines of their proofs.
In Section~\ref{subsec:properness_criterion},
we recall Kobayashi's criterion for proper actions of reductive Lie groups
on reductive homogeneous spaces as a preliminary.
In Sections~\ref{section:intro-def-ve} to \ref{section:intro-classification-proper},
we explain our main results
(Theorems~\ref{theorem:classification_of_symmetric_space_with_very_even_condition} to \ref{thm:proper-hurwitz-radon}).
For a summary of the main results, see Section~\ref{section:abstract}.
In Section~\ref{subsection:rigidity},
we present a result and questions concerning rigidity
of discontinuous groups arising from our main results (Corollary~\ref{cor:discrete-version} and Question~\ref{problem:rigidity}).

\subsection{Preliminary: Kobayashi's properness criterion}\label{subsec:properness_criterion}
In this subsection, we recall T.~Kobayashi's criterion 
for the properness of actions of reductive Lie groups on homogeneous spaces 
of reductive type. This criterion is applied in Sections~\ref{section:main-result},~\ref{section:proof_proper-HR} and~\ref{section:classification_VE}. 

Throughout this paper, when we deal with a homogeneous space $G/H$, we mainly assume the following setting:
 \begin{setting}
 \label{setting:semisimple}    
Let $G$ be a linear semisimple Lie group with connected complexification, 
and let $H$ be a closed \emph{reductive} subgroup of $G$. 
Namely, $H$ is stable under some Cartan involution of $G$ 
and has finitely many connected components.    
Then $G/H$ is called a homogeneous space \emph{of reductive type}.
\end{setting}

Under this setting,
we fix a maximal split abelian subalgebra $\mathfrak{a}$ of $\mathfrak{g}$ and write $W(\mathfrak{g},\mathfrak{a})$ for the Weyl group of the restricted root system of $(\mathfrak{g},\mathfrak{a})$ acting on $\mathfrak{a}$. 
We take a Cartan involution $\theta$ on $G$ such that $H$ is $\theta$-stable, and take a maximal abelian subspace $\mathfrak{a}'_{H}$
of $\mathfrak{h}^{-\theta}=\{X\in \mathfrak{h} \mid \theta(X)= -X\}$.
Then there exists $g\in G$
such that $\Ad(g)(\mathfrak{a}'_{H}) \subset \mathfrak{a}$, and put 
$\mathfrak{a}_{H}:=\Ad(g)(\mathfrak{a}'_{H})$.
The union $W(\mathfrak{g},\mathfrak{a})\cdot\mathfrak{a}_{H}:=\bigcup_{w\in W(\mathfrak{g},\mathfrak{a})}w\cdot\mathfrak{a}_{H}$ 
is independent of the choice of $\theta$, $g$ and $\mathfrak{a}'_{H}$.
Similarly, we also define a subspace $\mathfrak{a}_{L}$ of $\mathfrak{a}$ for another reductive subgroup $L$ of $G$.

\begin{fact}[Kobayashi {\cite[Thm.~4.1]{Kobayashi89}}]
\label{fact:properness_criterion}
Let $G,H,L$ be as above.
The action of $L$ on $G/H$ is proper if and only if $\mathfrak{a}_{L}\cap  (W(\mathfrak{g},\mathfrak{a})\cdot\mathfrak{a}_{H}) = \{0\}.$
\end{fact}

This criterion allows us to reduce
Problem~\ref{problem:classify-continuous-analog}
to a problem at the level of Lie algebras.

As an immediate corollary of Fact~\ref{fact:properness_criterion}, we see that properness is stable under local isomorphisms of homogeneous spaces. That is,
\begin{corollary}
    \label{cor:properness-locally-isom}
    Let $G_{1}/H_{1}$ and $G_{2}/H_{2}$ be homogeneous spaces  
as in Setting~\ref{setting:semisimple},  
and suppose there exists an isomorphism  
$\alpha\colon \mathfrak{g}_{1} \to \mathfrak{g}_{2}$ such that  
$\alpha(\mathfrak{h}_{1}) = \mathfrak{h}_{2}$.  
Let $L$ be a connected semisimple Lie group,  
and $\varphi_{i} \colon L \to G_{i}$ $(i=1,2)$ Lie group homomorphisms  
with finite kernel satisfying  
$\alpha \circ d\varphi_{1} = d\varphi_{2}$.  
Then the action of $L$ on $G_{1}/H_{1}$ via $\varphi_{1}$ is proper  
if and only if the action of $L$ on $G_{2}/H_{2}$ via $\varphi_{2}$ is proper.
\end{corollary}
In Corollary~\ref{cor:properness-locally-isom}, the assumption that each $G_{i}$ admits a connected complexification cannot be dropped.

\subsection{\texorpdfstring{$\VE$}{Property (VE)} for homogeneous spaces}
\label{section:intro-def-ve}

In this subsection, we define a notion called \emph{very even} for 
$\mathfrak{sl}(2,\R)$-homomorphisms into classical Lie algebras 
(Definition~\ref{def:iota-very-even}). 
We then introduce $\VE$ for homogeneous spaces 
$G/H$ (Definition~\ref{def:VE_for_G/H}), 
and give an infinitesimal classification of irreducible symmetric spaces 
with $\VE$ and admitting a proper $SL(2,\R)$-action
(Theorem~\ref{theorem:classification_of_symmetric_space_with_very_even_condition}).

Before introducing the definition of $\VE$, we make some remarks.  
The condition that an $\mathfrak{sl}(2,\R)$-homomorphism to $\mathfrak{g}$ is very even  
is not intrinsic to the Lie algebra $\mathfrak{g}$; rather, it depends on a representation $\iota$ of the Lie algebra $\mathfrak{g}$.  
Our aim is to define the notion of $\VE$ for a homogeneous space $G/H$ in a way that is independent of the choice of the representatives in its local isomorphism class.  
To this end, we introduce the following equivalence relation on pairs $(\mathfrak{g}, \iota)$.  
This equivalence relation will also play a role in our reformulation of the Hurwitz--Radon number in the next subsection.

\begin{definition}
\label{definition:classical-pair}
    Let $(\mathfrak{g}_i, \iota_i)$ $(i=1,2)$ be pairs of real Lie algebras and their representations. 
The pairs $(\mathfrak{g}_i, \iota_i)$ are said to be \emph{equivalent} if there exists a Lie algebra isomorphism $\varphi \colon \mathfrak{g}_1\to \mathfrak{g}_2$ such that 
$\iota_1$ is equivalent to $\iota_2\circ \varphi$ as a representation of $\mathfrak{g}_1$. 
Then we write \((\mathfrak{g}_1, \iota_1) \simeq (\mathfrak{g}_2, \iota_2)\).

A pair of a real Lie algebra $\mathfrak{g}$
and its representation $\iota$
is \emph{classical} if the pair $(\mathfrak{g},\iota)$ is equivalent 
to a pair of a classical real Lie algebra and its standard representation.
\end{definition}
For instance, 
the pair of $\mathfrak{so}(6,2)$ and its semispin representation (see Appendix~\ref{section:clifford-spin} for the definition)
is classical since this pair is equivalent to the pair of 
$\mathfrak{so}^{*}(8)$ and its standard representation via the triality of $\mathfrak{so}(8,\C)$.

We are now ready to introduce the definitions of being very even and of $\VE$. 
\begin{definition}
\label{def:iota-very-even}
Let $(\mathfrak{g},\iota)$ be a classical pair.
We say that a Lie algebra homomorphism $\varphi \colon \mathfrak{sl}(2,\R) \rightarrow \mathfrak{g}$ is \emph{very even for $\iota$}
if the $\mathfrak{sl}(2,\R)$-representation $\iota \circ \varphi$ decomposes as a direct sum of irreducible representations of even dimension.
\end{definition}

We formalize homogeneous spaces lying on the boundary between the existence
and non-existence of proper $SL(2,\R)$-actions
by the following property,
which captures the rarity 
of proper $SL(2,\mathbb{R})$-actions:

\begin{definition}\label{def:VE_for_G/H}
Let $G,H$ be as in Setting~\ref{setting:semisimple}.
We say that $G/H$ has \emph{$\VE$} if 
there exists a Lie algebra representation $\iota\colon \mathfrak{g}\rightarrow \mathfrak{gl}(N,\C)$ such that
\begin{enumerate}[label=$(\roman*)$]
    \item \label{def:VE_for_G/H-item:classical}
    $(\mathfrak{g},\iota)$ is a classical pair;
    \item \label{def:VE_for_G/H-item:VE}
    for any Lie algebra homomorphism $\varphi\colon \mathfrak{sl}(2,\R)\rightarrow \mathfrak{g}$,
    the action of $SL(2,\R)$ on $G/H$ via the lift $SL(2,\R)\rightarrow G$ of $\varphi$
    is proper if and only if $\varphi$ is very even for $\iota$.
\end{enumerate}
When it is necessary to specify the representation $\iota$, we say that $G/H$ has \emph{$\VE$ for $\iota$}. 
Furthermore, we say that $G/H$ has Property~$\widetilde{\text{(VE)}}$
if $G/H$ has $\VE$ for some $\iota$ which lifts to a Lie group representation
of $G$.
\end{definition}

In taking a lift of a Lie algebra homomorphism to a Lie group homomorphism in the above definition, we have used the following easy lemma.
    We will employ it throughout this paper:
    \begin{lemma}   
    \label{lemma:lift}
Let $G$ be a linear Lie group and $L$ a connected Lie group which admits a connected, simply-connected complexification. 
    Then any Lie algebra homomorphism from $\mathfrak{l}$ to $\mathfrak{g}$
    lifts to a Lie group homomorphism from $L$ to $G$.
    \end{lemma}

The following lemma ensures that 
$\VE$ is preserved under local isomorphisms of homogeneous spaces. Here, 
a homogeneous space $G_{1}/H_{1}$ is \emph{locally isomorphic} to
$G_{2}/H_{2}$ if 
there exists an isomorphism $\alpha\colon \mathfrak{g}_{1}\rightarrow \mathfrak{g}_{2}$ such that $\alpha(\mathfrak{h}_{1})=\mathfrak{h}_{2}$. 
\begin{lemma}
    \label{lem:VE-locally-isomorphic}
    Assume that two homogeneous spaces $G_{i}/H_{i}$ $(i=1,2)$ of reductive type 
    are locally isomorphic via an isomorphism $\alpha\colon \mathfrak{g}_{1}\rightarrow \mathfrak{g}_{2}$.
    If $G_{1}/H_{1}$ has $\VE$ for $\iota\colon \mathfrak{g}_{1}\rightarrow \mathfrak{gl}(N,\C)$, then $G_{2}/H_{2}$ has $\VE$ for $\iota\circ\alpha^{-1}$.
\end{lemma}

\begin{proof}
Let $\varphi\colon \mathfrak{sl}(2,\mathbb{R}) \rightarrow \mathfrak{g}_{1}$ be a Lie algebra homomorphism,  
and let $\widetilde{\varphi} \colon SL(2,\mathbb{R}) \rightarrow G_{1}$  
and $\widetilde{\alpha \circ \varphi} \colon SL(2,\mathbb{R}) \rightarrow G_{2}$  
be the lifts of $\varphi$ and $\alpha \circ \varphi$, respectively.  
By Corollary~\ref{cor:properness-locally-isom},  
the $SL(2,\mathbb{R})$-action on $G_{1}/H_{1}$ via $\widetilde{\varphi}$ is proper  
if and only if the $SL(2,\mathbb{R})$-action on $G_{2}/H_{2}$ via $\widetilde{\alpha \circ \varphi}$ is proper.

On the other hand, 
it is clear that  
$\varphi$ is very even for $\iota$  
if and only if  
$\alpha \circ \varphi$ is very even for $\iota \circ \alpha^{-1}$.  
Therefore, our assertion holds.
\end{proof}

As we will see in Section~\ref{section:intro-classification-proper}, homogeneous spaces with $\VE$ exhibit interesting phenomena in the classification of proper actions of semisimple Lie groups, involving the Hurwitz--Radon number.  
In what follows, we present several examples of symmetric spaces that have $\VE$.

Let $\sigma$ be an involutive automorphism
of a semisimple Lie group $G$, 
and $H$ an open subgroup of the fixed subgroup
$\{g\in G \mid \sigma(g) = g\}$. 
Then $G/H$ is called a \emph{semisimple symmetric space}.
Berger~\cite{Berger57}
classified irreducible semisimple 
symmetric spaces  
up to local isomorphism, 
and Okuda~\cite{Oku13} classified those admitting a proper $SL(2,\R)$-action. 
In this paper, we classify those with $\VE$ among them.

\begin{maintheorem}
\label{theorem:classification_of_symmetric_space_with_very_even_condition}
    \begin{enumerate}[label=(\roman*)]
        \item 
        Each symmetric space $G/H$ in 
        Table \ref{tab:very-even-symmetric} has Property~$(\widetilde{\text{VE}})$ for the standard representation of the classical Lie algebra $\mathfrak{g}$. 
        \item 
        \label{item:ve+sl(2)}
        Let $G/H$ be a semisimple symmetric space with
$\mathfrak{g}$ simple, and assume 
that $G$ admits a connected 
complexification. If $G/H$ has $\VE$ and 
admits a proper $SL(2,\R)$-action, then $G/H$ is locally isomorphic to one in Table~\ref{tab:very-even-symmetric}.
    \end{enumerate}
    \begin{table}
    \centering
    \begin{tabular}{cc}
        $G$ & $H$ \\
        \noalign{\hrule height 1.2pt}
        $SL(2N,\R)$ & $SO(N+1,N-1)$ \\
        $SL(2N,\C)$ & $SU(N+1,N-1)$ \\
        $SL(2N,\HA)$ & $Sp(N+1,N-1)$ \\
        $SO^{*}(4N)$ & $U(N+1,N-1)$\\
        $SO(2N,\C)$ & $SO(N+1,N-1)$ \\
        $SO(N,N)$ & $SO(p,p+1)\times SO(N-p,N-p-1)$ \\
        $SU(N,N)$ & $S(U(p,p+1)\times U(N-p,N-p-1))$\\
        $Sp(N,N)$ & $Sp(p,p+1)\times Sp(N-p,N-p-1)$ \\
        $SO^{*}(4N)$ & $SO^{*}(4p+2)\times SO^{*}(4N-4p-2)$ \\
        $SO(2N,\C)$ & $SO(2p+1,\C)\times SO(2N-2p-1,\C)$  \\
    \end{tabular}
    \caption{Symmetric spaces with $\VET$, where $0\leq p< N/2$.}
    \label{tab:very-even-symmetric}
\end{table}
\end{maintheorem}

\begin{remark}
Some symmetric spaces in Table~\ref{tab:very-even-symmetric} do not admit proper $SL(2,\R)$-actions; hence Table~\ref{tab:very-even-symmetric} does not provide a complete list of irreducible semisimple symmetric spaces with $\VE$ admitting a proper $SL(2,\mathbb{R})$-action.
However, if $N$ is even in the 5th, 6th, and 10th entries, 
the table provides a complete list. 
This fact follows either from Okuda's classification, or from Theorem~\ref{thm:classification-L} in Section~\ref{section:intro-classification-proper} via the isomorphism $SL(2,\R)\simeq Spin(2,1)$.
\end{remark}

We make a few comments on the list in Theorem~\ref{theorem:classification_of_symmetric_space_with_very_even_condition}.  
First, Table~\ref{tab:very-even-symmetric} includes the following important examples:
\begin{itemize}
    \item[6th] The homogeneous space $SO(N,N)/SO(N,N-1)$, which is isometric to the pseudo-Riemannian space form $\mathbf{H}_{+}^{N,N-1}$.
    
    \item[7th] A pseudo-Riemannian analogue of the complex hyperbolic space, given by $SU(N,N)/U(N,N-1)$. 
    This is a complex manifold of complex dimension $2N-1$, realized as the open subset of the complex projective space $P^{2N-1}\C$ consisting of lines $[v]$ such that $\langle v, v \rangle < 0$, where $\langle \cdot, \cdot \rangle$ denotes the standard hermitian form of signature $(N,N)$ on $\C^{2N}$;

    \item[10th] The complex sphere $SO(2N,\C)/SO(2N-1,\C)$ of complex dimension $2N-1$;
\end{itemize}

The symmetric spaces $SO(3,3)/O(3,\C)$ and $SO(6,2)/U(3,1)$ do not
appear in Table \ref{tab:very-even-symmetric} but have $\VE$. 
Indeed, these two symmetric spaces
are respectively locally isomorphic to $SL(4,\R)/SO(3,1)$ and $SO^*(8)/(SO^*(6)\times SO^*(2))$,  both of which are listed in Table~\ref{tab:very-even-symmetric}.
By Lemma~\ref{lem:VE-locally-isomorphic}, these symmetric spaces also have $\VE$.

The non-symmetric space $SO(8,\C)/Spin(7,\C)$ is locally isomorphic to $SO(8,\C)/SO(7,\C)$ (the fourth entry in Table~\ref{tab:very-even-symmetric}) via the triality of the Lie algebra $\mathfrak{so}(8,\C)$, and thus also has $\VE$.

Examples of homogeneous spaces that have $\VE$ but are not locally isomorphic to any symmetric space are given in 
Section~\ref{section:non-symmetric}.

\vspace{\baselineskip}

Theorem~\ref{theorem:classification_of_symmetric_space_with_very_even_condition} is proved in Section~\ref{section:classification_VE}.  
We briefly explain the idea of the proof.

Let $\varphi\colon \mathfrak{sl}(2,\R)\rightarrow \mathfrak{g}$ be a 
homomorphism, and $\varphi_{\C}\colon \mathfrak{sl}(2,\C)\rightarrow \mathfrak{g}_{\C}$ its complexification. 
Based on Kobayashi's properness criterion (Fact~\ref{fact:properness_criterion}), 
Okuda~\cite{Oku13} gave a criterion for determining whether the action of $SL(2,\R)$ on the symmetric space $G/H$ via the lift of $\varphi$ is proper, in terms of the weighted Dynkin diagram of $\varphi_\C$ and the Satake diagram of another real form $\mathfrak{g}^c$ of $\mathfrak{g}_\C$, which is called the \emph{$c$-dual} (see Fact~\ref{fact:okuda} for details).

On the other hand, we will give a criterion to determine whether $\varphi$ is very even, also in terms of its weighted Dynkin diagram (Proposition~\ref{prop:characterization_of_no_odd_dimensional}).
By comparing these two criteria, we determine whether each irreducible symmetric space has $\VE$.

\subsection{Reformulation of the \texorpdfstring{Hurwitz--Radon}{Hurwitz-Radon} number in terms of Lie algebra representations}
\label{section:intro-HR}
In this subsection, 
we reformulate the Hurwitz--Radon number $\rho(N)$ in terms of 
Lie algebra representations (Lemma~\ref{lemma:HR-our-interpretation}), 
introduce variants of $\rho(N)$ (Definition~\ref{def:HR-ours}), and determine their values in specific cases that will be needed later (Theorem~\ref{theorem:Hurwitz-radon}). In the next subsection, we will use our variants to give an application to Problem~\ref{problem:classify-continuous-analog}
for a homogeneous space $G/H$ with $\VE$ (see Theorems~\ref{thm:proper-imply-spin} and~\ref{thm:proper-hurwitz-radon}). 

Let $(\mathfrak{g},\iota)$ be a pair of a real reductive Lie algebra $\mathfrak{g}$ and a faithful representation 
$\iota\colon \mathfrak{g}\rightarrow \mathfrak{gl}(N,\C)$.
We assume that
the image $\iota(\mathfrak{g})$ is 
self-adjoint, namely, 
$\iota(\mathfrak{g})$ is closed under 
taking adjoint operators 
with respect to some positive definite 
hermitian form $\pairing{\cdot}{\cdot}$ on $\C^{N}$. 
Then there exists an involutive automorphism $\theta$ of $\mathfrak{g}$ such that
\begin{equation*}
\pairing{\iota(X)v}{w} = -\pairing{v}{\iota(\theta(X))w}
\quad (X\in \mathfrak{g},\ v,w \in \mathbb{C}^{N}). 
\end{equation*}
We denote by $\mathfrak{g} = \mathfrak{k} + \mathfrak{p}$ the eigenspace decomposition of $\mathfrak{g}$ with respect to $\theta$,  
namely, $\mathfrak{k}$ (resp.\ $\mathfrak{p}$) is the eigenspace of $\theta$ with eigenvalue $1$ (resp.\ $-1$).
If $\mathfrak{g}$ is semisimple, then any faithful representation $\iota$ of $\mathfrak{g}$ satisfies our assumption,  
and $\theta$ is a Cartan involution of $\mathfrak{g}$.

\begin{definition}
    \label{def:HR-ours}
    In the setting above,  
    we define two integers 
    $\rho^{(i)}(\mathfrak{g},\iota)$ ($i=1,2$)
    as the largest $n\in \N$ for which there exists
    an $\R$-linear map $f\colon \R^{n}\rightarrow \mathfrak{p}$ such that
    \begin{align}
        \rho^{(1)}(\mathfrak{g},\iota)&:\iota(f(v))^2=\|v\|^2 I_N \text{ for any } v\in \R^{n},
        \label{def:HR-ours-1}
        \\
        \rho^{(2)}(\mathfrak{g},\iota)&:\iota(f(v))\text{ is invertible for any non-zero } v\in \R^{n}.
        \label{def:HR-ours-2}
    \end{align}
Here, $\|v\|$ denotes the standard norm of $v\in \R^{n}$. 
\end{definition}

\begin{remark}
    \label{remark-HR-well-defined}
 The number $\rho^{(i)}(\mathfrak{g}, \iota)$ does not depend on the choice of the hermitian form $\pairing{\cdot}{\cdot}$ or the representatives of the equivalence class of the pair $(\mathfrak{g}, \iota)$ in the sense of Definition~\ref{definition:classical-pair}.  
This fact will be proved in Section~\ref{section:HR-well-defined} (see Proposition~\ref{prop:HR-pair-equiv}).   
\end{remark}

We now explain how $\rho^{(i)}(\mathfrak{g},\iota)$ gives an extension of the classical Hurwitz--Radon number $\rho(N)$. Recall the celebrated result by Hurwitz~\cite{Hurwitz22}, Radon~\cite{Radon22}, Eckmann~\cite{Eckmann43}, and Adams~\cite{Adams_sphere}:
\begin{align}
\label{eq:classical-HR}
\rho(N)=
\max\{ n\in \mathbb{N} &\mid
    \exists f\colon \mathbb{R}^{n} \rightarrow M(N,\mathbb{R})\text{ ($\mathbb{R}$-linear) s.t.\ }
    \\
    &\!^{t}f(v)f(v)=\|v\|^2I_{N}
    \text{ for any } v\in \mathbb{R}^n
    \} 
    \nonumber\\
=\max\{ n\in \mathbb{N} &\mid
    \exists f\colon \mathbb{R}^{n} \rightarrow M(N,\mathbb{R})\text{ ($\mathbb{R}$-linear) s.t.\ }
    \nonumber\\
    &f(v)\text{ is invertible for any non-zero } v\in \mathbb{R}^n 
    \}.    \nonumber
\end{align}
The second equality was established as an application of K-theory.

Our observation is that \eqref{eq:classical-HR} can be reformulated in terms of the Lie algebra $\mathfrak{so}(N,N)$ of the isometry group of $\mathbf{H}^{N,N-1}_{+}$ in Theorem~\ref{theorem:X(p,q)}: 
\begin{lemma}
\label{lemma:HR-our-interpretation}
    Let $\iota\colon \mathfrak{so}(N,N)\rightarrow\mathfrak{gl}(2N,\C)$ be the standard representation. Then, 
    we have 
    \begin{equation*}
        \rho(N)=\rho^{(1)}(\mathfrak{so}(N,N),\iota)=\rho^{(2)}(\mathfrak{so}(N,N),\iota).
    \end{equation*}
\end{lemma}
\begin{proof}
We adopt a standard realization of $\mathfrak{g}=\mathfrak{so}(N,N)$ as
\[
\{X\in M(2N,\R)\mid {}^{t}X I_{N,N}+I_{N,N}X=0\},
\]
and consider the Cartan decomposition $\mathfrak{g}=\mathfrak{k}+\mathfrak{p}$
with respect to the Cartan involution of $\mathfrak{g}$ given by $X\mapsto -{}^{t}X$.
Then we have
\[
\mathfrak{p}=\left\{\begin{pmatrix}
0 & A\\
{}^{t}A & 0
\end{pmatrix}\ \middle|\ A\in M(N,\R)\right\},
\]
and we define an $\R$-linear isomorphism
$\sigma\colon M(N,\R)\rightarrow \mathfrak{p}$ by
\[
\sigma(A)=\begin{pmatrix}
0 & A\\
{}^{t}A & 0
\end{pmatrix}.
\]
It is immediate to check that an $\R$-linear map $f\colon \R^{n}\rightarrow \mathfrak{p}$ satisfies the condition \eqref{def:HR-ours-1}
(resp.\ \eqref{def:HR-ours-2}) if and only if $\sigma^{-1}\circ f\colon \R^{n}\rightarrow M(N,\R)$
satisfies the first (resp.\ second) condition in \eqref{eq:classical-HR}.
This proves the assertion.
\end{proof}

In this paper, we determine the values of
\(\rho^{(i)}(\mathfrak{g},\iota)\)
for all classical pairs \((\mathfrak{g},\iota)\) in the sense of
Definition~\ref{definition:classical-pair}.
Moreover, we show that for most classical pairs the equality
\[
\rho^{(1)}(\mathfrak{g},\iota)
=
\rho^{(2)}(\mathfrak{g},\iota)
\]
holds.
We note that the inequality
$\rho^{(1)}(\mathfrak{g},\iota)
\leq
\rho^{(2)}(\mathfrak{g},\iota)$
always holds.

To state our theorem, we extend the Hurwitz--Radon number $\rho(N)$ to the case $N \in \Q^{\times}$. If $\ord (N) = 4a + b$ for some $a \in \Z$ and $b \in \{0,1,2,3\}$, we define
\[
\rho(N) := 8a + 2^b \in \Z.
\]
The following table lists the explicit values of $\rho(N)$.
\begin{table}[htbp]
\begin{tabular}{c!{\vrule width 1.2pt}c|c|c|c|c|c|c|c|c|c|c}
    $N$ & $\ldots$ & $1/16$ & $1/8$ & $1/4$ & $1/2$ & $1$ & $2$ & $4$ & $8$ & $16$ & $\ldots$  \\
    \hline
    $\rho(N)$ & $\ldots$ & $-7$ & $-6$ & $-4$ & $0$ 
    & $1$ & $2$ & $4$ & $8$ & $9$ & $\ldots$
\end{tabular}
\caption{Values of $\rho(N)$.}
\end{table}

\begin{maintheorem}  
\label{theorem:Hurwitz-radon}
    Let $\mathfrak{g}$ be a classical real Lie algebra and 
    $\iota$ its standard representation. 
    If $\mathfrak{g}\neq \mathfrak{sl}(2N+1,\D)$ ($\D=\R,\C,\HA$, $N\geq 1$), then 
    \[
    \rho^{(1)}(\mathfrak{g},\iota) = \rho^{(2)}(\mathfrak{g},\iota),
    \]
    and the common value is summarized in Table~\ref{tab:HR-values}.

    Moreover, in the exceptional case $\mathfrak{g} = \mathfrak{sl}(2N+1,\D)$, 
        we have 
        \[
        \rho^{(1)}(\mathfrak{sl}(2N+1,\D),\iota) = 0,\quad 
    \rho^{(2)}(\mathfrak{sl}(2N+1,\D),\iota) = 1.
    \]

\begin{table}[H]
\renewcommand{\arraystretch}{1.5}
\[
\begin{array}{ccc|ccc}
& \mathfrak{g} & \rho^{(1)}=\rho^{(2)} & &\mathfrak{g} & \rho^{(1)}=\rho^{(2)} \\
\noalign{\hrule height 1.2pt}
(a) &\mathfrak{so}(N,N) & \rho(N)  & 
(b) & \mathfrak{gl}(N,\mathbb{C}) & 2\ord(N)+1 
\\
&\mathfrak{gl}(N,\mathbb{R}) & \rho(N/2)+1  & &
\mathfrak{su}(N,N) & 2\ord(N)+2
\\
\cline{4-6}
&\mathfrak{sp}(N,\mathbb{R})   & \rho(N/2)+2 & (c) & \mathfrak{sl}(2N,\mathbb{R}) & \rho(N)+1\\ 
&\mathfrak{sp}(N,\mathbb{C})   & \rho(N/2)+3 & &
\mathfrak{sl}(2N,\mathbb{C}) & 2\ord(N)+3 \\
&\mathfrak{sp}(N,N)& \rho(N/2)+4 & &
\mathfrak{sl}(2N,\mathbb{H}) & \rho(N/2)+5
\\
\cline{4-6}
&\mathfrak{gl}(N,\mathbb{H}) & \rho(N/4)+5 & (d) & \mathfrak{sl}(1,\mathbb{D}) & 0
\\
& \mathfrak{so}^{*}(2N)& \rho(N/8) + 6 & & \mathfrak{su}(p,q;\mathbb{D})\,(p\neq q) & 0
\\
&\mathfrak{so}(N,\mathbb{C}) & \rho(N/16) + 7 &  & & \\
\end{array}
\]
\renewcommand{\arraystretch}{1}
\caption{Values of $\rho^{(i)}(\mathfrak{g},\iota)$ associated with classical pairs.}
\label{tab:HR-values}
\end{table}
\end{maintheorem}

\begin{remark}
    \begin{enumerate}
        \item Table~\ref{tab:HR-values} does not include $\mathfrak{u}(p,q)$ because 
        we have, by definition, 
\[
\rho^{(i)}(\mathfrak{u}(p,q),\iota)
=\rho^{(i)}(\mathfrak{su}(p,q),\iota),
\quad (i=1,2).
\]
        \item 
        For $\mathfrak{g} = \mathfrak{sl}(N, \D)$ and $\mathfrak{gl}(N, \D)$, the corresponding $\rho^{(i)}(\mathfrak{g},\iota)$ can be compared as follows:
\begin{align*}
\rho^{(1)}(\mathfrak{sl}(N,\D),\iota) &= 
\begin{cases}
\rho^{(1)}(\mathfrak{gl}(N,\D),\iota)    & \text{if } N \text{ is even}, \\
0    & \text{if } N \text{ is odd},
\end{cases} \\
\rho^{(2)}(\mathfrak{sl}(N,\D),\iota) &= 
\begin{cases}
\rho^{(2)}(\mathfrak{gl}(N,\D),\iota)    & \text{if } N \geq 2, \\
0    & \text{if } N = 1.
\end{cases}
\end{align*}

\item 
For a certain representation $\tau$ which is not standard, it is still possible
to compute $\rho^{(i)}(\mathfrak{g},\tau)$.
Let $\tau\colon \mathfrak{so}(6,2)\rightarrow \mathfrak{gl}(8,\C)$
be one of the semispin representations of 
$\mathfrak{so}(6,2)$. 
Then the pair $(\mathfrak{so}(6,2),\tau)$
is equivalent to $(\mathfrak{so}^{*}(8),\iota)$ for the standard representation $\iota$ of $\mathfrak{so}^{*}(8)$. Hence, 
by Theorem~\ref{theorem:Hurwitz-radon} with Remark~\ref{remark-HR-well-defined},
\[
\rho^{(i)}(\mathfrak{so}(6,2),\tau)=\rho^{(i)}(\mathfrak{so}^{*}(8),\iota)=6\quad (i=1,2).
\]
\end{enumerate}
\end{remark}

Theorem~\ref{theorem:Hurwitz-radon} will be proved in Section~\ref{section:proof-hurwitz-radon}. Here we outline the idea of the proof
for $(a)$ 
in Table~\ref{tab:HR-values}, 
which 
is 
a crucial case. 

For the eight classical Lie algebras listed in Table~\ref{tab:HR-values}~$(a)$, we make essential use of the following periodic chain of Lie algebras with period~8:
\begin{align*}
    \renewcommand{\arraystretch}{1.5}
    \begin{array}{lllll}
    \cdots &\subset \mathfrak{so}(N,N) &\subset \mathfrak{gl}(2N,\mathbb{R}) &\subset \mathfrak{sp}(2N,\mathbb{R}) &\subset \mathfrak{sp}(2N,\mathbb{C}) \\
    &\subset \mathfrak{sp}(2N,2N) &\subset \mathfrak{gl}(4N,\mathbb{H}) &\subset \mathfrak{so}^{*}(2(8N))
    &\subset \mathfrak{so}(16N,\mathbb{C}) \\
    &\subset \mathfrak{so}(16N,16N) &\subset \cdots. &
    \end{array}
    \renewcommand{\arraystretch}{1}
\end{align*}

For each adjacent pair of classical Lie algebras $(\mathfrak{g}, \mathfrak{h})$ in the above chains, we prove the inequality
\begin{equation}
\label{ineq:key-ineq}
\rho^{(i)}(\mathfrak{h},\iota)+1 \leq \rho^{(i)}(\mathfrak{g},\iota)
\quad (i=1,2)
\end{equation}
(see Corollary~\ref{cor:paraherm}).
By showing, using Lemma~\ref{lemma:HR-our-interpretation}, that this inequality is in fact an equality, a large part of the proof of case~$(a)$ in Table~\ref{tab:HR-values} follows.
The key inequality \eqref{ineq:key-ineq} can be proved by observing the following fact: each adjacent pair $(\mathfrak{g}, \mathfrak{h})$ forms a symmetric pair, and the associated symmetric pair $(\mathfrak{g}, \mathfrak{h}^{a})$ is parahermitian in the sense of Kaneyuki--Kozai~\cite{Kaneyuki_Kozai_1985}.

\vspace{\baselineskip}
Now we explain the relationship between Theorem~\ref{theorem:Hurwitz-radon} and previous works on the Hurwitz--Radon number by Adams--Lax--Phillips~\cite{Adams_Lax_Phillips65} and Au-Yeung~\cite{Yeung_71}.

For $\D=\R, \C, \HA$, we denote by $M_{x}(N, \D)$
the set of $N \times N$ matrices over $\D$ having property $x$, 
where $x$ stands for hermitian (h), skew-hermitian (sk-h),
symmetric (s), or skew-symmetric (sk-s).  
We define an integer $\D_x(N)$ as the largest $n\in \N$ for which there exists an $\R$-linear map $f\colon \R^n \to M_x(N,\D)$ such that  
$f(v)$ is invertible for any non-zero $v \in \R^n$.
Similarly, we denote by $\D(N)$ the quantity defined in the same way as above, replacing $M_x(N, \D)$ with $M(N, \D)$.
The number $\R(N)$ is nothing but the Hurwitz--Radon number $\rho(N)$ by \eqref{eq:classical-HR}.
The values of $\D(N)$ and $\D_{h}(N)$ were determined in \cite{Adams_Lax_Phillips65}, and those of the other $\D_x(N)$ were obtained in \cite{Yeung_71}.

One can verify that, except in the cases where except in the cases where $\D=\HA$ and $x$ is either $\mathrm{s}$ or $\mathrm{sk\text{-}s}$., the quantity $\D_x(N)$ coincides with $\rho^{(2)}(\mathfrak{g}, \iota)$ for some classical pair $(\mathfrak{g},\iota)$, as summarized in Table~\ref{tab:HR-ours-ALP} (note that $\R_{\mathrm{h}} = \R_{\mathrm{s}}$, $\R_{\mathrm{sk\text{-}h}} = \R_{\mathrm{sk\text{-}s}}$, and $\C_{\mathrm{h}} = \C_{\mathrm{sk\text{-}h}}$).  
Accordingly, Theorem~\ref{theorem:Hurwitz-radon} recovers the results of \cite{Adams_Lax_Phillips65} and \cite{Yeung_71}, up to these exceptions. Moreover, as we will see in the next subsection,
these quantities are naturally connected to the study of proper actions.
    \begin{table}[htbp]
    $\renewcommand{\arraystretch}{1.5}
    \begin{array}{cc|cc|cc}
    \rho^{(2)}(\mathfrak{g},\iota) & \mathfrak{g} &  \rho^{(2)}(\mathfrak{g},\iota)& \mathfrak{g} &  \rho^{(2)}(\mathfrak{g},\iota) & \mathfrak{g} \\
    \noalign{\hrule height 1.2pt}
    \R(N) & \mathfrak{so}(N,N) & \C(N) & \mathfrak{su}(N,N) & \HA(N) & \mathfrak{sp}(N,N)\\
     \R_{\text{s}}(N) & \mathfrak{gl}(N,\R) & \C_{\text{h}}(N) &  \mathfrak{gl}(N,\C) & \HA_{\text{h}}(N) & \mathfrak{gl}(N,\HA) \\
     \R_{\text{sk-s}}(N) & \mathfrak{so}(N,\C)  & \C_{\text{sk-s}}(N)& \mathfrak{so}^{*}(2N) &   \HA_{\text{sk-h}}(N)& \mathfrak{sp}(N,\C) \\ 
       &  & \C_{\text{s}}(N) & \mathfrak{sp}(N,\R) & &
\end{array}\renewcommand{\arraystretch}{1}$
\caption{Relation between $\D_x(N)$ and $\rho^{(2)}(\mathfrak{g},\iota)$.}\label{tab:HR-ours-ALP}
    \end{table}

\subsection{\texorpdfstring{Hurwitz--Radon number}{Hurwitz-Radon number} and proper actions}
\label{section:intro-classification-proper}
We now return to the main problem of this paper, Problem~\ref{problem:classify-continuous-analog}. 
We classify 
the isomorphism classes of non-compact semisimple Lie groups that act properly on homogeneous spaces with $\VET$:
\begin{theorem}
\label{thm:classification-L}
Let $L$ be a connected semisimple Lie group
with no compact factors, and 
$G/H$ a homogeneous space with $\VET$ for $\iota$. 
Then $G/H$ admits a proper $L$-action if and only if $L$ is isomorphic to $Spin(n,1)$ with $2\leq n\leq \rho^{(1)}(\mathfrak{g},\iota)$.
\end{theorem}

For a homogeneous space $G/H$ with $\VET$ for $\iota$, 
note that
$(\mathfrak{g},\iota)$ is a classical pair in the sense of 
Definition~\ref{definition:classical-pair}. 
Hence the value of $\rho^{(1)}(\mathfrak{g},\iota)$ can be computed by 
Theorem~\ref{theorem:Hurwitz-radon} together with Remark~\ref{remark-HR-well-defined}.  
Thus Theorem~\ref{thm:classification-L} provides 
a complete solution to Problem~\ref{problem:classify-continuous-analog} 
for such $G/H$, in terms of the Hurwitz--Radon number $\rho(N)$ and 
the $2$-adic valuation $\ord(N)$.

Let us prove Theorem~\ref{theorem:X(p,q)} as an illustrative example.

\begin{proof}[Proof of Theorem~\ref{theorem:X(p,q)}]
Let $L$ be a connected semisimple Lie group without compact factors.

We note the following natural diffeomorphisms: 
\[
\mathbf{H}^{N,N-1}_{+}\simeq O(N,N)/O(N,N-1)\simeq SO(N,N)/SO(N,N-1). 
\]
Since $O(N,N)$ is the full isometry group of 
$\mathbf{H}^{N,N-1}_{+}$ and $L$ is connected, 
the pseudo-Riemannian manifold 
$\mathbf{H}^{N,N-1}_{+}$ admits a proper isometric 
$L$-action 
if and only if $SO(N,N)/SO(N,N-1)$ admits a proper $L$-action.
By
Theorem~\ref{theorem:classification_of_symmetric_space_with_very_even_condition}
(see the $6$th entry of Table~\ref{tab:very-even-symmetric}), 
this homogeneous space has $\VET$ for the standard representation $\iota$ of $\mathfrak{so}(N,N)$.
Therefore, by Theorem~\ref{thm:classification-L}, the above condition holds if and only if
$L$ is isomorphic to $Spin(n,1)$ and
$2 \le n \le \rho^{(1)}(\mathfrak{so}(N,N),\iota)$.
Since $\rho^{(1)}(\mathfrak{so}(N,N),\iota)=\rho(N)$ by Lemma~\ref{lemma:HR-our-interpretation},
this completes the proof.
\end{proof}

As another example, we have: 
\begin{example}
The pseudo-Riemannian complex hyperbolic space
\[
G/H=SU(N,N)/U(N,N-1)
\]
has $\VET$ by Theorem~\ref{theorem:classification_of_symmetric_space_with_very_even_condition}. Hence 
$Spin(n,1)$ acts properly and isometrically on $G/H$
if and only if $N$ is divisible by $2^{\lceil n/2\rceil -1}$ by Theorems~\ref{theorem:Hurwitz-radon} and~\ref{thm:classification-L}.
Moreover, 
$G/H$ admits no proper isometric actions of 
connected non-compact semisimple Lie groups
other than $Spin(n,1)$, up to compact factors.
\end{example}

Theorem~\ref{thm:classification-L} follows from the following two theorems: 
\begin{maintheorem}
\label{thm:proper-imply-spin}
Let $G/H$ be a homogeneous space  with $\VE$.
Assume that  a connected semisimple Lie subgroup $L$ of $G$
without compact factors acts properly on $G/H$.
Then $L$ must be locally isomorphic to $SO(n,1)$ for some $n\geq 2$.
Further, if $G/H$ has $\VET$, then
$L$ must be globally isomorphic to $Spin(n,1)$.
\end{maintheorem}

\begin{maintheorem}
\label{thm:proper-hurwitz-radon}
Fix $n\geq 2$ and 
let $G/H$ be a homogeneous space with $\VE$ for $\iota$. Then 
$G/H$ admits a proper $Spin(n,1)$-action if and only if $n\leq \rho^{(1)}(\mathfrak{g},\iota)$.
\end{maintheorem}

Let us outline the strategy of the proofs of the two theorems.

Theorem~\ref{thm:proper-imply-spin} will be proved in Section~\ref{section:proper->spin}. 
We first characterize $\mathfrak{so}(n,1)$ among the non-compact simple Lie algebras via $\mathfrak{sl}(2,\R)$-triples (Proposition~\ref{prop:classification-not-noneven}), and use this characterization, together with $\VE$, to handle the case where $L$ is simple.  
In the general case where $L$ is not simple, the proof reduces to showing that $G/H$ does not admit a proper action of $SL(2,\R)\times SL(2,\R)$.  

Theorem~\ref{thm:proper-hurwitz-radon} will be proved in 
Section~\ref{section:proof_proper-HR}. 
For simplicity, we explain the idea of the proof 
in the case where $G$ is a classical Lie group,
$\tilde{\iota}\colon G\rightarrow GL(N,\C)$ is the standard representation,
and $G/H$ has $\VE$ for the differential homomorphism $\iota\colon \mathfrak{g}\rightarrow \mathfrak{gl}(N,\C)$ of $\tilde{\iota}$.
The proof essentially relies on the structure theory of Clifford algebras
and spin groups associated with indefinite 
quadratic forms. Their definitions and
notation are summarized in
Appendix~\ref{section:clifford-spin}.
Here we recall that $Spin(n,1)$ is realized inside the even Clifford algebra
$C^{+}(n,1)$ with unit $1$, and that $-1$ belongs to $Spin(n,1)$.

We first observe that the action of $Spin(n,1)$ on $G/H$
via a Lie group homomorphism
$\varphi\colon Spin(n,1)\rightarrow G$
is proper if and only if
\begin{equation}
\label{eq:intro-phi--1}
\tilde{\iota}\circ\varphi(-1)=-I_{N}.
\end{equation}
This follows from a combination of Kobayashi's properness criterion
(Fact~\ref{fact:properness_criterion})
and Property~(VE).

On the other hand, 
$n\leq \rho^{(1)}(\mathfrak{g},\iota)$
is equivalent to the existence of a Lie group homomorphism
$\varphi\colon Spin(n,1)\rightarrow G\subset GL(N,\C)$
such that
\begin{equation}
\label{eq:intro-phi-extend}
\varphi \text{ extends to a homomorphism from } C^{+}(n,1)\text{ to } M(N,\C).
\end{equation}
This is proved by combining standard arguments on Clifford algebras with the structure theory of real semisimple Lie groups.

The two conditions
\eqref{eq:intro-phi--1} and \eqref{eq:intro-phi-extend}
for a homomorphism $\varphi$ are not equivalent.
However, the existence of $\varphi$
satisfying \eqref{eq:intro-phi--1}
is equivalent to the existence of $\varphi$
satisfying \eqref{eq:intro-phi-extend}.
This equivalence yields Theorem~\ref{thm:proper-hurwitz-radon}.

One of the key points in proving this equivalence is that
irreducible representations of the spin group
that do not factor through the orthogonal group
behave similarly to the (semi)spin representations, with respect to representation-theoretic invariants
such as the dimension and the Cartan index
(see Section~\ref{section:embedding} for the Cartan index). 
These results are summarized, together with proofs,
in Appendix~\ref{appendix:spin-representation}.

\subsection{Questions}
\label{subsection:rigidity}
We return to our original motivation, Problem~\ref{problem:discrete-version}.  

Examining the interaction between actions of continuous groups 
(Problem~\ref{problem:classify-continuous-analog})
and those of discrete groups 
(Problem~\ref{problem:discrete-version})
has led to several significant problems,
as seen in Margulis superrigidity~\cite{Margulis_discrete_subgroup}, in  affine proper actions of non-virtually solvable
groups (e.g., \cite{Margulis-freegroup-anounce,Margulis-freegroup-1984,Danciger-Gueritaud-Kassel-affine-2020}), and in  the existence problem of 
compact Clifford–Klein forms (\cite{Kobayashi1992necessary}, \cite[Conj.~3.3.10]{KobayashiYoshino05}).
From this viewpoint, we formulate some natural questions arising from our theorem.

To begin with, observe that Theorem~\ref{thm:proper-imply-spin}, combined with Margulis superrigidity (see \cite[Thms.~VII.~5.6 and IX.~5.8]{Margulis_discrete_subgroup} and \cite[Thm.~4.1]{Corlette_superrigidity}), yields the following consequence:

\begin{corollary}
    \label{cor:discrete-version}
    Let $G/H$ be a homogeneous space with $\VE$ and $\Gamma$
    an irreducible uniform lattice of a linear semisimple Lie group $L$
    with no compact factors.
    If $L$ is either of real rank at least $2$ or locally isomorphic to $Sp(n,1)$ ($n\geq 2$) or $F_{4(-20)}$, then $G/H$ admits no properly discontinuous $\Gamma$-actions.
\end{corollary}

In view of Corollary~\ref{cor:discrete-version}, the following natural question arises, to which we currently do not have an answer:
\begin{question}
\label{problem:rigidity}
Let $G/H$ be a symmetric space with $\VE$ for $\iota$ and $\Gamma$
    an irreducible uniform lattice of a semisimple Lie group $L$.
    \begin{enumerate}[label=$(\roman*)$]
        \item 
        \label{problem:su(n,1)}
        Suppose $L=SU(n,1)$ $(n\geq 2)$. Does 
        $G/H$ admit a properly discontinuous $\Gamma$-action?
        \item 
        \label{problem:so(n,1)}
        Suppose $L=Spin(n,1)$. 
        Is it true that
        $G/H$ admits a properly discontinuous $\Gamma$-action  
        if and only if $n \leq \rho^{(1)}(\mathfrak{g}, \iota)$?
    \end{enumerate}
\end{question}

We make several remarks concerning 
Question~\ref{problem:rigidity}~\ref{problem:so(n,1)}.

First, the ``if''-part follows immediately from 
Theorem~\ref{thm:proper-hurwitz-radon}.

Second, when $n=2$, the ``only if''-part also holds.
Indeed, for symmetric spaces $G/H$, Okuda proved in \cite[Thm.~1.3]{Oku13}
that $G/H$ admits a properly discontinuous action of a surface group of any
genus $\ge 2$ if and only if it admits a proper action of
$SL(2,\R)\simeq Spin(2,1)$.
Hence the ``only if''-part follows from
Theorem~\ref{thm:proper-hurwitz-radon}.

Finally, let us mention a special case, namely,  
$n=3$ and $G/H=\mathbf{H}^{N,N-1}_{+}$. 
We note that any cocompact Kleinian group in $PSL(2,\C)$ lifts to a cocompact discrete subgroup of $SL(2,\C)\simeq Spin(3,1)$ (Culler--Shalen~\cite{Culler-Shalen83} and Thurston~\cite{Thurston-3fold}). 
Hence, the question is equivalent to the following:  
is it true that $\mathbf{H}^{N,N-1}_{+}$ admits a properly discontinuous, 
isometric action of a cocompact Kleinian group  
if and only if $N$ is divisible by $4$?  
At present the authors do not know the answer even in this special case.

\section{Semisimple Lie groups that never act properly}
\label{section:proper->spin}
This section is devoted to the proof of Theorem~\ref{thm:proper-imply-spin}. 
Namely, we show that if $G/H$ is a homogeneous space with $\VE$ (resp.\ $\VET$) in the sense of Definition~\ref{def:VE_for_G/H}, then 
it never admits proper actions of non-compact semisimple Lie groups not locally isomorphic (resp.\ globally isomorphic) to $Spin(n,1)$
up to compact factors. 

In Section~\ref{section:characterization-so(n,1)}, we give a characterization
of the non-compact simple Lie algebra $\mathfrak{so}(n,1)$.
Combining $\VE$ with this characterization, we prove Theorem~\ref{thm:proper-imply-spin}
in Section~\ref{section:proof-proper-imply-spin}.

\subsection{Characterization of the  
orthogonal Lie algebra of real rank one}
\label{section:characterization-so(n,1)}
Let $\mathfrak{g}$ be a real semisimple Lie algebra
and $\varphi\colon \mathfrak{sl}(2,\R)\rightarrow \mathfrak{g}$ a Lie algebra homomorphism.
Recall that $\varphi$ is \emph{even} if all the eigenvalues of 
$\mathrm{ad}(\varphi(A_{0}))$
in $\mathfrak{g}$ are even, 
where we have used the symbol
\[A_{0}=\begin{pmatrix}
    1 & 0 \\
    0 & -1
\end{pmatrix}\in \mathfrak{sl}(2,\R).\]

In this subsection, we show the following characterization of 
the Lie algebra $\mathfrak{so}(n,1)$ in terms of even $\mathfrak{sl}(2,\R)$-homomorphisms, which plays an important role in the proof of Theorem~\ref{thm:proper-imply-spin}:
\begin{proposition}
\label{prop:classification-not-noneven}
Let $\mathfrak{l}$ be a non-compact real simple Lie algebra. 
If any Lie algebra homomorphism from $\mathfrak{sl}(2,\R)$ to $\mathfrak{l}$
is even, then $\mathfrak{l}$ is isomorphic to $\mathfrak{so}(n,1)$ for some $n\geq 2$.
\end{proposition}

To prove this proposition, we first give several lemmas.

\begin{lemma}
    \label{lem:even-subalg}
    Let $\psi\colon \mathfrak{l}\rightarrow \mathfrak{g}$ be an 
    injective Lie algebra homomorphism of real semisimple Lie algebras.
    For a Lie algebra homomorphism $\varphi\colon \mathfrak{sl}(2,\R)\rightarrow \mathfrak{l}$, 
    if 
    $\psi \circ \varphi\colon \mathfrak{sl}(2,\R)\rightarrow \mathfrak{g}$
    is even, then so is $\varphi$.
\end{lemma}

\begin{proof}
Since $\psi\colon \mathfrak{l}\to\mathfrak{g}$ is injective, the eigenvalues of
$\ad(\varphi(A_{0}))$ on $\mathfrak{l}$ also occur as eigenvalues of
$\ad(\psi(\varphi(A_{0})))$ on $\mathfrak{g}$.
Therefore, the assertion follows from the definition of being even.
\end{proof}

\begin{lemma}
\label{lem:sl3-su21-so2,3}
Any non-compact real simple Lie algebra not isomorphic to 
$\mathfrak{so}(n,1)$ for all $n\geq 2$
contains a Lie subalgebra isomorphic to either
$\mathfrak{sl}(3,\R)$, $\mathfrak{su}(2,1)$, or $\mathfrak{so}(3,2)$. 
\end{lemma}

\begin{proof}
Using the classification of real simple Lie algebras, we prove the assertion case by case. For the classical non-compact simple Lie algebras, the proof is presented in Table~\ref{table:classical-non-compact-simple}.
\begin{table}
\renewcommand{\arraystretch}{1}
\begin{tabular}{cl}
    $\mathfrak{sl}(n,\R)$ $(n\geq 2)$ & 
        \begin{tabular}{lc}
        $\simeq \mathfrak{so}(2,1)$ &  $(n=2)$ \\
        $\supset \mathfrak{sl}(3,\R)$ & $(n\geq 3)$
        \end{tabular}\\ 
    \hline
    $\mathfrak{sl}(n,\C)$ $(n\geq 2)$ & 
        \begin{tabular}{lc}
        $\simeq \mathfrak{so}(3,1)$ &  $(n=2)$ \\
        $\supset \mathfrak{sl}(3,\R)$ & $(n\geq 3)$
        \end{tabular}\\ 
    \hline
     $\mathfrak{sl}(n,\HA)$ $(n\geq 2)$ & 
        \begin{tabular}{lc}
        $\simeq \mathfrak{so}(5,1)$ &  $(n=2)$ \\
        $\supset \mathfrak{sl}(3,\R)$ & $(n\geq 3)$
        \end{tabular}\\ 
    \hline
    $\mathfrak{so}(p,q)$ $(p\geq q\geq 1)$ & 
        \begin{tabular}{ll}
        $\mathfrak{so}(p,1)$ &  $(q=1)$ \\
        not simple & $((p,q)=(2,2))$ \\
        $\supset \mathfrak{so}(3,2)$ & (otherwise)
        \end{tabular}\\ 
    \hline
     $\mathfrak{su}(p,q)$ $(p\geq q\geq 1)$ & 
        \begin{tabular}{lc}
         $\simeq \mathfrak{so}(2,1)$ &  $(p=1)$ \\
        $\supset \mathfrak{su}(2,1)$ & $(p>1)$
        \end{tabular}\\ 
    \hline
    $\mathfrak{sp}(p,q)$ $(p\geq q\geq 1)$ & 
        \begin{tabular}{lc}
        $\simeq \mathfrak{so}(4,1)$ &  $(p=1)$ \\
        $\supset \mathfrak{su}(2,1)$ & $(p>1)$
        \end{tabular}\\ 
    \hline
    $\mathfrak{sp}(n,\R)$ $(n\geq 1)$ & 
        \begin{tabular}{lc}
        $\simeq \mathfrak{so}(2,1)$ &  $(n=1)$ \\
        $\supset \mathfrak{sp}(2,\R)\simeq \mathfrak{so}(3,2)$ & $(n\geq 2)$ 
        \end{tabular}\\ 
    \hline
    $\mathfrak{sp}(n,\C)$ $(n\geq 1)$ & 
        \begin{tabular}{lc}
        $\simeq \mathfrak{so}(3,1)$ & $(n=1)$\\
        $\supset \mathfrak{sp}(2,\R) \simeq \mathfrak{so}(3,2)$ & $(n\geq 2)$
        \end{tabular}\\ 
        \hline
    $\mathfrak{so}^{*}(2n)$ $(n\geq 2)$ & 
        \begin{tabular}{lc}
        not simple &  $(n=2)$ \\
        $\supset \mathfrak{so}^{*}(6)\simeq \mathfrak{su}(3,1)\supset \mathfrak{su}(2,1)$ & $(n\geq 3)$ 
        \end{tabular}\\ 
    \hline
    $\mathfrak{so}(n,\C)$ $(n\geq 2)$ & 
        \begin{tabular}{ll}
        not simple &  $(n=2,4)$ \\
        $\simeq \mathfrak{so}(3,1)$ &  $(n=3)$ \\
        $\supset \mathfrak{so}(3,2)$ & $(n\geq 5)$ 
        \end{tabular}\\ 
\end{tabular}
\renewcommand{\arraystretch}{1}
\caption{Non-compact classical Lie algebras.}
\label{table:classical-non-compact-simple}
\end{table}

For types $E$ and $F$, the following inclusion relations can be read off from the classification table of irreducible semisimple symmetric pairs due to Berger~\cite{Berger57} (see also 
the tables in \cite[Sect.~1]{OshimaSekiguchi84}).
\[
\begin{array}{ccccccccc}
   \mathfrak{e}_{8,\C}  & \supset & \mathfrak{e}_{7,\C} & \supset &\mathfrak{e}_{6,\C} & \supset & \mathfrak{f}_{4,\C} & \supset &\mathfrak{sp}(3,\C), \\
   \mathfrak{e}_{8(8)}  & \supset & \mathfrak{e}_{7(7)} & \supset &\mathfrak{e}_{6(6)} & \supset & \mathfrak{f}_{4(4)} & \supset &\mathfrak{sp}(2,1), \\
    & & & & \mathfrak{e}_{6(2)} & \supset &\mathfrak{f}_{4(4)} & \supset&\mathfrak{sp}(2,1),\\
       \mathfrak{e}_{8(-24)} & \supset & \mathfrak{e}_{7(-5)} & \supset & \mathfrak{e}_{6(-14)} & \supset &\mathfrak{f}_{4(-20)} & \supset & \mathfrak{sp}(2,1), \\
    & & \mathfrak{e}_{7(-25)} & \supset & \mathfrak{e}_{6(-26)} & \supset &\mathfrak{f}_{4(-20)} & \supset & \mathfrak{sp}(2,1). \\ 
\end{array}
\]
Since $\mathfrak{su}(2,1)\subset \mathfrak{sp}(2,1)\subset \mathfrak{sp}(3,\C)$, 
the assertion is proved for all the non-compact simple Lie algebras of types $E$ and $F$.

The non-compact simple Lie algebras of type $G$ are only $\mathfrak{g}_{2(2)}$ and its complexification $\mathfrak{g}_{2,\mathbb{C}}$.
It is known that $\mathfrak{g}_{2(2)}$ contains $\mathfrak{su}(2,1)$, which is proved, for example, via a construction of $\mathfrak{g}_{2(2)}$ using the split octonions over $\R$ (see Yokota~\cite[Sect.~6(1)]{Yokota_g2} for details).
This completes the proof for type $G$.
\end{proof}

We are ready to prove Proposition~\ref{prop:classification-not-noneven}.
\begin{proof}[Proof of Proposition~\ref{prop:classification-not-noneven}]
We prove the contrapositive.
Assume that $\mathfrak{l}$ is not isomorphic to $\mathfrak{so}(n,1)$ for any $n \geq 2$,  
we show that there exists a non-even Lie algebra homomorphism from $\mathfrak{sl}(2,\R)$ to $\mathfrak{l}$.

By Lemmas~\ref{lem:even-subalg}~and~\ref{lem:sl3-su21-so2,3}, 
we may assume that $\mathfrak{l}$ is 
either $\mathfrak{sl}(3,\R)$, $\mathfrak{su}(2,1)$, or $\mathfrak{so}(3,2)$. For $\mathfrak{l}=\mathfrak{sl}(3,\R)$, 
the standard inclusion $\mathfrak{sl}(2,\R)\hookrightarrow \mathfrak{sl}(3,\R)$ is non-even. For $\mathfrak{l}=\mathfrak{su}(2,1)$,
the composite of the isomorphism 
$\mathfrak{sl}(2,\R)\simeq \mathfrak{su}(1,1)$ and 
the standard inclusion 
$\mathfrak{su}(1,1)\hookrightarrow \mathfrak{su}(2,1)$
is non-even. 
For $\mathfrak{l}=\mathfrak{so}(3,2)\simeq \mathfrak{sp}(2,\R)$, 
the composite of the isomorphism $\mathfrak{sl}(2,\R)\simeq \mathfrak{sp}(1,\R)$
and the standard inclusion 
$\mathfrak{sp}(1,\R)\hookrightarrow \mathfrak{sp}(2,\R)$
is non-even. Hence the proof is completed.
\end{proof}

\subsection{Proof of 
\texorpdfstring{Theorem~\ref{thm:proper-imply-spin}}{Theorem C}}
\label{section:proof-proper-imply-spin}
We first prove several lemmas (Lemma~\ref{lem:ve->e} to Lemma~\ref{lem:very-even-odd}),  
and then complete
the proof of Theorem~\ref{thm:proper-imply-spin}. 

We say that a homogeneous space $G/H$ of reductive type in the sense of Setting~\ref{setting:semisimple} has \emph{Property~(E)}
if the following condition holds:
for any Lie algebra homomorphism $\varphi\colon \mathfrak{sl}(2,\R)\rightarrow \mathfrak{g}$,
if the $SL(2,\R)$-action on $G/H$ via the lift of $\varphi$ is proper, then $\varphi$ is even.

$\VE$ implies $\PE$ as follows: 
\begin{lemma}
    \label{lem:ve->e}
    Let $(\mathfrak{g},\iota)$ be a classical pair and $\varphi\colon \mathfrak{sl}(2,\R)\rightarrow \mathfrak{g}$ a Lie algebra homomorphism. If $\varphi$ is very even for $\iota$, then 
    it is also even. 
    In particular, 
    if a homogeneous space $G/H$ of reductive type has $\VE$ for $\iota$,
    then it also has $\PE$. 
\end{lemma}
\begin{proof}
    Since the second assertion follows from the first assertion,
    we focus on proving the first assertion.
    
    If $\varphi$ is very even for $\iota\colon \mathfrak{g}\rightarrow \mathfrak{gl}(N,\C)$, then all the eigenvalues of $\iota(\varphi(A_{0}))$ 
    are odd. Hence all the eigenvalues of $\ad(\iota(\varphi(A_{0})))$
    in $\mathfrak{gl}(N,\C)$ are even. 
    Since $\iota$ is injective, it follows that 
    all the eigenvalues of $\ad(\varphi(A_{0}))$ in 
    $\mathfrak{g}$ 
    are also even. This means that $\varphi$ is even.
\end{proof}

Using the characterization of $\mathfrak{so}(n,1)$ given in the previous subsection (Proposition~\ref{prop:classification-not-noneven}),
we obtain the following rigidity result from $\PE$,
which is weaker than Theorem~\ref{thm:proper-imply-spin}:
\begin{proposition}
    \label{prop:proper->Spin}
    Let $G/H$ be a homogeneous space with $\PE$.
    If $L$ is a non-compact simple Lie subgroup of $G$ that acts properly on $G/H$, then $L$ is locally isomorphic to $SO(n,1)$ for some $n\geq 2$.
\end{proposition}

\begin{proof}
    Take an arbitrary homomorphism 
    $\varphi\colon \mathfrak{sl}(2,\R)\rightarrow \mathfrak{l}$.
    Let $\psi \colon \mathfrak{l}\to \mathfrak{g}$ be the inclusion map.
    Since the $L$-action on $G/H$ is proper, the $SL(2,\R)$-action on $G/H$ via the lift of the composition $\psi \circ \varphi$ is proper. Therefore, $\psi\circ \varphi$ is even since $G/H$ has $\PE$. By Lemma~\ref{lem:even-subalg}, $\varphi$ is even. Hence, by Proposition~\ref{prop:classification-not-noneven}, $\mathfrak{l}$ is isomorphic to $\mathfrak{so}(n,1)$ for some $n\geq 2$. Thus the assertion is proved. 
\end{proof}

In Proposition~\ref{prop:proper->Spin}, we assume that $L$ is simple.
To treat the case where $L$ is not simple, the following lemma is needed:
\begin{lemma}
\label{lem:very-even-odd}
Let $(\mathfrak{g},\iota)$ be a classical pair in the sense of Definition~\ref{definition:classical-pair} and 
$\varphi\colon \mathfrak{sl}(2,\R)\oplus \mathfrak{sl}(2,\R)\rightarrow \mathfrak{g}$ a Lie algebra homomorphism.
If the restriction of $\varphi$ to the $i$th factor is very even for $\iota$ for $i=1,2$, then the restriction of $\varphi$ to the diagonal $\mathfrak{sl}(2,\R)$-factor 
is not very even for $\iota$.
\end{lemma}

\begin{proof}
    Let $\pi_{k}$ be the $k$-dimensional irreducible complex representation of $\mathfrak{sl}(2,\R)$ ($k\in \N$). The
    representation $\iota\circ\varphi$ of $\mathfrak{sl}(2,\R)\oplus \mathfrak{sl}(2,\R)$ is decomposed into 
    the sum of irreducible subrepresentations isomorphic to 
    the external tensor products
    $\pi_{k}\boxtimes \pi_{l}$ $(k,l\in \N)$.
    The restriction of $\pi_{k}\boxtimes \pi_{l}$ to the first (resp.\ second) $\mathfrak{sl}(2,\R)$-factor is equivalent to    
    a direct sum of copies of $\pi_{k}$ (resp.\ $\pi_{l}$).
    Hence, by the assumption on $\varphi$,
    the representation $\pi_{k}\boxtimes \pi_{l}$ does not occur in
    $\iota\circ\varphi$ unless both $k$ and $l$ are even. 
    Therefore, it follows from the Clebsch--Gordan rule 
    \[
    \pi_{k}\otimes \pi_{l}\simeq 
    \pi_{k+l-1}\oplus \pi_{k+l-3}\oplus\cdots \oplus \pi_{|k-l|+1},
    \]
    that the restriction of the
    representation $\iota\circ\varphi$
    to the diagonal $\mathfrak{sl}(2,\R)$-factor is decomposed into 
    the sum of odd-dimensional irreducible subrepresentations.
    In particular, 
    the restriction of $\varphi$
    to the diagonal $\mathfrak{sl}(2,\R)$
    is not very even for $\iota$.
    This proves the assertion.
\end{proof}

We are ready to prove Theorem~\ref{thm:proper-imply-spin}.

\begin{proof}[Proof of Theorem~\ref{thm:proper-imply-spin}]
    Let $\iota\colon \mathfrak{g}\rightarrow \mathfrak{gl}(N,\C)$ be 
    a representation of $\mathfrak{g}$ such that
    $G/H$ has $\VE$ for $\iota$, and let $L$ be a connected semisimple Lie subgroup of $G$
    with no compact factors, acting properly on $G/H$.

    We prove that the Lie algebra $\mathfrak{l}$ of $L$ is simple.
    For the sake of contradiction, assume that $\mathfrak{l}$ has two (non-compact) simple factors. Then
    there exists a Lie group homomorphism
    $\varphi\colon SL(2,\R)\times SL(2,\R)\rightarrow L \hookrightarrow G$ with finite kernel by Lemma~\ref{lemma:lift}.
    Let $\varphi_{1}$, $\varphi_{2}$, and $\varphi_{\Delta}$ be 
    the restrictions of $\varphi$ to the first, second, and diagonal $SL(2,\R)$ factors, respectively.
    Since the $L$-action on $G/H$ is proper, 
    the three $SL(2,\R)$-actions on $G/H$ via 
    $\varphi_{1}$, $\varphi_{2}$, and $\varphi_{\Delta}$
    are proper. Hence the differentials of $\varphi_{1}$, $\varphi_{2}$, and $\varphi_{\Delta}$ are all very even for $\iota$ since $G/H$ has $\VE$ for $\iota$.
    This is a contradiction by Lemma~\ref{lem:very-even-odd}.
    Hence $\mathfrak{l}$ is simple.

    By the second assertion of Lemma~\ref{lem:ve->e}, $G/H$ has $\PE$. Hence it follows from  Proposition~\ref{prop:proper->Spin} that 
    $L$ is locally isomorphic to $SO(n,1)$.
    Thus the first assertion is proved.

    To prove the latter part of the theorem, we assume that the above $\iota\colon \mathfrak{g}\rightarrow \mathfrak{gl}(N,\C)$ lifts to 
    a Lie group homomorphism $\tilde{\iota}\colon G\rightarrow GL(N,\C)$.
    From the previous paragraph, we can and do take an isomorphism $\psi\colon \mathfrak{so}(n,1)\rightarrow \mathfrak{l}$.
    By Lemma~\ref{lemma:lift}, 
    $\psi$ lifts to a Lie group homomorphism  
    $\tilde{\psi}\colon Spin(n,1)\rightarrow L$. 
    Since $L$ is connected, $\tilde{\psi}$ is surjective.

    We note that the kernel of $\tilde{\psi}$ is contained in the center of $Spin(n,1)$. 
    For the sake of contradiction, we assume that $\tilde{\psi}$ is not injective.
    Here we think of $SL(2,\R)$ as a closed subgroup of $Spin(n,1)$ via $SL(2,\R) \simeq Spin(2,1)\hookrightarrow Spin(n,1)$.
    Since the center of $Spin(n,1)$ is contained in
    $SL(2,\R)$ by Lemma~\ref{lemma:spin-center}, the restriction of $\tilde{\psi}$ to $SL(2,\R)$ is not injective.
    Thus the restriction of the representation $\tilde{\iota}\circ\tilde{\psi}$ to $SL(2,\R)$ is not faithful, and hence contains an odd-dimensional irreducible subrepresentation. 
    This means that 
    the restriction of 
    $\psi$ to $\mathfrak{sl}(2,\R)$ is not very even for $\iota$.
    Since $G/H$ has $\VE$ for $\iota$,  the $SL(2,\R)$-action on $G/H$ via $\tilde{\psi}$ is not proper,
    which contradicts the properness of the $L$-action on $G/H$.
    Hence $\tilde{\psi}\colon Spin(n,1)\rightarrow L$ is an isomorphism. This completes the proof.
\end{proof}

\section{\texorpdfstring{Hurwitz--Radon}{Hurwitz-Radon} numbers
associated with classical pairs}
\label{section:proof-hurwitz-radon}
The main goal of this section is the proof of Theorem~\ref{theorem:Hurwitz-radon}. Namely, we compute the values of $\rho^{(i)}(\mathfrak{g}, \iota)$ introduced in
Definition~\ref{def:HR-ours} for the Lie algebras $\mathfrak{g}$ and their standard representations $\iota$, in terms of the classical Hurwitz--Radon number $\rho(N)$ and the $2$-adic valuation $\ord(N)$.

In Section~\ref{section:HR-well-defined}, 
we prove that $\rho^{(i)}(\mathfrak{g},\iota)$ depends only on the equivalence class
of the pair $(\mathfrak{g},\iota)$ in the sense of
Definition~\ref{definition:classical-pair}.
In Section~\ref{section:parahermitian-ineq}, 
using this result, we give several inequalities for $\rho^{(i)}$ between two different pairs
$(\mathfrak{g},\iota_{\mathfrak{g}})$ and 
$(\mathfrak{h},\iota_{\mathfrak{h}})$. 
In
Section~\ref{section:proof-of-B},
we use these inequalities to prove
Theorem~\ref{theorem:Hurwitz-radon}.

\subsection{Well-definedness of \texorpdfstring{$\rho^{(i)}(\mathfrak g,\iota)$}{our 
formulation of the Hurwitz-Radon numbers}
}
\label{section:HR-well-defined}

Let $\iota\colon \mathfrak{g}\rightarrow \mathfrak{gl}(N,\C)$ be 
a faithful representation of a reductive Lie algebra $\mathfrak{g}$.   
We assume that $\iota$
satisfies the following condition:  
\begin{condition}
\label{condition:hurwitz-radon}
    There exists a positive definite hermitian form $\pairing{\cdot}{\cdot}$ on the representation space $\C^N$ of $\iota$ 
    such that $\iota(\mathfrak{g})$ is closed under taking adjoint operators with respect to $\pairing{\cdot}{\cdot}$. Then we define 
    an involution $\theta$ of $\mathfrak{g}$ such that
\begin{equation}
\label{eq:cartan-involution}
\pairing{\iota(X)v}{w}=-\pairing{v}{\iota(\theta(X))w}\quad (X\in \mathfrak{g},\ v,w\in \C^{N}). 
\end{equation}
\end{condition}
Condition~\ref{condition:hurwitz-radon} ensures that $\theta$ is a Cartan involution on the semisimple part of $\mathfrak{g}$, and it is automatically satisfied if $\mathfrak{g}$ is semisimple. 

We denote by $\mathfrak{k}$ and $\mathfrak{p}$ the $(+1)$- and $(-1)$-eigenspaces
of $\theta$, respectively.
Recall that $\rho^{(i)}(\mathfrak{g},\iota)$ is the largest
integer $n\in\N$ for which there exists an $\R$-linear map
$f\colon \R^{n}\to \mathfrak{p}$ satisfying the following conditions: 
\begin{description}
    \item [$\rho^{(1)}(\mathfrak{g},\iota)$]
    
    $\iota(f(v))^2=\|v\|^2I_{N}$ for any $v\in \R^n$;
    \item [$\rho^{(2)}(\mathfrak{g},\iota)$]
    $\iota(f(v))$ is invertible for any non-zero $v\in \R^n$.
\end{description}

In this subsection, we prove  that
$\rho^{(i)}(\mathfrak{g},\iota)$ is independent of the choice of a hermitian form $\pairing{\cdot}{\cdot}$ in Condition~\ref{condition:hurwitz-radon}, and moreover, that it is also independent
of the choice of a representative of the equivalence class of the pair
$(\mathfrak{g},\iota)$ in the sense of Definition~\ref{definition:classical-pair}. In what follows, we 
write $\rho^{(i)}(\mathfrak{g}, \iota, \pairing{\cdot}{\cdot})$
for $\rho^{(i)}(\mathfrak{g}, \iota)$ to emphasize the hermitian form. 
\begin{proposition} 
    \label{prop:HR-pair-equiv}
    Let $(\mathfrak{g},\iota,\pairing{\cdot}{\cdot})$ 
    and $(\mathfrak{g}',\iota',\pairing{\cdot}{\cdot}')$
    be triples satisfying the above conditions. 
    If $(\mathfrak{g},\iota)$ is equivalent to $(\mathfrak{g}',\iota')$, then we have 
    \[
    \rho^{(i)}(\mathfrak{g},\iota,\pairing{\cdot}{\cdot})
        =\rho^{(i)}(\mathfrak{g}',\iota', \pairing{\cdot}{\cdot}')\quad(i=1,2).
    \]
\end{proposition}

Proposition~\ref{prop:HR-pair-equiv} is proved in the following two steps:  

\begin{lemma} 
    Let $\iota\colon \mathfrak{g}\rightarrow \mathfrak{gl}(N,\C)$ and $\pairing{\cdot}{\cdot}$ be as above.
    \begin{enumerate}[label=(\roman*)]
        \item 
        \label{item:independent-form}
        Assume that  
        $\iota(\mathfrak{g})$ is self-adjoint with respect to another positive definite hermitian form $\pairing{\cdot}{\cdot}'$ on $\C^{N}$. Then, 
        \[
        \rho^{(i)}(\mathfrak{g},\iota,\pairing{\cdot}{\cdot})=\rho^{(i)}(\mathfrak{g},\iota,\pairing{\cdot}{\cdot}')\quad(i=1,2).
        \]
        \item 
        \label{item:independent-pair}
        Assume that a pair $(\mathfrak{g}',\iota')$
        is equivalent to 
        $(\mathfrak{g},\iota)$. 
        Then
        $\iota'(\mathfrak{g}')$ is self-adjoint with respect to 
        some positive definite hermitian form $\pairing{\cdot}{\cdot}'$ 
        on $\C^{N}$, 
        and we have 
        \[
        \rho^{(i)}(\mathfrak{g},\iota,\pairing{\cdot}{\cdot})
        =\rho^{(i)}(\mathfrak{g}',\iota', \pairing{\cdot}{\cdot}')\quad(i=1,2).
        \]
    \end{enumerate}
\end{lemma}

Although the proof is standard, we include it for completeness.

\begin{proof}
      \ref{item:independent-form}. Let $\theta$ and $\theta'$ be the involutive automorphisms of $\mathfrak{g}$ associated with the hermitian forms $\pairing{\cdot}{\cdot}$ and $\pairing{\cdot}{\cdot}'$, respectively, satisfying the relation \eqref{eq:cartan-involution}. Let $\mathfrak{g} = \mathfrak{k} + \mathfrak{p} = \mathfrak{k}' + \mathfrak{p}'$ be the eigenspace decompositions corresponding to $\theta$ and $\theta'$, respectively.

We prove the following claim:

\medskip
\noindent
\textbf{Claim.}
There exists an inner automorphism $\varphi$ of $\mathfrak{g}$
such that
\[
\varphi \circ \theta = \theta' \circ \varphi.
\]

Assuming \textbf{Claim}, we obtain \ref{item:independent-form}.
Indeed, since $\varphi$ is inner, there exists $g \in GL(N,\C)$
such that
$\Ad(g)\circ \iota = \iota \circ \varphi$.
Hence we obtain the following commutative diagram:
\[
   \xymatrix{
    \mathfrak{g} \ar[r]^{\theta} \ar[d]^{\varphi} 
    & \mathfrak{g} \ar[d]^{\varphi} \ar[r]^-{\iota} & M(N,\C) \ar[d]^{\Ad(g)} \\
    \mathfrak{g} \ar[r]^{\theta'}   & \mathfrak{g} \ar[r]^-{\iota} & M(N,\C)
   }\quad.
\]
Restricting $\varphi$ to $\mathfrak{p}$,
we obtain a linear isomorphism from 
$\mathfrak{p}$ to $\mathfrak{p}'$.
Since $\Ad(g)$ is a $\C$-algebra automorphism of $M(N,\C)$,
the assertion \ref{item:independent-form} follows immediately.

We prove \textbf{Claim}. Since the automorphisms $\theta$ and $\theta'$ preserve both the semisimple part and the center of 
the reductive Lie algebra $\mathfrak{g}$, it suffices to treat the semisimple and abelian cases.

In the semisimple case,  
$\theta$ and $\theta'$ are both Cartan involutions on $\mathfrak{g}$. 
Hence they are conjugate by an inner automorphism (see, e.g., \cite[Chap.~III, Thm.~7.2]{HelgasonDiffLieSymm2001}), which proves \textbf{Claim}.

In the abelian case, we prove the stronger statement that
$\theta = \theta'$.
Since $\iota(\mathfrak{g})$ is simultaneously diagonalizable,
we define $\R$-subspaces
$\mathfrak{g}_{r}$ and $\mathfrak{g}_{p}$ as follows:
\begin{align*}
\mathfrak{g}_r &:= \{X \in \mathfrak{g} \mid \text{all eigenvalues of } \iota(X) \text{ are real}\}, \\
\mathfrak{g}_p &:= \{X \in \mathfrak{g} \mid \text{all eigenvalues of } \iota(X) \text{ are purely imaginary}\}.
\end{align*}
By definition, all elements of $\iota(\mathfrak{k})$ and $\iota(\mathfrak{p})$ are skew-hermitian and hermitian with respect to $\pairing{\cdot}{\cdot}$, respectively. Thus,
$\mathfrak{k}\subset \mathfrak{g}_p$ and $\mathfrak{p}\subset \mathfrak{g}_r$. 
Since $\mathfrak{g} = \mathfrak{k} + \mathfrak{p}$ and 
$\mathfrak{g}_p\cap \mathfrak{g}_r=\{0\}$, 
it follows that $\mathfrak{k} = \mathfrak{g}_p$ and $\mathfrak{p} = \mathfrak{g}_r$. Similarly, $\mathfrak{k}' = \mathfrak{g}_p$ and $\mathfrak{p}' = \mathfrak{g}_r$.
Hence we conclude that $\theta = \theta'$.

This completes the proof of \textbf{Claim}, and thus \ref{item:independent-form} is proved.
        
    \ref{item:independent-pair}.  
    The argument is straightforward.
    Since $(\mathfrak{g}', \iota')$ and 
    $(\mathfrak{g}, \iota)$ are equivalent, 
    we can take 
        $g\in GL(N,\C)$ and 
        an automorphism $\varphi\colon \mathfrak{g}\rightarrow \mathfrak{g}'$ such that 
        $\Ad(g^{-1})\circ\iota= \iota'\circ\varphi$.
We define a positive definite hermitian form $\pairing{\cdot}{\cdot}'$ on $\C^N$ and an involutive automorphism $\theta'$ of $\mathfrak{g}'$ as follows: 
    \[
    \pairing{v}{w}':=\pairing{gv}{gw},\quad\theta':=\varphi\circ\theta\circ\varphi^{-1}. 
        \]
Then, it follows by direct computation that
\[
\pairing{\iota'(X')v}{w}' 
= -\pairing{v}{\iota'(\theta'(X'))w}'
\quad (X'\in \mathfrak{g}',\ v,w \in \mathbb{C}^{N}).
\]
In particular, $\iota'(\mathfrak{g}')$ is self-adjoint with respect to $\pairing{\cdot}{\cdot}'$. 
The assertion~\ref{item:independent-pair} is now clear. 
\end{proof}

\subsection{Inequalities along periodic chains
of classical Lie algebras 
}
\label{section:parahermitian-ineq}

The main goal of this subsection is to show 
$8+2$ key inequalities for the proof of Theorem~\ref{theorem:Hurwitz-radon}. Namely, we consider the $8$-periodic chain
\begin{align}
\label{sequence:8-period}
    \cdots \subset &\mathfrak{so}(N,N)\subset \mathfrak{gl}(2N,\mathbb{R})
    \subset \mathfrak{sp}(2N,\mathbb{R}) \subset \mathfrak{sp}(2N,\mathbb{C}) \nonumber \\
    \subset & \mathfrak{sp}(2N,2N) \subset \mathfrak{gl}(4N,\mathbb{H})
    \subset \mathfrak{so}^{*}(2(8N))
    \subset \mathfrak{so}(16N,\mathbb{C}) \\
    \subset & \mathfrak{so}(16N,16N) \subset \cdots, \nonumber
\end{align}
and the $2$-periodic chain
\begin{equation}
\label{sequence:2-period}
    \cdots \subset \mathfrak{gl}(N,\C) \subset \mathfrak{su}(N,N)
    \subset \mathfrak{gl}(2N,\C) \subset \mathfrak{su}(2N,2N) \subset \cdots.
\end{equation}
We show that, for any pair of adjacent classical Lie algebras
$\mathfrak{h}\subset\mathfrak{g}$ appearing in these chains,
the values of $\rho^{(i)}$ differ by at least one
(Corollary~\ref{cor:paraherm}). The observation behind this result is that the associated pair of the symmetric pair $(\mathfrak{g},\mathfrak{h})$ is parahermitian (Proposition~\ref{prop:ineq-paraherm}).

We retain the notation in Section~\ref{section:HR-well-defined}. 
Let $\mathfrak{h}$ be a $\theta$-stable subalgebra of $\mathfrak{g}$.
Then its restricted representation $\iota|_{\mathfrak{h}}$
satisfies Condition~\ref{condition:hurwitz-radon}, and thus 
$\rho^{(i)}(\mathfrak{h},\iota|_{\mathfrak{h}})$ is defined.
Moreover, in defining $\rho^{(i)}(\mathfrak{h},\iota|_{\mathfrak{h}})$,
we may use the decomposition 
\[
\mathfrak{h}
=
(\mathfrak{k}\cap\mathfrak{h}) +(\mathfrak{p}\cap\mathfrak{h}).
\]
By this observation, we give inequalities comparing the
integers 
$\rho^{(i)}(\mathfrak{g},\iota)$ and
$\rho^{(i)}(\mathfrak{h},\iota|_{\mathfrak{h}})$
(Lemma~\ref{lem:inclusion_to_inequality} and Proposition~\ref{prop:ineq-paraherm}). 
    
First, we note the following easy inequality (we omit its proof):
\begin{lemma}
    \label{lem:inclusion_to_inequality}
    Let $\mathfrak{h}$ be a
    $\theta$-stable subalgebra of $\mathfrak{g}$.
    Then we have 
    \[
    \rho^{(i)}(\mathfrak{h},\iota|_{\mathfrak{h}})\leq \rho^{(i)}(\mathfrak{g},\iota)\quad (i=1,2).
    \]
\end{lemma}

Next, we give a sufficient condition under which a sharper inequality holds 
in the case where $(\mathfrak{g},\mathfrak{h})$ is a semisimple symmetric pair. 
Let $\sigma$ be the corresponding involution of $\mathfrak{g}$, 
and assume that $\theta$ is a Cartan involution of $\mathfrak{g}$
commuting with $\sigma$. 
We put 
\begin{align*}
    \mathfrak{h}^{a}&:=
\{X\in \mathfrak{g}\mid \sigma \theta (X)=X\}=(\mathfrak{k}\cap \mathfrak{h})+ (\mathfrak{p}\cap \mathfrak{q}), \\
\mathfrak{q}^{a}&:=
\{X\in \mathfrak{g}\mid \sigma \theta (X)=-X\}=(\mathfrak{k}\cap \mathfrak{q})+ (\mathfrak{p}\cap \mathfrak{h}),
\end{align*}
where $\mathfrak{q}=\{X\in \mathfrak{g}\mid \sigma(X)=-X\}$. 
The pair $(\mathfrak{g},\mathfrak{h}^{a})$
is called the \emph{associated symmetric pair} of $(\mathfrak{g},\mathfrak{h})$. For each irreducible symmetric pair $(\mathfrak{g},\mathfrak{h})$,
the associated pair $(\mathfrak{g},\mathfrak{h}^{a})$
is listed in \cite[Sect.~1]{OshimaSekiguchi84}.

\begin{proposition}
    \label{prop:ineq-paraherm}
    In the setting above, 
    assume that 
    there exists
    $Z\in \mathfrak{p}\cap \mathfrak{h}^{a}(=\mathfrak{p}\cap \mathfrak{q})$ such that:
    \begin{enumerate}[label=$(\roman*)$]
    \item 
    \label{item:associated-parahermitian}
    $\mathfrak{h}^{a}$ is the centralizer of $Z$ in $\mathfrak{g}$;

    \item 
    \label{item:involution}
    $\iota(Z)^2=I_{N}$ holds and conjugation by $\iota(Z)$ preserves $\iota(\mathfrak{g})$.
    \end{enumerate}
    Then
    \[
        \rho^{(i)}(\mathfrak{h},\iota|_{\mathfrak{h}})+1
        \leq \rho^{(i)}(\mathfrak{g},\iota)
        \quad (i=1,2).
    \]   
\end{proposition}

\begin{remark}
The condition~\ref{item:associated-parahermitian} in 
Proposition~\ref{prop:ineq-paraherm} means that the associated pair
$(\mathfrak{g},\mathfrak{h}^a)$ is parahermitian in the sense of 
Kaneyuki--Kozai~\cite{Kaneyuki_Kozai_1985}. 
Irreducible parahermitian symmetric pairs are listed in 
\cite{Kaneyuki_Kozai_1985}.
\end{remark}

\begin{proof}[Proof of Proposition~\ref{prop:ineq-paraherm}]
    Conjugation by $\iota(Z)$ induces an involution of $\mathfrak{g}$
    by the condition~\ref{item:involution} and 
    this involution coincides with the involution corresponding to the associated symmetric pair $(\mathfrak{g},\mathfrak{h}^a)$ by  the condition~\ref{item:associated-parahermitian}.
    In particular, the matrix $\iota(Z)$
    and any matrix in $\iota(\mathfrak{p}\cap\mathfrak{h})(\subset \iota(\mathfrak{q}^a))$
    anticommute.

    Now we prove our inequalities.
    Given an $\R$-linear map $f\colon \R^{n}\rightarrow \mathfrak{p}\cap \mathfrak{h}$,
    we define an $\R$-linear map $\check{f}\colon \R^{n}\oplus \R\rightarrow \mathfrak{p}$
    by $\check{f}(v,t) := f(v) + tZ$. 
    Since 
    $\iota(f(v))$ and 
    $\iota(Z)$ anticommute, we have
    \begin{equation}
    \label{eq:(A+B)^2=A^2+B^2}
    \iota(\check{f}(v,t))^2= \iota(f(v))^2+t^2I_{N},
    \end{equation}
    by 
    the condition~\ref{item:involution}. 

    First, assume that $\iota(f(v))^2 = \|v\|^2I_{N}$ holds. Then, by \eqref{eq:(A+B)^2=A^2+B^2},
    we have 
    \[
    \iota(\check{f}(v,t))^2 
    =(\|v\|^2+t^{2})I_{N}.
    \]
    Hence we obtain $\rho^{(1)}(\mathfrak{h},\iota|_{\mathfrak{h}})+1 
    \leq \rho^{(1)}(\mathfrak{g},\iota)$.
    
    Next, assume that $\iota(f(v))$ is invertible. 
    Then $\iota(f(v))^2$ is positive definite 
    with respect to $\pairing{\cdot}{\cdot}$. 
    Hence, 
    the equality \eqref{eq:(A+B)^2=A^2+B^2} implies that  $\iota(\check{f}(v,t))^2$ is positive definite for 
    any $t$. In particular, $\iota(\check{f}(v,t))$ is invertible. 
    Hence we obtain $\rho^{(2)}(\mathfrak{h},\iota|_{\mathfrak{h}})+1 
    \leq \rho^{(2)}(\mathfrak{g},\iota)$.
    Thus the proof is complete. 
\end{proof}

Applying these two inequalities, we give several inequalities (Corollaries~\ref{corollary:auxiliary-inequality} and~\ref{cor:paraherm}) that will be used in the next subsection.
As a preparation, we note the following straightforward lemma.
\begin{lemma}
\label{lemma:HR-min}
Let $\iota_{1},\iota_{2}$
be two faithful representations of  $\mathfrak{g}$ satisfying 
Condition~\ref{condition:hurwitz-radon} for the same involution $\theta$. 
Then we have 
    \[
    \rho^{(i)}(\mathfrak{g},\iota_{1}\oplus\iota_{2})=\min(\rho^{(i)}(\mathfrak{g},\iota_{1}),\rho^{(i)}(\mathfrak{g},\iota_{2}))\quad (i=1,2).
    \]
\end{lemma}

As a corollary, we obtain the following: 
\begin{corollary}
\label{corollary:compare-restriction}
Let $\mathfrak{h}$ be a $\theta$-stable subalgebra of $\mathfrak{g}$, 
and 
$\iota_{\mathfrak{h}}$ a faithful representation of $\mathfrak{h}$.
Assume that the restriction $\iota|_{\mathfrak{h}}$
decomposes as a direct sum of representations each of which is equivalent to $\iota_{\mathfrak{h}}\circ\tau$ for some automorphism $\tau$ of $\mathfrak{h}$.
Then
\[
\rho^{(i)}(\mathfrak{h},\iota_{\mathfrak{h}})
=
\rho^{(i)}(\mathfrak{h},\iota|_{\mathfrak{h}})
\quad (i=1,2).
\]
\end{corollary}

\begin{proof}    
    Any subrepresentation $\iota'$ of $\iota|_{\mathfrak{h}}$
    satisfies Condition~\ref{condition:hurwitz-radon}. Moreover, 
    if $\iota'$ is equivalent to
    the representation $\iota_{\mathfrak{h}}\circ \tau$ for some automorphism $\tau$ of $\mathfrak{h}$, then
    the pairs 
    $(\mathfrak{h},\iota')$ and 
    $(\mathfrak{h},\iota_{\mathfrak{h}}\circ \tau)$ are 
    equivalent in the sense of
Definition~\ref{definition:classical-pair}, and thus we have $\rho^{(i)}(\mathfrak{h},\iota')=\rho^{(i)}(\mathfrak{h},\iota_{\mathfrak{h}}\circ\tau)$ by Proposition~\ref{prop:HR-pair-equiv}. 
Since, by assumption, $\iota|_{\mathfrak{h}}$ decomposes as a direct sum of such representations,
the assertion follows from Lemma~\ref{lemma:HR-min}.
\end{proof}

Using this, we obtain the following as a corollary of
Lemma~\ref{lem:inclusion_to_inequality}:
\begin{corollary}
\label{corollary:auxiliary-inequality}
Let $(\mathfrak{g},\mathfrak{h})$ be the following pairs of classical Lie algebras with 
$\iota_{\mathfrak{g}},\iota_{\mathfrak{h}}$ 
the standard representations: 
\begin{enumerate}[label=(\roman*)]
    \item 
    \label{item:gl-sl}
    $(\mathfrak{gl}(N,\D),\mathfrak{sl}(N,\D))$ $(\D=\R,\C,\HA)$;
    \item 
    \label{item:glC-glR}
    $(\mathfrak{gl}(2N,\R), \mathfrak{gl}(N,\C))$;
    \item 
    \label{item:su-sp}
    $(\mathfrak{sp}(2N,\R), \mathfrak{su}(N,N))$. 
\end{enumerate}
Then we have 
\[
\rho^{(i)}(\mathfrak{h},\iota_{\mathfrak{h}})\leq \rho^{(i)}(\mathfrak{g},\iota_{\mathfrak{g}})\quad (i=1,2). 
\]
\end{corollary}

\begin{proof}
    The branching rules for the restriction $\iota_{\mathfrak{g}}|_{\mathfrak{h}}$
are as follows.
In Case~\ref{item:gl-sl}, it is $\iota_{\mathfrak{h}}$.
In Cases~\ref{item:glC-glR} and~\ref{item:su-sp},
it is $\iota_{\mathfrak{h}}\oplus (\iota_{\mathfrak{h}}\circ\tau)$,
where $\tau\colon \mathfrak{h}\to \mathfrak{h}$ denotes the involution given by
complex conjugation.
Therefore, by Corollary~\ref{corollary:compare-restriction}, we have
\[
\rho^{(i)}(\mathfrak{h},\iota_{\mathfrak{g}}|_{\mathfrak{h}})
=
\rho^{(i)}(\mathfrak{h},\iota_{\mathfrak{h}}).
\]
Furthermore, Lemma~\ref{lem:inclusion_to_inequality} yields the desired assertion.
\end{proof}

As a corollary of Proposition~\ref{prop:ineq-paraherm}, we obtain the following: 
\begin{corollary}
\label{cor:paraherm}
Let $(\mathfrak{g},\mathfrak{h})$ be one of the classical symmetric pairs
in Table~\ref{table:parahermitian}, 
and $\iota_{\mathfrak{g}},\iota_{\mathfrak{h}}$ the standard representations of 
$\mathfrak{g},\mathfrak{h}$,
respectively.
Then  we have
$\rho^{(i)}(\mathfrak{h},\iota_{\mathfrak{h}})+1 
\leq \rho^{(i)}(\mathfrak{g},\iota_{\mathfrak{g}})$ for 
each $i=1,2$.

\begin{table}[htbp]
\begin{tabular}{c|c|c|c|c}
    $\mathfrak{g}$ & $\mathfrak{h}$ & $\mathfrak{h}^{a}$ & $Z\in \mathfrak{h}^{a}$
    & $\iota_{\mathfrak{g}}|_{\mathfrak{h}}$\\
    \noalign{\hrule height 1.2pt}
    $\mathfrak{sl}(2N,\R)$ & $\mathfrak{so}(N,N)$ & $\mathfrak{sl}(N,\R)^{\oplus 2}\oplus \R$ & $(0,0,1)$
    & $\iota_{\mathfrak{h}}$\\
    $\mathfrak{sp}(N,\R)$ & $\mathfrak{gl}(N,\R)$ & $\mathfrak{h}$ & $I_{N}$ & $\iota_{\mathfrak{h}}\oplus (\iota_{\mathfrak{h}}\circ\theta)$ \\
    $\mathfrak{sp}(N,\C)$ & $\mathfrak{sp}(N,\R)$ & $\mathfrak{gl}(N,\C)$ & $I_{N}$ & $\iota_{\mathfrak{h}}$\\
    $\mathfrak{sp}(N,N)$ & $\mathfrak{sp}(N,\C)$ & $\mathfrak{gl}(N,\HA)$ & $I_{N}$ & $\iota_{\mathfrak{h}}\oplus (\iota_{\mathfrak{h}}\circ\theta)$ \\
    $\mathfrak{sl}(2N,\HA)$ & $\mathfrak{sp}(N,N)$ & $\mathfrak{sl}(N,\HA)^{\oplus 2}\oplus \R$ & $(0,0,1)$ & $\iota_{\mathfrak{h}}$\\
    $\mathfrak{so}^{*}(4N)$ & $\mathfrak{gl}(N,\HA)$ & $\mathfrak{h}$ & $I_{N}$ & $\iota_{\mathfrak{h}}\oplus (\iota_{\mathfrak{h}}\circ\theta)$ \\
    $\mathfrak{so}(2N,\C)$ & $\mathfrak{so}^{*}(2N)$ & 
    $\mathfrak{gl}(N,\C)$ & $I_{N}$ & $\iota_{\mathfrak{h}}$ \\
    $\mathfrak{so}(N,N)$ & $\mathfrak{so}(N,\C)$ & $\mathfrak{gl}(N,\R)$ & $I_N$ & $\iota_{\mathfrak{h}}\oplus (\iota_{\mathfrak{h}}\circ\theta)$\\
    $\mathfrak{sl}(2N,\C)$ & $\mathfrak{u}(N,N)$ & $\mathfrak{sl}(N,\C)^{\oplus 2}\oplus \R$ & $(0,0,1)$ & $\iota_{\mathfrak{h}}$\\
    $\mathfrak{su}(N,N)$ & $\mathfrak{gl}(N,\C)$ & $\mathfrak{h}$ & $I_N$
    & $\iota_{\mathfrak{h}}\oplus (\iota_{\mathfrak{h}}\circ\theta)$
\end{tabular}
\caption{Some classical symmetric pairs $(\mathfrak{g},\mathfrak{h})$ whose associated pairs $(\mathfrak{g},\mathfrak{h}^a)$ are parahermitian.}
\label{table:parahermitian}
\end{table}
\end{corollary}

\begin{proof}
Choose a Cartan involution $\theta$ of $\mathfrak{g}$ that commutes with the involution associated with
the symmetric pair $(\mathfrak{g},\mathfrak{h})$.
Then the last column of Table~\ref{table:parahermitian} lists the branching rule of the restriction
$\iota_{\mathfrak{g}}|_{\mathfrak{h}}$.
By this branching rule and Corollary~\ref{corollary:compare-restriction}, we have
\[
\rho^{(i)}(\mathfrak{h},\iota_{\mathfrak{h}})
=
\rho^{(i)}(\mathfrak{h},\iota_{\mathfrak{g}}|_{\mathfrak{h}})
\qquad (i=1,2).
\]
Hence it suffices to verify the conditions~\ref{item:associated-parahermitian}
and~\ref{item:involution} in Proposition~\ref{prop:ineq-paraherm}.

The associated pair $(\mathfrak{g},\mathfrak{h}^{a})$ is isomorphic to 
the corresponding pair listed in
Table~\ref{table:parahermitian}, which can also be confirmed from the tables in
\cite[Sect.~1]{OshimaSekiguchi84}.
The center of $\mathfrak{h}^{a}$ is one-dimensional, and its centralizer in $\mathfrak{g}$
coincides with $\mathfrak{h}^{a}$.
Although this fact can be checked directly in each case, it also follows from the fact that
$(\mathfrak{g},\mathfrak{h}^{a})$ in
Table~\ref{table:parahermitian} appears in the classification table of irreducible parahermitian symmetric pairs
due to Kaneyuki--Kozai (see \cite[p.~97]{Kaneyuki_Kozai_1985}).

Let $Z$ be a generator of the center of $\mathfrak{h}^{a}$ as given in
Table~\ref{table:parahermitian}.
Then the condition~\ref{item:associated-parahermitian} follows from the fact in the previous paragraph.
The condition~\ref{item:involution} can also be verified case by case.
For instance, in the first entry of Table~\ref{table:parahermitian},
we have $\iota_{\mathfrak{g}}(Z)=I_{N,N}$, from which
the conditions~\ref{item:associated-parahermitian} and~\ref{item:involution}
follow immediately.
The remaining cases are treated in the same way.

Therefore, Proposition~\ref{prop:ineq-paraherm} implies the assertion.
\end{proof}

\begin{remark}
    \label{remark:HR-equality}
    In the next subsection, we will show that the inequalities in
Corollary~\ref{cor:paraherm}
are in fact equalities.
\end{remark}

\subsection{Proof of 
\texorpdfstring{Theorem~\ref{theorem:Hurwitz-radon}}{Theorem B}
}
\label{section:proof-of-B}
Throughout this subsection, when we consider the integer
$\rho^{(i)}(\mathfrak{g},\iota)$, we always assume that $\iota$ is the standard
representation of the classical Lie algebra $\mathfrak{g}$, and we therefore
omit $\iota$ and simply write $\rho^{(i)}(\mathfrak{g})$.

We prove Theorem~\ref{theorem:Hurwitz-radon} by dividing the classical 
Lie algebras into the following five cases:
\begin{enumerate}[label=$( \alph* )$]
    \item \label{item:real8}
    $\mathfrak{so}(N,N)$, $\mathfrak{gl}(N,\mathbb{R})$, 
    $\mathfrak{sp}(N,\mathbb{R})$, $\mathfrak{sp}(N,\mathbb{C})$, \\
    $\mathfrak{sp}(N,N)$, $\mathfrak{gl}(N,\mathbb{H})$, 
    $\mathfrak{so}^{*}(2N)$, $\mathfrak{so}(N,\mathbb{C})$.
    \item \label{item:complex2}
    $\mathfrak{su}(N,N)$, $\mathfrak{gl}(N,\mathbb{C})$.
    \item \label{item:sl(n,F)}
    $\mathfrak{sl}(2N,\mathbb{D})$ \quad $(\mathbb{D}=\mathbb{R},\mathbb{C},\mathbb{H},\, N\ge 1)$.
    \item[$(c')$] 
    \label{item:sl(odd)}
    $\mathfrak{sl}(2N+1,\mathbb{D})$ \quad $(\mathbb{D}=\mathbb{R},\mathbb{C},\mathbb{H},\, N\ge 1)$.
    \item \label{item:sl(1,D)}
    $\mathfrak{sl}(1,\mathbb{D})$ \ $(\mathbb{D}=\mathbb{R},\mathbb{C},\mathbb{H})$, \quad 
    $\mathfrak{su}(p,q;\D)$ \ $(p\neq q)$.
\end{enumerate}

The cases \ref{item:real8}, \ref{item:complex2}, \ref{item:sl(n,F)}, and 
\ref{item:sl(1,D)} correspond to the entries in Table~\ref{tab:HR-values} of 
Theorem~\ref{theorem:Hurwitz-radon}, while  
\hyperref[item:sl(odd)]{$(c')$} corresponds to the exceptional cases 
in Theorem~\ref{theorem:Hurwitz-radon} where 
$\rho^{(1)}\neq \rho^{(2)}$.

We first consider Case~\ref{item:sl(1,D)}.
For $\mathfrak{g}=\mathfrak{sl}(1,\mathbb{D})$, 
its $\mathfrak{p}$-part is zero, and for $\mathfrak{g}=\mathfrak{su}(p,q;\D)$ with $p\neq q$, 
its $\mathfrak{p}$-part
has no invertible matrices in $M(p+q,\D)$. 
Hence we have 
\[
\rho^{(1)}(\mathfrak{g})=\rho^{(2)}(\mathfrak{g})=0
\]
for the classical Lie algebras $\mathfrak{g}$ in Case \ref{item:sl(1,D)}.

In what follows, we consider Cases
\ref{item:real8}, \ref{item:complex2}, \ref{item:sl(n,F)}, and \hyperref[item:sl(odd)]{$(c')$}
in this order.
We first record a simple lemma that will be used repeatedly:
\begin{lemma}
\label{lemma:gl>0}
    We have $\rho^{(i)}(\mathfrak{gl}(N,\D))\geq 1$
    for $\D=\R,\C,\HA$ and $i=1,2$. 
\end{lemma}
\begin{proof}
    Let $\mathfrak{g}=\mathfrak{gl}(N,\D)$. In defining
$\rho^{(i)}(\mathfrak{g})$, we may and do assume that $\mathfrak{p}$ consists of all $N\times N$ hermitian matrices over $\D$.
If we define $f\colon \R\to\mathfrak{p}$ by $f(t)=tI_{N}$,
it is immediate that $f$ satisfies the conditions in the definition of
$\rho^{(i)}(\mathfrak{g})$.
Therefore, $\rho^{(i)}(\mathfrak{g})\geq 1$ for each $i=1,2$.
\end{proof}

\vspace{\baselineskip}

\noindent \textbf{Proof of Case~\ref{item:real8}.}
In the proof of Case~\ref{item:real8}, we adopt the following strategy.
Using the inequalities for the $8$-periodic chain \eqref{sequence:8-period}
established in the previous subsection (namely, the first eight inequalities in
Corollary~\ref{cor:paraherm}),
we reduce the argument to 
Lemma~\ref{lemma:HR-our-interpretation}, namely,
our interpretation of the classical result of
Hurwitz--Radon--Eckmann--Adams: 
\begin{equation*}
        \rho(N)=\rho^{(1)}(\mathfrak{so}(N,N),\iota)=\rho^{(2)}(\mathfrak{so}(N,N),\iota).
    \end{equation*}

Consider the following sequence $\{a_i\}_{i=1}^{\infty}$ of period $8$:
\[
\begin{array}{c|ccccccccc}
i & 1 & 2 & 3 & 4 & 5 & 6 & 7 & 8 & \cdots \\
\hline
a_{i} & 2 & 1 & 1 & 1 & 2 & 2 & 2 & 1 & \cdots
\end{array}
\]
We also consider the sequence $\{b_i\}_{i=1}^{\infty}$ defined by the recurrence relation $b_{1}=1$ and $b_{i+1}=a_{i}b_{i}$:
\[
\begin{array}{c|cccccccccc}
i & 1 & 2 & 3 & 4 & 5 & 6 & 7 & 8 & 9 & \cdots \\
\hline
b_{i} & 1 & 2 & 2 & 2 & 2 & 4 & 8 & 16 & 16 & \cdots
\end{array}
\]
Note that the sequence $\{b_i\}_{i=1}^{\infty}$ gives the coefficients of $N$ appearing 
in the $8$-periodic chain \eqref{sequence:8-period} of classical Lie algebras. 

We use the following lemma to give a uniform proof of Case~\ref{item:real8}: 
\begin{lemma}
\label{lem:f(i,N)}
Suppose that a function $f\colon \N_{+}\times \N_{+}\rightarrow \N$
satisfies the following four conditions:
For any $i,N\in \N_{+}$, 
\begin{enumerate}[label=$(\roman*)$]
    \item\label{item:lem:f(i,N)-ineq-i} $f(i,N)+1\leq f(i+1,a_{i}N)$; 
        \item\label{item:lem:f(i,N)-period8} $f(i+8,N)=f(i,N)$;
    \item\label{item:lem:f(i,N)-ineq-2-6} $f(2,N)\geq 1$ and  $f(6,N)\geq 1$;
    \item\label{item:lem:f(i,N)-i=1} $f(1,N)=\rho(N)$.
\end{enumerate}    
Then we have
\begin{equation}
\label{lem:f(i,N)-identity}
f(i,N)=\rho(N/b_{i})+i-1
\quad (i,N\in \N_{+}).
\end{equation}  
\end{lemma}

We collect the following claim in advance, as it will be used repeatedly in the proof of the lemma:
\begin{claim}
In the setting of Lemma~\ref{lem:f(i,N)},
we put 
\[
b_{i,j} :=b_{i}b_{j}^{-1}=\prod_{k=j}^{i-1}a_{k} 
\quad (1\leq j\leq i). 
\]
\begin{enumerate}[label=$(\roman*)$]
    \item \label{item:claim-1}  
    For each $j\in \N_{+}$, 
    the function on $\{j,j+1,\ldots\}\times \N_{+}$
    defined by $(i,N)\mapsto f(i, b_{i,j}N)-i+1$ 
    is monotone increasing with respect to $i$. 
    \item \label{item:claim-2} $f(i,N) \leq f(i,16N)-8$ for any $i,N\in \N_{+}$.
\end{enumerate}
\end{claim}

\begin{proof}[Proof of Claim]
\ref{item:claim-1}. 
By the condition \ref{item:lem:f(i,N)-ineq-i} in Lemma~\ref{lem:f(i,N)}, 
we have
\[
f(i+1,b_{i+1,j}N)-i=f(i+1,a_{i}b_{i,j}N)-i
\geq f(i,b_{i,j}N)-i+1,
\]
which shows \ref{item:claim-1}.

\ref{item:claim-2}. 
By the conditions \ref{item:lem:f(i,N)-ineq-i} and \ref{item:lem:f(i,N)-period8} in Lemma~\ref{lem:f(i,N)}, we have 
\[
f(i,16N)=f(i,\bigl(\prod_{k=1}^{8}a_{k}\bigr) N)\geq f(i-8,N)+8=f(i,N)+8,
\]
which shows \ref{item:claim-2}. 
\end{proof}

We now prove Lemma~\ref{lem:f(i,N)}.
Our strategy is to first show the assertion in the case where $N$ is sufficiently divisible by $2$,
and then reduce the general case to this one.
In the following proof, we will use without further mention the equality
\[
\rho(16t)=\rho(t)+8\quad (t\in \mathbb{Q}^\times),
\]
which follows immediately from the definition of $\rho(t)$.
\begin{proof}[Proof of Lemma~\ref{lem:f(i,N)}]
Note that both sides of \eqref{lem:f(i,N)-identity} are $8$-periodic in $i$.
Therefore, throughout the proof we may and do assume that
\[
1\le i\le 8.
\]

Let us prove \eqref{lem:f(i,N)-identity} in the case where $N$ is divisible by $b_i$.
Equivalently, we show that
\begin{equation}
\label{eq:f(i,b_i N)}
f(i,b_i M)=\rho(M)+i-1 \quad (M\in\N_{+}).
\end{equation}

Putting $g(i,M):=f(i,b_iM)-i+1$, it suffices to show that $g(i,M)=\rho(M)$.
We prove this by showing that $g(k,M)$ is monotone increasing with respect to $k$,
and that $g(1,M)=g(9,M)=\rho(M)$.

The former follows by applying \textbf{Claim}~\ref{item:claim-1} in the case $j=1$.
For the latter, the equality $g(1,M)=\rho(M)$ follows from the condition~\ref{item:lem:f(i,N)-i=1}.
Moreover, by the  conditions~\ref{item:lem:f(i,N)-i=1} and~\ref{item:lem:f(i,N)-period8}, we have
\[
g(9,M)
= f(9,b_{9}M)-8
= f(1,16M)-8
= \rho(16M)-8
= \rho(M),
\]
which proves the latter assertion. Therefore, \eqref{eq:f(i,b_i N)} follows.

In what follows, we consider the remaining case where $N\notin b_i\N$.
In this case, we prove \eqref{lem:f(i,N)-identity} by showing both inequalities.

We show $(\text{LHS})\leq (\text{RHS})$ in \eqref{lem:f(i,N)-identity} for $N\notin b_i\N$. Since $1\leq i\leq 8$,
we have $16N\in b_i\N$. Combining \textbf{Claim}~\ref{item:claim-2} with \eqref{eq:f(i,b_i N)}, we obtain
\[
f(i,N) \le f(i,16N)-8
= \rho(16N/b_i)+i-9
= \rho(N/b_i)+i-1.
\]
This proves $(\text{LHS})\leq (\text{RHS})$. 

To prove $(\text{LHS})\ge (\text{RHS})$ in
\eqref{lem:f(i,N)-identity},
we first treat the case where $N$ is odd, and then reduce the general case
$N\notin b_i\N$ to the odd case.

When $N$ is odd, the values of $(\text{RHS})$ for $i=1,\ldots,8$
are as follows:
\[
\begin{array}{c|c|c|c|c|c|c|c|c}
    i & 1 & 2 & 3 & 4 & 5 & 6 & 7 & 8 \\
    \hline
    (\text{RHS}) & 1 & 1 & 2 & 3 & 4 & 1 & 0 & 0
\end{array}
\]
The inequality $(\text{LHS})\geq (\text{RHS})$ follows from
the condition~\ref{item:lem:f(i,N)-i=1} for $i=1$,
from the condition~\ref{item:lem:f(i,N)-ineq-2-6} for $i=2,6$,
and is trivial for $i=7,8$.
For $i=3,4,5$, we have $b_{i,2}=1$. 
Hence, applying \textbf{Claim}~\ref{item:claim-1} to the case $j=2$, we have
\[
f(i,N)-i+1\geq f(2,N)-1 \geq 0.
\]
Therefore, we have $(\text{LHS})=f(i,N)\ge i-1=(\text{RHS})$ for $i=3,4,5$.

We show $(\text{LHS})\geq (\text{RHS})$
for a general $N\notin b_i\N$. 
Let $j$ be the largest integer $\leq i$ such that $N$ is divisible by 
$b_{i,j}$. Then $M:=N/b_{i,j}$ is odd.
By \textbf{Claim}~\ref{item:claim-1}, we have
\[
f(i,N)-i+1\geq f(j,M)-j+1\geq \rho(M/b_j)=\rho(N/b_i).
\]
This proves $(\text{LHS})\geq (\text{RHS})$.
Thus the proof is complete.
\end{proof}

Let us prove Theorem~\ref{theorem:Hurwitz-radon} in Case~\ref{item:real8}. 
We define a classical Lie algebra $\mathfrak{g}_{i}(N)$ $(i,N\in \N_{+})$ as follows: 
\begin{table}[htbp]
\[
    \renewcommand{\arraystretch}{1.5}
    \begin{array}{c!{\vrule width 1.2pt}c|c|c|c}
    i\bmod 8 & 1 & 2 & 3 & 4 \\
    \mathfrak{g}_{i}(N) & \mathfrak{so}(N,N) & \mathfrak{gl}(N,\R) & \mathfrak{sp}(N,\R) & \mathfrak{sp}(N,\C) \\
    \hline
        i\bmod 8 & 5 & 6 & 7 & 8 \\
    \mathfrak{g}_{i}(N) & \mathfrak{sp}(N,N) & \mathfrak{gl}(N,\HA) & \mathfrak{so}^{*}(2N) & \mathfrak{so}(N,\C)
\end{array}
    \renewcommand{\arraystretch}{1}
\]
    \caption{Definition of $\mathfrak{g}_{i}(N)$.}
    \end{table}
    
In Case~\ref{item:real8}, Theorem~\ref{theorem:Hurwitz-radon}
is nothing but the conclusion obtained by applying
Lemma~\ref{lem:f(i,N)} to the following two functions:
\[
(i,N) \mapsto \rho^{(1)}(\mathfrak{g}_{i}(N)) \text{ and } \rho^{(2)}(\mathfrak{g}_{i}(N)).
\]
Hence, it suffices to verify the conditions~\ref{item:lem:f(i,N)-ineq-i}--\ref{item:lem:f(i,N)-i=1} in Lemma~\ref{lem:f(i,N)}. 

The condition~\ref{item:lem:f(i,N)-ineq-i} follows from  Corollary~\ref{cor:paraherm} (see the first eight entries in Table~\ref{table:parahermitian})
together with Corollary~\ref{corollary:auxiliary-inequality}~\ref{item:gl-sl}. 
The condition~\ref{item:lem:f(i,N)-period8} is obvious by definition.
The condition~\ref{item:lem:f(i,N)-ineq-2-6} is verified by Lemma~\ref{lemma:gl>0}. 
The condition~\ref{item:lem:f(i,N)-i=1} is nothing but Lemma~\ref{lemma:HR-our-interpretation}, namely,
our interpretation of the classical result of Hurwitz, Radon, Eckmann, and Adams:  
\[
\rho^{(1)}(\mathfrak{so}(N,N))=\rho^{(2)}(\mathfrak{so}(N,N))=\rho(N).
\]
Thus the proof of Case~\ref{item:real8} is complete by Lemma~\ref{lem:f(i,N)}.

\vspace{\baselineskip}

\noindent \textbf{Proof of Case~\ref{item:complex2}.}
The proof of Case~\ref{item:complex2} relies on the result of  Case~\ref{item:real8}.
In this case, we utilize the chain \eqref{sequence:2-period} of period $2$.

The following lemma was inspired by the argument starting in the latter half of page 321 in \cite{Adams_Lax_Phillips65}:
\begin{lemma}\label{lem:f(N)}
Let $k\in \N$. 
Suppose that a function $f\colon \N_+\to \N$ satisfies the following three conditions: 
for any $N\in \N_+$,
\begin{enumerate}[label=$(\roman*)$]
   \item\label{item:complex_1} $f(N)+2\leq f(2N)$, 
   \item\label{item:complex_2} $f(N)\leq \rho(N)+k$, 
   \item\label{item:complex_3} $f(N)\geq k$. 
\end{enumerate}    
Then, we have 
$f(N)=2\ord (N)+k$. 
\end{lemma}

\begin{proof}
We note that $\rho(N)=2\ord (N)$ if $\ord(N)\equiv 1 \mod 4$.
For a fixed $N\in \N_{+}$,
take $e\in \N$ satisfying $\ord(2^e N)\equiv 1\mod 4$. 
Then, by the condition \ref{item:complex_2}, we have 
\begin{equation}
   f(2^e N)\leq  \rho(2^e N)+k=2\ord(2^e N)+k=2e+2\ord(N)+k. \label{eq:complex_upper_bound}
\end{equation}

Put $N':=N/2^{\ord(N)}\in \N_{+}$.
By the conditions~\ref{item:complex_1} and~\ref{item:complex_3}, 
we have 
\[
f(N)\geq 2\ord(N)+f(N')\geq 2\ord(N)+k.
\]
Hence, again by the condition~\ref{item:complex_1}, we have 
\begin{equation}
f(2^eN)\geq 2e + f(N)
\geq 2e+2\ord(N)+k.
\label{eq:complex_lower_bound}
\end{equation}
From the inequalities (\ref{eq:complex_upper_bound}) and (\ref{eq:complex_lower_bound}), 
we see that all the inequalities in (\ref{eq:complex_lower_bound}) must be the equalities. 
Thus we obtain $f(N)=2\ord(N)+k$. 
\end{proof}

In Case~\ref{item:complex2}, Theorem~\ref{theorem:Hurwitz-radon}
is nothing but the conclusion obtained by applying Lemma~\ref{lem:f(N)} to the following two choices of $f$ and $k$:
\begin{itemize}
\item $f(N)=\rho^{(i)}(\mathfrak{gl}(N,\C))$ and $k=1$; 
\item $f(N)=\rho^{(i)}(\mathfrak{su}(N,N))(=\rho^{(i)}(\mathfrak{u}(N,N)))$ and $k=2$. 
\end{itemize}

Hence, it suffices to check the conditions \ref{item:complex_1}--\ref{item:complex_3}. 
By Corollary~\ref{cor:paraherm} (see the last two entries in Table~\ref{table:parahermitian}) with Corollary~\ref{corollary:auxiliary-inequality}~\ref{item:gl-sl}, we have 
\begin{align*}
\rho^{(i)}(\mathfrak{gl}(N,\C)) + 2
&\leq \rho^{(i)}(\mathfrak{su}(N,N)) + 1
\leq \rho^{(i)}(\mathfrak{gl}(2N,\C)), \\
\rho^{(i)}(\mathfrak{su}(N,N)) + 2
&\leq \rho^{(i)}(\mathfrak{gl}(2N,\C)) + 1
\leq \rho^{(i)}(\mathfrak{su}(2N,2N)).
\end{align*}
Hence, the condition~\ref{item:complex_1} is verified.

Next, we show the condition~\ref{item:complex_2}. 
By Corollary~\ref{corollary:auxiliary-inequality}~\ref{item:glC-glR} and~\ref{item:su-sp}, we obtain 
\begin{align*}
\rho^{(i)}(\mathfrak{gl}(N,\C))&\leq \rho^{(i)}(\mathfrak{gl}(2N,\R))=\rho(N)+1,\\
\rho^{(i)}(\mathfrak{su}(N,N))&\leq \rho^{(i)}(\mathfrak{sp}(2N,\R))=\rho(N)+2. 
\end{align*}
Here, we used the result in Case~\ref{item:real8}. 
Thus the condition~\ref{item:complex_2} is verified.

Finally, we show the condition~\ref{item:complex_3}.
We have $\rho^{(i)}(\mathfrak{gl}(N,\C)) \geq 1$ by Lemma~\ref{lemma:gl>0}.
Further, again by Corollary~\ref{cor:paraherm}, we obtain 
\[
\rho^{(i)}(\mathfrak{su}(N,N)) \geq \rho^{(i)}(\mathfrak{gl}(N,\C)) + 1 \geq 2.
\]
Hence the condition~\ref{item:complex_3} is proved. Thus
the proof is completed in Case~\ref{item:complex2} by Lemma~\ref{lem:f(N)}. 

\vspace{\baselineskip}

\noindent \textbf{Proof of Case~\ref{item:sl(n,F)}.}
    Note that $\mathfrak{gl}(2N,\D)$ and $\mathfrak{su}(N,N;\D)$ are adjacent terms in the periodic chains~\eqref{sequence:8-period} and~\eqref{sequence:2-period}.
Hence we have
\[
\rho^{(i)}(\mathfrak{gl}(2N,\D))
=
\rho^{(i)}(\mathfrak{su}(N,N;\D))+1.
\]
By combining this equality with Corollary~\ref{cor:paraherm}
(the ninth entry in Table~\ref{table:parahermitian}) 
and Corollary~\ref{corollary:auxiliary-inequality}~\ref{item:gl-sl}, we obtain
\[
\rho^{(i)}(\mathfrak{gl}(2N,\D))
\leq
\rho^{(i)}(\mathfrak{sl}(2N,\D))
\leq
\rho^{(i)}(\mathfrak{gl}(2N,\D)),
\]
and hence
\[
\rho^{(i)}(\mathfrak{sl}(2N,\D))=\rho^{(i)}(\mathfrak{gl}(2N,\D)).
\]
This completes the proof of Case~\ref{item:sl(n,F)}.

\vspace{\baselineskip}

\noindent \textbf{Proof of Case~\hyperref[item:sl(odd)]{$(c')$}.}
We note that the (standard) $\mathfrak{p}$-part of $\mathfrak{sl}(k,\mathbb{D})$
consists of all hermitian matrices over $\mathbb{D}$ whose trace is zero.

For $\rho^{(1)}$, Lemma~\ref{lemma:sl-square-odd} below implies that
\[
\rho^{(1)}(\mathfrak{sl}(2N+1,\D)) = 0 \quad (\D=\R,\C,\HA).
\]

For $\rho^{(2)}$, 
the diagonal matrix
$\diag(2N,-1,\ldots,-1)$
is invertible and belongs to the $\mathfrak{p}$-part of $\mathfrak{sl}(2N+1,\D)$. 
Hence, 
\[
\rho^{(2)}(\mathfrak{sl}(2N+1,\D)) \geq 1. 
\]
Furthermore, we have
\[
\rho^{(2)}(\mathfrak{sl}(2N+1,\D))
\le \rho^{(2)}(\mathfrak{gl}(2N+1,\D))
= 1,
\]
by Corollary~\ref{corollary:auxiliary-inequality}~\ref{item:gl-sl} and the results of Cases~\ref{item:real8},\ref{item:complex2}. 
Therefore,
\[
\rho^{(2)}(\mathfrak{sl}(2N+1,\D)) = 1 \quad (\D=\R,\C,\HA).
\]
This completes the proof of Case~\hyperref[item:sl(odd)]{$(c')$}.

    In the above proof, we have used the following elementary lemma (we omit its proof):
    \begin{lemma}
        \label{lemma:sl-square-odd}
        There is no odd-size square matrix over $\D=\R,\C,\HA$
        whose trace is zero and whose square is the identity matrix.
    \end{lemma}

    This completes the proof of Theorem~\ref{theorem:Hurwitz-radon}.

\section{\texorpdfstring{$\VE$}{Property (VE)} and the \texorpdfstring{Hurwitz--Radon}{Hurwitz-Radon} number}
\label{section:proof_proper-HR}
This section is devoted to the proof of Theorem~\ref{thm:proper-hurwitz-radon}. 
That is, we show that, if $G/H$ has $\VE$ for $\iota$, then
the maximal integer $n$ for which $Spin(n,1)$ acts properly on $G/H$
is given by the Hurwitz–Radon number $\rho^{(1)}(\mathfrak{g}, \iota)$,
introduced in Section~\ref{section:intro-HR} and computed in Section~\ref{section:proof-hurwitz-radon}. 

\subsection{Reduction of 
\texorpdfstring{Theorem~\ref{thm:proper-hurwitz-radon}}{Theorem D}
to a certain proposition
}

In this subsection, we reduce Theorem~\ref{thm:proper-hurwitz-radon}
to the following setting, and then prove Theorem~\ref{thm:proper-hurwitz-radon}
by combining two results,
namely, Lemma~\ref{lem:2_equiv_n} and
Proposition~\ref{proposition:characterization:HR-spin} below. The proof of Proposition~\ref{proposition:characterization:HR-spin}
will be given in the subsequent subsections.
\begin{setting}
    \label{setting:classical-semisimple}
    Let $G$ be a classical semisimple Lie subgroup of $GL(N,\C)$, 
    $\tilde{\iota}\colon G\rightarrow GL(N,\C)$ the inclusion map, 
    and $\iota:=d\tilde{\iota}$ the differential map.
    For the sake of clarity, $G$ is one of the following subgroups:
    \begin{align*}
SL(N,\mathbb{R}),\ SL(N,\mathbb{C}),\ SL(N/2,\mathbb{H}),\ SU(p,q),\
SO(p,q),\\ SO(N,\mathbb{C}),\ SO^{*}(2(N/2)),\ Sp(N/2,\mathbb{R}),\ Sp(N/2,\mathbb{C}),\ Sp(r,s),
\end{align*}
where $p+q = N$ and $r + s = N/2$, and 
we assume that $N/2$ is an integer whenever it appears.
\end{setting}

\begin{lemma}\label{lem:2_equiv_n}
In Setting~\ref{setting:classical-semisimple}, 
we assume that a homogeneous space $G/H$ of reductive type 
has $\VE$ for $\iota$. Then, for any non-trivial homomorphism $\varphi\colon Spin(n,1)\to G$ $(n\geq 2)$, the $Spin(n,1)$-action on $G/H$ via $\varphi$ is proper if and only if 
$\tilde{\iota}(\varphi(-1))=-I_{N}$. 
\end{lemma}

\begin{proof}
    In the following, we think of $SL(2,\R)$ as a
    closed Lie subgroup of $Spin(n,1)$
    by $SL(2,\R)\simeq Spin(2,1)\subset Spin(n,1)$.
    Under this embedding, the element $-I_{2}\in SL(2,\R)$ corresponds to $-1\in Spin(n,1)$, and hence we identify these two elements.
    
    For a Lie group homomorphism $\varphi\colon Spin(n,1)\rightarrow G$,
    we consider the following four claims:
    \begin{enumerate}[label=$(\arabic*)$]
        \item
        \label{lem:2_equiv_n:item:spin(n,1)-proper}
        The $Spin(n,1)$-action on $G/H$ via $\varphi$ is proper; 
        \item 
        \label{lem:2_equiv_n:item:spin(2,1)-proper}
        The $SL(2,\R)$-action on $G/H$ via 
        the restriction $\varphi|_{SL(2,\R)}$ is proper;
        \item
        \label{lem:2_equiv_n:item:very-even}
        The restriction $d\varphi|_{\mathfrak{sl}(2,\R)}$ is very even for $\iota$;
        \item 
        \label{lem:2_equiv_n:item:half-int}
        $\tilde{\iota}(\varphi(-1))=-I_{N}$. 
    \end{enumerate}

    To prove our Lemma (\ref{lem:2_equiv_n:item:spin(n,1)-proper}
    $\Leftrightarrow$ \ref{lem:2_equiv_n:item:half-int}), 
    let us show \ref{lem:2_equiv_n:item:spin(n,1)-proper}
    $\Leftrightarrow$ \ref{lem:2_equiv_n:item:spin(2,1)-proper} $\Leftrightarrow$ \ref{lem:2_equiv_n:item:very-even} $\Leftrightarrow$ \ref{lem:2_equiv_n:item:half-int}.
    Since $Spin(n,1)$ has the same real rank as $SL(2,\R)$,
    \ref{lem:2_equiv_n:item:spin(n,1)-proper}
    $\Leftrightarrow$ \ref{lem:2_equiv_n:item:spin(2,1)-proper}
    follows from Kobayashi's properness criterion (Fact~\ref{fact:properness_criterion}).
    Since $G/H$ has $\VE$ for $\iota$, we have \ref{lem:2_equiv_n:item:spin(2,1)-proper} $\Leftrightarrow$ \ref{lem:2_equiv_n:item:very-even}.
    Further, \ref{lem:2_equiv_n:item:very-even} $\Leftrightarrow$  \ref{lem:2_equiv_n:item:half-int} follows from the representation theory of $SL(2,\R)$. Hence, 
    the assertion is proved.
\end{proof}

The equivalent conditions in Lemma~\ref{lem:2_equiv_n}
are related to the Hurwitz--Radon number $\rho^{(1)}(\mathfrak{g},\iota)$ as follows:
\begin{proposition}
\label{proposition:characterization:HR-spin}
    Let $n\geq 2$.
    In Setting~\ref{setting:classical-semisimple},
    $n\leq \rho^{(1)}(\mathfrak{g},\iota)$ holds if and only if 
    there exists a homomorphism $\varphi\colon Spin(n,1)\rightarrow G$
    with $\tilde{\iota}(\varphi(-1))=-I_{N}$.
\end{proposition}

We postpone the proof of Proposition~\ref{proposition:characterization:HR-spin} to the subsequent subsections and prove Theorem~\ref{thm:proper-hurwitz-radon}.

\begin{proof}[Proof of Theorem~\ref{thm:proper-hurwitz-radon} assuming   Proposition~\ref{proposition:characterization:HR-spin}]
As a corollary of the properness criterion (Fact~\ref{fact:properness_criterion}), 
recall Corollary~\ref{cor:properness-locally-isom} and Lemma~\ref{lem:VE-locally-isomorphic}. Namely, 
the property that $G/H$ admits a proper $Spin(n,1)$-action,
as well as $\VE$,
depends only on the local isomorphism class of $G/H$.
Hence, for the proof, we may  replace $G/H$ by a locally isomorphic homogeneous space. 
Since $(\mathfrak{g},\iota)$ is a classical pair, we may assume Setting~\ref{setting:classical-semisimple}.
Under this setting, the assertion follows immediately from
Lemma~\ref{lem:2_equiv_n}
and Proposition~\ref{proposition:characterization:HR-spin}.
\end{proof}

\subsection{Strategy of the proof of 
\texorpdfstring{Proposition~\ref{proposition:characterization:HR-spin}}{Proposition 5.3}}
\label{section:proof- proposition:characterization:HR-spin}
In this subsection, we explain the strategy of the proof of
Proposition~\ref{proposition:characterization:HR-spin} and outline the structure of the subsequent subsections.

Proposition~\ref{proposition:characterization:HR-spin} is exactly the equivalence 
\ref{item:HR-nleqrho(G)} $\Leftrightarrow$ \ref{item:Spin(n,1)toG_half-int} in the following: 
\begin{proposition}\label{prop:key_for_HR}
    In Setting~\ref{setting:classical-semisimple},
    the following claims are equivalent for any integer $n\geq 2$:
    \begin{enumerate}[label=$(\arabic*)$]
        \item \label{item:HR-nleqrho(G)} $n\leq \rho^{(1)}(\mathfrak{g},\iota)$;
        \item \label{item:Spin(n,1)toG_half-int} There exists a Lie group 
        homomorphism $\varphi\colon Spin(n,1)\rightarrow G$
        such that $\tilde{\iota}(\varphi(-1))=-I_{N}$;
        \item \label{item:Spin(n,1)toG_spinor}
        There exists a Lie group 
        homomorphism $\varphi\colon Spin(n,1)\rightarrow G$
        such that the representation $\tilde{\iota}\circ\varphi$
        is equivalent to a direct sum of several copies of the spin representation $S$ when $n$ is even, and to a direct sum of several copies of the semispin representations $S_{1}$
        and $S_{2}$ when $n$ is odd.
    \end{enumerate}
\end{proposition}
Here, we refer to Appendix~\ref{section:clifford-spin} for 
the (semi)spin representations of the spin groups. 
 
The remainder of Section~\ref{section:proof_proper-HR} is devoted to the proof of Proposition~\ref{prop:key_for_HR}.  
We divide the proof into the two equivalences in Proposition~\ref{prop:key_for_HR}:  
\ref{item:HR-nleqrho(G)} $\Leftrightarrow$ \ref{item:Spin(n,1)toG_spinor} and  
\ref{item:Spin(n,1)toG_half-int} $\Leftrightarrow$ \ref{item:Spin(n,1)toG_spinor}.
The first is proved in Section~\ref{section:clifford-hurwitz-radon} using Clifford algebras.
The second is proved
in Sections~\ref{section:half-int-rep-spin-rep-SL(N)} to \ref{section:half-int-rep-spin-rep}.
In the second equivalence, 
the implication \ref{item:Spin(n,1)toG_spinor}
$\Rightarrow$ \ref{item:Spin(n,1)toG_half-int} is trivial by the construction of the (semi)spin representation. 

The proof of the implication
\ref{item:Spin(n,1)toG_half-int}
$\Rightarrow$
\ref{item:Spin(n,1)toG_spinor}
is carried out by dividing the classical groups $G$
into four cases.
We first treat the simplest case $G=SL(N,\C)$
in Section~\ref{section:half-int-rep-spin-rep-SL(N)}.
The proof for the remaining cases is completed in
Section~\ref{section:half-int-rep-spin-rep}.

To construct homomorphisms from $Spin(n,1)$ to a classical group $G$,
we adopt a standard approach:
we analyze which (not necessarily irreducible) representations of $Spin(n,1)$ factor through $G$.
Section~\ref{section:embedding}
presents a general method for this purpose,
based on real, quaternionic, orthogonal, and symplectic representation theory.
For $G=SO(p,q),\,SU(p,q),\,Sp(p,q)$,
this reduces the question to determining the signatures of the representations.
In Section~\ref{section:determine-signature},
we compute the signatures 
in a special setting needed for
Section~\ref{section:half-int-rep-spin-rep}.

\subsection{Proof of  \texorpdfstring{\ref{item:HR-nleqrho(G)} $\Leftrightarrow$ \ref{item:Spin(n,1)toG_spinor} in Proposition~\ref{prop:key_for_HR}}{the equivalence of (1) and (3) in Proposition 5.4}}
\label{section:clifford-hurwitz-radon}
Let $G$ and $\iota$ be as in Setting~\ref{setting:classical-semisimple}, 
and $\mathfrak{g}=\mathfrak{k}+\mathfrak{p}$ a Cartan decomposition. 
    
    \ref{item:HR-nleqrho(G)} $\Rightarrow$ \ref{item:Spin(n,1)toG_spinor}.
    Assume $n\leq \rho^{(1)}(\mathfrak{g},\iota)$, and  take an $\R$-linear map
    $f\colon \R^{n}\rightarrow \mathfrak{p}$ such that
    $\iota(f(v))^2 = \|v\|^2 I_{N}$ holds for any $v\in \R^{n}$.
    By the universality of the Clifford algebra $C(n)$, 
    the map $\iota\circ f$ is extended
    to an $\R$-algebra homomorphism 
    $F\colon C(n) \rightarrow M(N,\C)$.
    Using the $\R$-algebra isomorphism $\eta\colon C(n)\rightarrow C^{+}(n,1)$ in Lemma~\ref{lemma:c(n)-c(n,1)+}, 
    we define a representation $\tau\colon Spin(n,1)\rightarrow GL(N,\C)$ by 
    the restriction of $F\circ\eta^{-1}$
    to $Spin(n,1)$. Thus we obtain the following commutative diagram:
    \[
   \xymatrix{
    \mathfrak{p}_{\mathfrak{spin}(n,1)} \ar[r]^-{\eta^{-1}}_-{\simeq} \ar@{}[d]|{\bigcap} 
    & \R^{n} \ar@{}[d]|{\bigcap} \ar[r]^-{f} & \mathfrak{p} \ar@{}[d]|{\bigcap}^{\ \iota} \\ 
    C^{+}(n,1) \ar[r]^-{\eta^{-1}}_-{\simeq} \ar@{}[d]|{\bigcup} & C(n) \ar[r]^-{F} & M(N,\C) \ar@{}[d]|{\bigcup} \\ 
    Spin(n,1) \ar[rr]^-{\tau} & & GL(N,\C). \\ 
   }
    \]
    Here we note from \eqref{eq:clifford-algebra} that 
    the $C^{+}(n,1)$-module $\C^{N}$ via $F\circ\eta^{-1}$ is 
    decomposed into a direct sum of several copies of the unique simple 
    $C^{+}(n,1)$-module (resp.\ the two simple 
    $C^{+}(n,1)$-modules) if $n$ is even (resp.\ odd).
    Hence, by the construction of the spin representation $S$ (resp.\ the semispin representations $S_{1},S_{2}$), we see that the representation $\tau$ 
    is equivalent to a direct sum of several copies of $S$ (resp.\ $S_{1},S_{2}$) if $n$ is even (resp.\ odd).
    To prove \ref{item:HR-nleqrho(G)} $\Rightarrow$ \ref{item:Spin(n,1)toG_spinor}, 
    it suffices to show 
    that $\tau(Spin(n,1))\subset G$ holds.
    
    Now we consider the differential $d\tau\colon \mathfrak{spin}(n,1)\rightarrow M(N,\C)$ of $\tau$.
    The restriction of $d\tau$ to $\mathfrak{p}_{\mathfrak{spin}(n,1)}$ coincides with 
    $f\circ\eta^{-1}\colon \mathfrak{p}_{\mathfrak{spin}(n,1)} \rightarrow \mathfrak{p}$. Hence  
    we have $d\tau(\mathfrak{p}_{\mathfrak{spin}(n,1)})\subset \mathfrak{p}$. 
    Since the simple Lie algebra
    $\mathfrak{spin}(n,1)$ is generated by $\mathfrak{p}_{\mathfrak{spin}(n,1)}$, we obtain
    $d\tau(\mathfrak{spin}(n,1))\subset \mathfrak{g}$. 
    Hence $\tau(Spin(n,1))$ is contained in $G$ since $Spin(n,1)$  is connected. Thus \ref{item:HR-nleqrho(G)} $\Rightarrow$ \ref{item:Spin(n,1)toG_spinor} is proved.

    \ref{item:Spin(n,1)toG_spinor} $\Rightarrow$ \ref{item:HR-nleqrho(G)}.
    Let $\varphi\colon Spin(n,1)\rightarrow G$ be 
    a Lie group homomorphism
    satisfying the condition in \ref{item:Spin(n,1)toG_spinor}.
    By the definitions of $S$, $S_{1}$ and $S_{2}$,
    we can and do take an $\R$-algebra homomorphism 
    $\Phi\colon C^{+}(n,1)\rightarrow M(N,\C)$ such that 
    the restriction of $\Phi$ to $Spin(n,1)$
    coincides with $\varphi$.
    The composition of $\eta|_{\R^{n}}$ and 
    the differential $d\varphi$
    gives an $\R$-linear map $f \colon \R^{n}\rightarrow \mathfrak{p}$.
    Thus we obtain the following commutative diagram:
    \[
    \xymatrix{
    & Spin(n,1) \ar[r]^-{\varphi} \ar@{}[d]|{\bigcap}  & G \ar@{}[d]|{\bigcap}  \\ 
    C(n) \ar[r]^-{\eta}_-{\simeq} \ar@{}[d]|{\bigcup}  & C^{+}(n,1) \ar[r]^-{\Phi} \ar@{}[d]|{\bigcup} & M(N,\C) \ar@{}[d]|{\bigcup}^{\ \iota} \\
    \R^{n} \ar[r]^-{\eta}_-{\simeq}
    \ar@/_20pt/[rr]_{f} & \mathfrak{p}_{\mathfrak{spin}(n,1)} \ar[r]^-{d\varphi} & \mathfrak{p}
    }
    \]
    Hence, for any $v\in \R^n$, we have 
    \[
    \iota(f(v))^2=\Phi(\eta(v))^2
    =\Phi(\eta(v^2))
    =\Phi(\eta(\|v\|^2))
    =\|v\|^2 I_{N}.
    \]
    Thus \ref{item:Spin(n,1)toG_spinor} $\Rightarrow$ \ref{item:HR-nleqrho(G)} is proved.

\subsection{Proof of 
\texorpdfstring{Proposition~\ref{prop:key_for_HR} \ref{item:Spin(n,1)toG_half-int} $\Rightarrow$ 
    \ref{item:Spin(n,1)toG_spinor}}{the equivalence of (2) and (3) in Proposition 5.4} for \texorpdfstring{$G=SL(N,\C)$}{G=SL(N,C)}} 
    \label{section:half-int-rep-spin-rep-SL(N)}

In this subsection, we prove the implication
\ref{item:Spin(n,1)toG_half-int} $\Rightarrow$
\ref{item:Spin(n,1)toG_spinor} in
Proposition~\ref{prop:key_for_HR}
in the simplest case where $G=SL(N,\C)$.

Let $\varphi\colon Spin(n,1)\to SL(N,\C)$ be a Lie group homomorphism satisfying
the condition in Proposition~\ref{prop:key_for_HR}~\ref{item:Spin(n,1)toG_half-int}. 
Namely, the representation of $Spin(n,1)$ defined by 
\[
\tau:=\tilde{\iota}\circ \varphi,
\]
satisfies 
\begin{equation}
\label{eq:tau-1=-1}
    \tau(-1)=-I_{N},
\end{equation}
where $\tilde{\iota}\colon SL(N,\C)\hookrightarrow GL(N,\C)$ is the natural inclusion.
Then we define a representation $\tau'$ of $Spin(n,1)$ to be a direct sum of copies of the (semi)spin representation as follows: 
\begin{equation}
\label{eq:replace-tau}
\tau' := 
    \begin{cases}
        S^{\oplus\frac{\dim \tau}{\dim S}}
        & (n\in 2\N), \\
        S_{1}^{ \oplus\frac{\dim\tau}{\dim S_1} } 
        & (n\in 2\N+1). 
    \end{cases}
\end{equation}

First, let us verify that, in the definition of $\tau'$,
the multiplicities
$\frac{\dim \tau}{\dim S}$ and
$\frac{\dim \tau}{\dim S_{1}}$
are integers, i.e., that $\tau'$ is well-defined. 

We introduce some notation.
For a Lie group $L$,
let $\irr(L)$ denote the set of equivalence classes of
irreducible finite-dimensional continuous complex representations of $L$.
We set
\begin{equation}
    \label{def:half-int}
    \Lambda_{n}
    :=
    \{
    [\pi]\in \irr(Spin(n,1))
    \mid \pi(-1)=-\id
    \}.
\end{equation}
Since the center of $Spin(n,1)$ is $\{\pm 1\}$
(Lemma~\ref{lemma:spin-center}), $[\pi]\in \Lambda_{n}$ if and only if
$\pi$ does not factor through the double covering
$Spin(n,1)\to SO_{0}(n,1)$.

In the remainder of Section~\ref{section:proof_proper-HR}, we use the following abuse of notation:
for an irreducible representation $\pi$,
we write $\pi\in \Lambda_{n}$,
or say that $\pi$ belongs to $\Lambda_{n}$ if the equivalence class $[\pi]$ belongs to $\Lambda_{n}$.

By \eqref{eq:tau-1=-1}, every irreducible component of $\tau$
belongs to $\Lambda_{n}$.
Therefore, the following fact guarantees that $\tau'$ is well-defined: 
\begin{fact}
\label{fact:half-int-rep-dim}
For any $\pi \in \Lambda_{n}$,  
its dimension $\dim \pi$ is divisible by  
$\dim S=2^{n/2}$ if $n$ is even,  
and by $\dim S_{1}=2^{(n-1)/2}$ if $n$ is odd.
\end{fact}

Although Fact~\ref{fact:half-int-rep-dim} may be known to experts,  
we could not find a suitable reference, and thus give a proof in Appendix~\ref{appendix:spin-representation}.

We now complete the proof.
The representation $\tau'$ has the same dimension $N$ as $\tau$,
and therefore defines a Lie group homomorphism
$Spin(n,1)\to SL(N,\C)$.
By construction, this homomorphism satisfies
the condition in
Proposition~\ref{prop:key_for_HR} \ref{item:Spin(n,1)toG_spinor}.
Hence, in the case $G=SL(N,\C)$,
the implication
\ref{item:Spin(n,1)toG_half-int} $\Rightarrow$
\ref{item:Spin(n,1)toG_spinor}
in Proposition~\ref{prop:key_for_HR} is proved.

In Section~\ref{section:half-int-rep-spin-rep}, we will treat the remaining classical groups $G$.
As in the case $G=SL(N,\C)$,
we replace the representation $\tau$ arising from
$\varphi\colon Spin(n,1)\rightarrow G$
by a direct sum of copies of
$S$, $S_{1}$, and $S_{2}$
without changing the dimension.
The new issue is whether the resulting representation factors through $G$.
In the next two subsections, we give the necessary tools to address this problem.

\subsection{Homomorphisms from Lie groups to classical Lie groups}
\label{section:embedding}
In this subsection, we discuss whether $N$-dimensional complex representations of Lie groups factor through 
the classical Lie subgroups of $GL(N,\C)$.
The results in this subsection may be known to experts,
but we were unable to find an appropriate reference.
Therefore, we include the proofs for the reader's convenience.

In what follows, we use the following notation:
\begin{notation}
\label{notation:subset}
For two subsets $S,T\subset GL(N,\C)$, 
we write $S\subset_{\Int} T$ if there exists $g\in GL(N,\C)$ such that $g^{-1}Sg\subset T$.
\end{notation}

\begin{problem}
\label{prob:embed}
Let 
$\tau\colon L\rightarrow GL(N,\C)$ be a (not necessarily irreducible) continuous representation of a Lie group $L$, and let
$G$ be one of the following subgroups of $GL(N,\C)$:
    \begin{align*}
GL(N,\mathbb{R}),\ GL(N/2,\mathbb{H}),\ O(N,\mathbb{C}),\ Sp(N/2,\mathbb{C}),\\  
U(p,q),\
O(p,q),\ 
O^{*}(2(N/2)),\ Sp(N/2,\mathbb{R}),\ Sp(r,s),
\end{align*}
where $p+q = N$ and $r + s = N/2$, and 
we assume that $N/2$ is an integer whenever it appears.
Then determine whether $\im \tau \subset_{\Int} G$.
\end{problem}

In this subsection, we explain how to solve Problem~\ref{prob:embed}, except for determining the signature. We divide our discussion into the following three cases: 
\begin{enumerate}[label=$(\alph*)$]
    \item 
    \label{item:RHOSp}
    $GL(N,\R),\ GL(N/2,\HA),\ O(N,\C),\ Sp(N/2,\C)$ (Lemma~\ref{lemma:criterion_for_embedding});
    \item 
    \label{item:U(p,q)}
    $U(p,q)$  (Lemma~\ref{lem:criterion_for_Upq});
    
    \item \label{item:O(p,q)Sp(p,q)Sp(n,R)Ostar}
    $O(p,q),\ O^{*}(2(N/2)),\ Sp(N/2,\R),\ Sp(r,s)$ (Proposition~\ref{prop:embedding-criterion}).
\end{enumerate}
The criteria in Cases~\ref{item:U(p,q)} and~\ref{item:O(p,q)Sp(p,q)Sp(n,R)Ostar}
leave some ambiguity in the signature.
However, in the setting needed for our purposes,
this ambiguity will be resolved in the next subsection.

\vspace{\baselineskip}
\noindent \textbf{Case \ref{item:RHOSp}.}
We denote by $\Rep(L)$ the category whose objects are finite-dimensional
complex continuous representations of a Lie group $L$ and whose morphisms
are $L$-intertwining operators.
Let $\tau\mapsto \overline{\tau}$ (resp.\ 
$\tau\mapsto \tau^{\vee}$)
be the covariant (resp.\ contravariant) functor 
of complex conjugate representations (resp.\ dual representations)
in $\Rep(L)$. Note that the squares of these functors are naturally identified with the identity functor of $\Rep(L)$.

For $\dagger\in \{-,\vee\}$ and $\varepsilon\in \{\pm 1\}$,  
we define a subgroup $G_{N}^{\dagger,\varepsilon}$ of $GL(N,\C)$ 
as shown in Table~\ref{table:RHOSp}.
\begin{table}[H]
        \centering
        \begin{tabular}{c|c|c|c|c}
           $G^{\dagger, \varepsilon}_{N}$  & $GL(N,\R)$ & $GL(N/2,\HA)$ & $O(N,\C)$ & $Sp(N/2,\C)$ \\
           \noalign{\hrule height 1.2pt}
           \hline
            $\dagger$ & $-$ & $-$ & $\vee$ & $\vee$\\
            $\varepsilon$ & $1$ & $-1$ & $1$ & $-1$
        \end{tabular}
        \caption{Definition of $G^{\dagger, \varepsilon}_{N}$.}
        \label{table:RHOSp}
    \end{table}
Put 
\begin{align}
\label{eq:delta_dagger}
\delta\equiv\delta(\dagger)
:=\begin{cases}
    1 & (\dagger=-), \\
    -1 & (\dagger=\vee).
\end{cases}    
\end{align}
Then, for an $N$-dimensional representation $\tau$ of $L$,
\begin{align}
\label{equiv:G-dagger-epsilon}
\im\tau \subset_{\Int} G^{\dagger,\varepsilon}_{N}
\Leftrightarrow &\text{ there exists an isomorphism } f\colon \tau\rightarrow \tau^{\dagger} \\
&\text{ such that } (f^{\dagger})^{\delta(\dagger)}\circ f=\varepsilon \operatorname{id}_{\tau},
\nonumber
\end{align}
where $f^{\dagger}$ denotes the image of the morphism $f$ under the functor $\dagger$. 

For each $\dagger \in\{-,\vee\}$, we briefly recall an invariant of an \emph{irreducible}
representation $\pi$
with $\pi^{\dagger} \simeq \pi$.
By Schur's lemma, one can easily check that there exists an isomorphism $f\colon 
\pi\rightarrow \pi^{\dagger}$
such that $(f^{\dagger})^{\delta(\dagger)}\circ f = \lambda\id_{\pi}$
with $\lambda\in \{\pm 1\}$. 
Then we put 
\begin{equation}
\label{def:index}
\ind^{\dagger}\pi:=\lambda \in \{\pm 1\},
\end{equation}
which is independent of the choice of the isomorphism $f$. By \eqref{equiv:G-dagger-epsilon}, it follows that, if $\pi$ is irreducible, then
\[
\im \pi \subset_{\Int} G^{\dagger,\varepsilon}_{N}
\iff \pi^{\dagger}\simeq \pi \text{ and }
\ind^{\dagger}\pi= \varepsilon .
\]
For the computations of $\ind^{-}\pi$ and $\ind^{\vee}\pi$
when $L$ is a real semisimple Lie group,
see Theorem~3 in \S8 and Theorem~2 in \S7 of \cite{Oni04}.
In \cite{Oni04}, $\ind^{-}\pi$ is called the \emph{Cartan index} and denoted by
$\varepsilon(\mathfrak{g}_{0},\rho_{0})$.

In this paper, we need to study Problem \ref{prob:embed} even for 
not necessarily irreducible representations. 
This can be reduced to the irreducible case by using the following lemma:
\begin{lemma}[{cf.\ \cite[Prop.~2.9]{Tojo2019Classification}}]\label{lemma:criterion_for_embedding}
    Fix $\varepsilon \in \{\pm1\}$ and $\dagger\in \{-,\vee\}$.
    Suppose that a continuous representation $\tau\colon L\rightarrow GL(N,\C)$ of a Lie group $L$ is completely reducible.
    Then $\im \tau \subset_{\Int} G^{\dagger, \varepsilon}_{N}$ holds if and only if for any irreducible representation $\pi$ of $L$, we have
    \begin{align}
    \label{eq:lemma:criterion_for_embedding}
    \begin{cases}
        (\varepsilon\ind^{\dagger}\pi)^{[\tau:\pi]}=1 & \text{if }\pi^{\dagger}\simeq \pi, \\
        [\tau:\pi] = [\tau:\pi^{\dagger}] & \text{if }\pi^{\dagger}\not \simeq \pi. 
    \end{cases}
    \end{align}

\end{lemma}

\begin{proof} 
We put 
\begin{equation*}
    A:=\{[\pi] \in \irr(L) \mid \pi^{\dagger} \simeq \pi\} 
    \text{ and }
    B:= (\irr(L)\smallsetminus A)/\langle \dagger \rangle,
\end{equation*}
where $(-)/\langle \dagger \rangle$ means the quotient space by 
the action $\pi\mapsto\pi^\dagger$.

From \eqref{equiv:G-dagger-epsilon}, 
if $\im \tau \subset_{\Int} G^{\dagger, \varepsilon}_{N}$ holds, then $\tau^{\dagger}$ must be equivalent to $\tau$. Consequently, for the purpose of proving the assertion, we may assume initially that $[\tau:\pi] = [\tau:\pi^{\dagger}]$ holds for each $[\pi] \in B$.
Then the restriction of the elements in $\Hom_L(\tau, \tau^{\dagger})$ to each isotypic component of $\tau$ provides the following isomorphism of $\C$-vector spaces:
\begin{align*}
    \Hom_{L}(\tau,\tau^{\dagger}) \simeq 
    \bigoplus_{[\pi]\in A} &M([\tau:\pi],\C) \otimes_{\C} \Hom(\pi,\pi^{\dagger}) \\
    \oplus \bigoplus_{[\pi]\in B} &(M([\tau:\pi],\C) \otimes_{\C} \Hom(\pi,\pi) \\
    & \oplus M([\tau:\pi],\C) \otimes_{\C}\Hom(\pi^{\dagger},\pi^{\dagger})).
\end{align*}

For each $[\pi] \in A$, choose an isomorphism $h_{\pi} \in \Hom(\pi, \pi^{\dagger})$ such that $(h_{\pi}^\dagger)^{\delta(\dagger)} \circ h_{\pi} = (\ind^{\dagger}\pi) \id_{\pi}$. Under the above isomorphism, the element corresponding to $f \in \Hom_{L}(\tau, \tau^{\dagger})$ is denoted by
\[
((A_{f,\pi} \otimes h_{\pi})_{[\pi] \in A}, (B_{f,\pi} \otimes \id_{\pi}, C_{f,\pi} \otimes \id_{\pi^\dagger})_{[\pi] \in B}).
\]
Then $(f^\dagger)^{\delta(\dagger)} \circ f = \varepsilon \id_{\tau}$ holds if and only if the following condition holds for each irreducible representation $\pi$:
\begin{equation*}
\begin{cases}
    (A_{f,\pi}^{\dagger})^{\delta(\dagger)}A_{f,\pi} = (\varepsilon \ind^{\dagger} \pi)I_{[\tau:\pi]} & \text{if }[\pi] \in A, \\[1ex]
    (C^{\dagger}_{f,\pi})^{\delta(\dagger)}B_{f,\pi} = \varepsilon I_{[\tau:\pi]} & \text{if }[\pi] \in B.
\end{cases}
\end{equation*}
Here, we denote by $g^{\dagger}$ 
the complex conjugate of a matrix $g$ if $\dagger = -$ and 
the transpose of $g$ if $\dagger = \vee$.

From \eqref{equiv:G-dagger-epsilon}, $\im \tau \subset_{\Int} G^{\dagger, \varepsilon}_{N}$ holds if and only if there exist matrices $A_{f,\pi}, B_{f,\pi},$ and $C_{f,\pi}$ satisfying the above condition. 
In the case $[\pi] \in A$, such a matrix $A_{f,\pi}$ always exists (put $A_{f,\pi}:=I_{[\tau:\pi]}$) if $\varepsilon \ind^{\dagger} \pi = 1$.
If $\varepsilon \ind^{\dagger} \pi = -1$, the existence of such an $A_{f,\pi}$ is equivalent to $[\tau:\pi]$ being even. These can be summarized as the condition $(\varepsilon \ind^{\dagger} \pi)^{[\tau:\pi]} = 1$. 
In the case $[\pi] \in B$, such matrices $B_{f,\pi}$ and $C_{f,\pi}$ always exist  (put $B_{f,\pi}:=I_{[\tau:\pi]}$ and $C_{f,\pi}:=\varepsilon I_{[\tau:\pi]}$). 
This completes the proof.
\end{proof}

\vspace{\baselineskip}
\noindent \textbf{Case~\ref{item:U(p,q)}.}
To determine whether or not $\im \tau\subset_{\Int} U(p,q)$ (up to signature), we use
the contravariant functor $\tau\mapsto \tau^{*}:= \overline{\tau^{\vee}} (\simeq \overline{\tau}^{\vee})$.
\begin{lemma}\label{lem:criterion_for_Upq}
Let $\tau\colon L\to GL(N,\C)$ be a (not necessarily irreducible) continuous representation of a Lie group $L$. Then there exist some $p,q\in \N$ with $p+q=N$ such that $\im \tau\subset_{\Int} U(p,q)$ if and only if $\tau^{*}$ is equivalent to $\tau$. 
\end{lemma}

\begin{proof}
    We note that $\im \tau\subset_{\Int} U(p,q)$
    if and only if there exists a non-degenerate $\tau(L)$-invariant 
    hermitian form on $\C^{N}$ of signature $(p,q)$.
    Hence the ``only if''-part of our assertion is obvious, 
    and thus we focus on the proof of the ``if''-part.

    In the following, we prove that if $\tau^{*}\simeq \tau$  in $\Rep(L)$, 
    there exists an isomorphism $h\colon \tau\rightarrow \tau^{*}$ in $\Rep(L)$ with 
    $h^{*} = h$. Here we have identified $(\tau^{*})^{*}$
    naturally with $\tau$.
    If we could prove this, 
    we obtain a non-degenerate $\tau(L)$-invariant hermitian form on $\C^{N}$ by $(u,v)\mapsto \pairing{h(u)}{v}$, where 
    $\pairing{\cdot}{\cdot}$ is the natural pairing of $(\C^{N})^{\vee}$ and $\C^{N}$.
    Hence the ``if''-part of our assertion is also proved.

    Given an isomorphism $f\colon \tau\rightarrow \tau^{*}$
    in $\Rep(L)$, we define a polynomial map $\tilde{f}\colon \C\rightarrow \Hom_{L}(\tau,\tau^{*})$ by
    \[
    \tilde{f}(t):= \left(\frac{f+f^{*}}{2}\right) - \sqrt{-1}t
    \left(\frac{f-f^{*}}{2}\right).
    \]
    We note that the condition that $\tilde{f}(t) \in \Hom_{L}(\tau,\tau^{*})$ is an isomorphism is a
    Zariski-open condition with respect to $t\in \C$,
    and that $\tilde{f}(\sqrt{-1}) = f$ is an isomorphism.
    Hence there exists $t_{0}\in \R$ such that 
    $h:=\tilde{f}(t_{0})$ is also an isomorphism.
    Further, we have $h^{*} = h$ because of $t_{0}\in \R$,
    and thus obtain the desired isomorphism $h$. 
    From the above, the proof of our assertion is complete.
\end{proof}

\vspace{\baselineskip}
\noindent \textbf{Case~\ref{item:O(p,q)Sp(p,q)Sp(n,R)Ostar}}
Using the following proposition, 
we can reduce Case~\ref{item:O(p,q)Sp(p,q)Sp(n,R)Ostar} to Case~\ref{item:RHOSp} (up to signature): 
\begin{proposition}
\label{prop:embedding-criterion}
Let $\tau\colon L\rightarrow GL(N,\mathbb{C})$
be a continuous representation of a Lie group $L$. 
Then the following claims hold: 
\begin{align*}
\E p,q\in\N,\ p+q=N\text{ and }\im\tau\subset_{\Int} 
O(p,q)
&\Leftrightarrow 
\begin{cases}
\im\tau\subset_{\Int} O(N,\C); & \\
\im\tau\subset_{\Int} GL(N,\R);
\end{cases} \\
\im\tau\subset_{\Int} Sp(N/2,\R)
&\Leftrightarrow 
\begin{cases}
\im\tau\subset_{\Int} Sp(N/2,\C); & \\
\im\tau\subset_{\Int} GL(N,\R);
\end{cases} \\
\E r,s\in\N,\ r+s=N/2\text{ and }
\im\tau\subset_{\Int} Sp(r,s)
&\Leftrightarrow
\begin{cases}
\im\tau\subset_{\Int} Sp(N/2,\C); & \\
\im\tau\subset_{\Int} GL(N/2,\HA);
\end{cases} \\
\im\tau\subset_{\Int} O^{*}(2(N/2))
&\Leftrightarrow 
\begin{cases}
\im\tau\subset_{\Int} O(N,\C); & \\
\im\tau\subset_{\Int} GL(N/2,\HA).
\end{cases}
\end{align*}
\end{proposition}
We will prove this assertion in Appendix~\ref{appendix:embedding}.

\subsection{Signatures of \texorpdfstring{$Spin(k,l)$}{Spin(k,l)}-invariant hermitian forms}
\label{section:determine-signature}
In this subsection, we determine the signatures of
$Spin(k,l)$-invariant hermitian forms in a special setting
(Lemma~\ref{lemma:signature_of_half_integral_rep}),
and use this result to strengthen the criteria in Cases~\ref{item:U(p,q)} and~\ref{item:O(p,q)Sp(p,q)Sp(n,R)Ostar}
of the previous subsection
(Lemma~\ref{lemma:embeddability-spin-u(p,q)-spin}).

\begin{lemma}\label{lemma:signature_of_half_integral_rep}
    Let $\tau\colon Spin(k,l)\rightarrow GL(N,\C)$ be 
    a continuous (not necessarily irreducible) representation such that $\tau(-1)=-I_{N}$. Suppose that $k+l\geq 3$ and $\min(k,l)\geq 1$. Assume one of the following claims:
    \begin{itemize}
        \item $\im \tau \subset_{\Int} O(p,q)$ $(p+q=N)$.
        \item $\im \tau \subset_{\Int} U(p,q)$ $(p+q=N)$.
        \item $N$ is even and $\im \tau \subset_{\Int} Sp(r,s)$ $(r+s=N/2)$.
    \end{itemize}
    Then we have $p=q$ and $r=s$.
\end{lemma}

\begin{proof}

By restricting $\tau$ to $Spin(2,1)\subset Spin(k,l)$, 
it suffices to consider the case $(k,l)=(2,1)$.
Since $\tau(-1)=-I_{N}$, 
it follows that $\tau$ is a continuous representation of 
$SL(2,\mathbb{R})\simeq Spin(2,1)$ 
whose irreducible components are all even-dimensional.

Since $O(p,q)\subset U(p,q)$ and 
$Sp(r,s)\subset U(2r,2s)$, 
it suffices to treat the case $\im \tau \subset_{\Int} U(p,q)$. 
Given an arbitrary $SL(2,\R)$-invariant non-degenerate hermitian form $h$ on the representation space $V$ of $\tau$,
we prove that $h$ has the same number of positive and negative signs.
For $m\in \N$, 
we denote by $V_{m}$ the $m$-dimensional irreducible $SL(2,\R)$-module.
Via an isomorphism of $SL(2,\R)$-modules
$V\simeq \bigoplus_{m\in \N} \C^{[V:V_{2m}]}\otimes_{\C}V_{2m}$,
we identify $h$ with the sum of 
the tensor products of non-degenerate hermitian forms on $\C^{[V:V_{2m}]}$ and 
the $SL(2,\R)$-invariant non-degenerate hermitian forms $h_{2m}$ on $V_{2m}$ for all $m\in \N$. 
Since $h_{2m}$ is of  signature $(m,m)$ for each $m\in \N$, our claim is proved.
\end{proof}

We label the following four classical groups in Case~\ref{item:O(p,q)Sp(p,q)Sp(n,R)Ostar} by symbols $G^{s}_{N}$ as in Table~\ref{table:def-Gs} using a map $s \colon \{-, \vee\} \to \{\pm 1\}$.
We apply the following lemma 
to treat Cases~\ref{item:U(p,q)} and~\ref{item:O(p,q)Sp(p,q)Sp(n,R)Ostar} in the next subsection:

\renewcommand{\arraystretch}{1.5}
\begin{table} 
\begin{tabular}{c!{\vrule width 1.2pt}c|c}
    $G^{s}_{N}$ & $s(-)=1$ & $s(-)=-1$ \\
    \noalign{\hrule height 1.2pt}
    \hline
    $s(\vee)=1$ & $SO(N/2,N/2)$ & $SO^{*}(2(N/2))$ \\
    \hline 
    $s(\vee)=-1$ & $Sp(N/2,\R)$ & $Sp(N/4,N/4)$ \\
\end{tabular}
\caption{Definition of $G^{s}_{N}$.}
\label{table:def-Gs}
\end{table}
\begin{lemma}
    \label{lemma:embeddability-spin-u(p,q)-spin}
    Let $\tau\colon Spin(k,l)\rightarrow GL(N,\C)$
    be a (not necessarily irreducible) continuous representation with $\tau(-1)=-I_{N}$.
    Suppose that $k+l\geq 3$ and $\min(k,l)\geq 1$. 
    Then the following two claims hold:
    \begin{enumerate}[label=$(\roman*)$]
        \item \label{lemma:embeddability-spin-u(p,q)-spin:u(p,q)} $\im\tau \subset_{\Int} SU(N/2,N/2)$ 
        if and only if $\tau^{*}$ is equivalent to $\tau$.
        \item \label{lemma:embeddability-spin-u(p,q)-spin:Gs} 
        Let $s\colon \{-,\vee\}\rightarrow \{\pm 1\}$ be a map.
        Then 
        $\im \tau \subset_{\Int} G^{s}_{N}$  
        if and only if $\im \tau \subset_{\Int} G^{\dagger,s(\dagger)}_{N}$ for each $\dagger\in\{-,\vee\}$.
    \end{enumerate}
We recall Table~\ref{table:RHOSp} in the previous subsection for the definition of $G^{\dagger,\varepsilon}_{N}$.
\end{lemma}
\renewcommand{\arraystretch}{1}
\begin{proof}
Since $Spin(k,l)$ is connected and semisimple, 
    \ref{lemma:embeddability-spin-u(p,q)-spin:u(p,q)} and \ref{lemma:embeddability-spin-u(p,q)-spin:Gs}
    are proved immediately by combining Lemma~\ref{lemma:signature_of_half_integral_rep}
    with Lemma~\ref{lem:criterion_for_Upq} and Proposition~\ref{prop:embedding-criterion}, respectively.
\end{proof}

\subsection{Proof of 
\texorpdfstring{Proposition~\ref{prop:key_for_HR} \ref{item:Spin(n,1)toG_half-int} $\Rightarrow$ 
    \ref{item:Spin(n,1)toG_spinor}}{the equivalence of Proposition 5.4 (2) and (3)} for the remaining classical groups}
    \label{section:half-int-rep-spin-rep}

Let $G$ be a classical semisimple subgroup of $GL(N,\C)$
other than $SL(N,\C)$,
and let $\tilde{\iota}\colon G \hookrightarrow GL(N,\C)$
be the natural inclusion.
In this section, we prove the implication
\ref{item:Spin(n,1)toG_half-int}
$\Rightarrow$
\ref{item:Spin(n,1)toG_spinor}
in Proposition~\ref{prop:key_for_HR}
for $G$.
Namely, let $\varphi\colon Spin(n,1)\rightarrow G$ be a Lie group homomorphism
such that $\tilde{\iota}(\varphi(-1))=-I_{N}$.
Then we construct a representation of $Spin(n,1)$
of the same dimension as the representation $\tilde{\iota}\circ\varphi$,
given as a direct sum of several copies of the
(semi)spin representations $S,S_{1},S_{2}$ of $Spin(n,1)$,
which factors through $G$. 

To state a lemma (Lemma~\ref{lem:construction_pi_dagger}) yielding the desired representation, we present a result concerning the invariant
$\ind^{\dagger}\pi$ given in 
\eqref{def:index}. 
The following fact shows that
irreducible representations of $Spin(n,1)$ belonging to 
$\Lambda_n$, i.e., not factoring through $SO(n,1)$, 
exhibit the same behavior as the (semi)spin representations, from the viewpoint of real, quaternionic, orthogonal, and symplectic representation theory.
\begin{fact}
\label{fact:index_for_spin1}
    Fix $n\geq 2$ and  $\dagger\in \{-,\vee,*\}$.
    \begin{itemize}
        \item Suppose $n$ is even.  
        Then $\pi^{\dagger}\simeq \pi$ for any $\pi\in \Lambda_{n}$.
        Further, if $\dagger\in \{-,\vee\}$, then 
        $\ind^{\dagger}\pi = \ind^{\dagger}S$.         
        \item Suppose $n$ is odd.
        Then exactly one of the following holds:
        \begin{enumerate}[label=$(\alph*)_{\dagger}$]
            \item 
            $\pi^{\dagger}\simeq \pi$
            for any $\pi\in \Lambda_{n}$. 
            Further, if $\dagger\in \{-,\vee\}$, then $\ind^{\dagger}\pi = \ind^{\dagger}S_{1}=\ind^{\dagger}S_{2}$ holds.  
            \item 
            $\pi^{\dagger}\not\simeq \pi$ 
            for any $\pi\in \Lambda_{n}$. 
            Further, $S_{1}^{\dagger} \simeq S_{2}$.
        \end{enumerate}
    \end{itemize}
\end{fact}

This fact follows from the explicit value of the index for irreducible representations of $Spin(n,1)$,  
but we give a more direct proof in Appendix~\ref{appendix:spin-representation},  since it explains more transparently why Theorem~\ref{thm:proper-hurwitz-radon} holds.

For $\dagger\in \{-, \vee,*\}$, we put 
\[
\N(\dagger):=
\{n\in 2\N + 1\mid (a)_{\dagger}\text{ holds in Fact~\ref{fact:index_for_spin1}}\} \subset 2\N+1.
\] 
Although an explicit description of $\N(\dagger)$
is not required in the proof of Proposition~\ref{prop:key_for_HR},
it can be described as follows:
\[
\N(\dagger)=
\begin{cases}
\{n\in \N_{+} \mid 
\delta(\dagger)n  \equiv 1,-3 \mod 8\}  & (\dagger=-,\vee), \\
\emptyset & (\dagger=*),
\end{cases}
\]
where $\delta(\dagger)$ is defined in
\eqref{eq:delta_dagger}.

Let $\dagger \in \{-,\vee, *\}$
and suppose that a representation $\tau\colon Spin(n,1)\rightarrow GL(N,\C)$ satisfies
\begin{align}
&\tau^{\dagger} \simeq \tau, 
\label{eq:tau-self-dagger}\\
&\tau(-1)=-I_{N}.
\label{eq:tau-half-int}
\end{align}
Then we define a representation $\tau_{\dagger}$ 
as a direct sum of several copies of the (semi)spin representations,
as follows:
\begin{align}
\label{eq:tau_dagger}
\tau_{\dagger} := 
    \begin{cases}
        S^{\oplus\frac{\dim \tau}{\dim S}}
        & (n\in 2\N), \\
        S_{1}^{\oplus\frac{\dim \tau}{\dim S_{1}}} 
        & (n\in \N(\dagger)), \\
        (S_{1}\oplus S_{2})^{\oplus\frac{\dim \tau}{2\dim S_{1}}}
         & (n\in(2\N+1) \smallsetminus\N(\dagger)).
    \end{cases}
\end{align}
In the first two cases, the fact that $\tau_{\dagger}$ is well-defined
follows from the 
condition~\eqref{eq:tau-half-int} and 
Fact~\ref{fact:half-int-rep-dim},
in the same way as shown in
Section~\ref{section:half-int-rep-spin-rep-SL(N)}.

In the third case, for any irreducible representation $\pi\in \Lambda_{n}$,
the representation $\pi^{\dagger}$ is not equivalent to $\pi$.
Moreover, by \eqref{eq:tau-self-dagger},
we have $[\tau:\pi]=[\tau:\pi^{\dagger}]$.
Hence $\tau$ decomposes as a direct sum of representations of the form
$\pi\oplus \pi^{\dagger}$.
Again by Fact~\ref{fact:half-int-rep-dim},
it follows that $\dim\tau$ is divisible by $2\dim S_{1}$.
Therefore $\tau_{\dagger}$ is well-defined also in this case.

The passage from $\tau$ to $\tau_{\dagger}$ preserves
several properties of the representation as follows:
\begin{lemma}
    \label{lem:construction_pi_dagger}
    Fix $\dagger \in \{-,\vee,*\}$ and 
    let $\tau,\tau_{\dagger}$ be as above. 
    Then
    \begin{equation*}
        \dim \tau_{\dagger}=\dim \tau
        \quad \text{and}\quad
        (\tau_{\dagger})^{\dagger}\simeq \tau_{\dagger}.
    \end{equation*}
    Moreover, assume $\dagger\in\{-,\vee\}$ and $\varepsilon=\pm 1$.
    Then the following hold:
    \begin{enumerate}[label=$(\roman*)$]
        \item 
        \label{item:tau_dagger-subset-g-dagger}
        
    If $\im \tau \subset_{\Int} G^{\dagger,\varepsilon}_{N}$ holds, then $\im \tau_{\dagger} \subset_{\Int} G^{\dagger,\varepsilon}_{N}$
    also holds;
    \item 
    \label{item:tau_dagger-subset-g-dagger*}
    Assume  $n\in (2\N+1)\smallsetminus \N(\dagger)$.
    If $\im \tau \subset_{\Int} G^{\dagger*,\varepsilon}_{N}$ holds, then $\im \tau_{\dagger} \subset_{\Int} G^{\dagger*,\varepsilon}_{N}$
    also holds.
    \end{enumerate} 
\end{lemma}
Here, if $\dagger=-$ (resp.\ $\vee$), then $\dagger*=\vee$ (resp.\ $-$).
\begin{proof}
    By definition, it is obvious that $\dim \tau_{\dagger}=\dim \tau$.     
    The fact that
    $(\tau_{\dagger})^{\dagger}\simeq \tau_{\dagger}$
    follows immediately from
    Fact~\ref{fact:index_for_spin1}.

    \ref{item:tau_dagger-subset-g-dagger}.
    Fix $\dagger\in\{-,\vee\}$ and $\varepsilon=\pm 1$. 
    Let us check the condition \eqref{eq:lemma:criterion_for_embedding}
    in Lemma~\ref{lemma:criterion_for_embedding} for the representation
    $\tau_{\dagger}$ by distinguishing the following three cases:
    $n \in 2\N$, 
    $n \in \N(\dagger)$, 
    and 
    $n \in (2\N+1)\smallsetminus \N(\dagger)$.
   
    In the case where $n\in 2\N$,
    we have $\pi^{\dagger}\simeq \pi$ and $\ind^{\dagger}S=\ind^{\dagger}\pi$ for all $\pi\in \Lambda_{n}$
    by Fact~\ref{fact:index_for_spin1}.
    Since $\im \tau \subset_{\Int} G^{\dagger,\varepsilon}_{N}$, it follows from  Lemma~\ref{lemma:criterion_for_embedding} that 
    for any $\pi\in \Lambda_{n}$ we have 
    $(\varepsilon\ind^{\dagger}\pi)^{[\tau:\pi]} =1$.
    Here we note that 
    \[
    \dim\tau=\sum_{\pi\in \Lambda_{n}}[\tau:\pi]\dim\pi. 
    \]
    Hence, we have 
    \[
    (\varepsilon\ind^{\dagger}S)^{[\tau_{\dagger}:S]}
    = 
    (\varepsilon\ind^{\dagger}S)^{\frac{\dim \tau}{\dim S}}
    =
    \prod_{\pi\in \Lambda_{n}} ((\varepsilon\ind^{\dagger}\pi)^{[\tau:\pi]})^{\frac{\dim \pi}{\dim S}} = 1.
    \] 
    Thus the condition \eqref{eq:lemma:criterion_for_embedding}
    holds for the representation $\tau_{\dagger}$. 
    
    In the case where $n\in \N(\dagger)$,
the condition \eqref{eq:lemma:criterion_for_embedding} is verified
by replacing $S$ with $S_{1}$ in the argument for the case $n\in 2\N$.
    
    In the case where $n\in (2\N+1)\smallsetminus\N(\dagger)$, 
    we have $S_{1}^{\dagger}\simeq S_{2}$ by Fact~\ref{fact:index_for_spin1}. By the definition of $\tau_{\dagger}$, we have  
    $[\tau_{\dagger}:S_{1}]=[\tau_{\dagger}:S_{2}]$,
    which is the condition \eqref{eq:lemma:criterion_for_embedding}.

    Thus, by Lemma~\ref{lemma:criterion_for_embedding},
    we conclude 
    $\im \tau_{\dagger} \subset_{\Int} G^{\dagger,\varepsilon}_{N}$
    in all cases.

    \ref{item:tau_dagger-subset-g-dagger*}. 
    Let $\dagger\in\{-,\vee\}$, $\varepsilon=\pm 1$, and $n\in (2\N+1)\smallsetminus \N(\dagger)$.
    We now check the condition \eqref{eq:lemma:criterion_for_embedding}
    in Lemma~\ref{lemma:criterion_for_embedding} for the representation
    $\tau_{\dagger}$ by distinguishing cases according to whether $n$ belongs to $\N(\dagger*)$ or not.

In the case where $n\not\in \N(\dagger*)$, we have $S_{1}^{\dagger*}\simeq S_{2}$.
Since $[\tau_{\dagger}:S_{1}]=[\tau_{\dagger}:S_{2}]$ by the definition of $\tau_{\dagger}$,
the condition \eqref{eq:lemma:criterion_for_embedding} is satisfied.

In the case where $n\in \N(\dagger*)$, $\pi^{\dagger*}\simeq \pi$ and
$\ind^{\dagger*}\pi=\ind^{\dagger*} S_{1}$ for any $\pi\in \Lambda_{n}$.
Since $\im \tau \subset_{\Int} G^{\dagger*,\varepsilon}_{N}$,
it follows from Lemma~\ref{lemma:criterion_for_embedding} that $(\varepsilon\,\ind^{\dagger*}\pi)^{[\tau:\pi]}=1$ for all $\pi\in \Lambda_{n}$. 
Here we note that 
\[
\dim\tau =2\sum_{\pi\in \Lambda_{n}/\langle \dagger\rangle}[\tau:\pi]\dim \pi,
\]
where $\Lambda_{n}/\langle\dagger\rangle $ is the quotient set of $\Lambda_{n}$
by the action $\pi\mapsto \pi^{\dagger}$.  
Therefore,
\begin{align*}
(\varepsilon\,\ind^{\dagger*}S_{1})^{[\tau_{\dagger}:S_{1}]}
&=
(\varepsilon\,\ind^{\dagger*}S_{1})^{\frac{\dim\tau}{2\dim S_{1}}} \\
&=
\prod_{\pi\in \Lambda_{n}/\langle \dagger\rangle}
\bigl((\varepsilon\,\ind^{\dagger*}\pi)^{[\tau:\pi]}\bigr)^{\frac{\dim \pi}{\dim S_{1}}}
=1.
\end{align*}
Similarly, $(\varepsilon\,\ind^{\dagger*}S_{2})^{[\tau_{\dagger}:S_{2}]}=1$. 
Thus the condition \eqref{eq:lemma:criterion_for_embedding} is verified. 

By Lemma~\ref{lemma:criterion_for_embedding}, we conclude 
$\im \tau_{\dagger} \subset_{\Int} G^{\dagger*,\varepsilon}_{N}$
in all cases.
\end{proof}

Returning to our main task, we prove
\ref{item:Spin(n,1)toG_half-int} $\Rightarrow$ \ref{item:Spin(n,1)toG_spinor}
in Proposition~\ref{prop:key_for_HR}
for the classical subgroups $G$ of $GL(N,\C)$ other than $SL(N,\C)$.

Let $\tilde{\iota}\colon G\rightarrow GL(N,\C)$ be the natural inclusion and 
$\varphi\colon Spin(n,1)\rightarrow G$ be a 
Lie group homomorphism satisfying the condition in \ref{item:Spin(n,1)toG_half-int}. 
Then the representation 
$\tau:=\tilde{\iota}\circ \varphi$ 
satisfies $\im \tau \subset_{\Int} G$ and $\tau(-1)=-I_{N}$.

We note that if $G$ is one of 
$SO(p,q)$, $SU(p,q)$, or $Sp(r,s)$,
then Lemma~\ref{lemma:signature_of_half_integral_rep}
implies that $p=q$ and $r=s$.
Therefore, it suffices to assume that $G$ is one of the following:
\begin{enumerate}[label=$(\alph*)$]
    \item 
    \label{item:RHOSp-spin}
    $G_{N}^{\dagger,\varepsilon}\cap SL(N,\C)$ (Table~\ref{table:RHOSp});
    \item 
    \label{item:U(p,q)-spin}
    $SU(N/2,N/2)$;
    \item 
    \label{item:O(p,q)Sp(p,q)Sp(n,R)Ostar-spin}
    $G_{N}^{s}$  (Table~\ref{table:def-Gs}).
\end{enumerate}
To prove the assertion, it suffices in each case to choose an appropriate 
$\star\in \{-,\vee,*\}$ and show that the representation $\tau_{\star}$ defined in \eqref{eq:tau_dagger}
factors through $G$. 
Indeed, by Lemma~\ref{lem:construction_pi_dagger},
$\tau_{\star}$ has the same dimension as $\tau$, 
and by construction it is given as a direct sum of several copies of the 
(semi)spin representations $S$, $S_{1}$, and $S_{2}$. 
This means that $\tau_{\star}$ induces a Lie group homomorphism from $Spin(n,1)$ to $G$ satisfying the condition in \ref{item:Spin(n,1)toG_spinor}.

\ref{item:RHOSp-spin}. 
By Lemma~\ref{lem:construction_pi_dagger}~\ref{item:tau_dagger-subset-g-dagger},
$\tau_{\dagger}$ is a desired representation.

\ref{item:U(p,q)-spin}. 
By Lemma~\ref{lem:construction_pi_dagger}, we have $(\tau_{*})^{*}\simeq \tau_{*}$.
Hence $\im \tau_{*}\subset_{\Int} SU(N/2,N/2)$ holds by Lemma~\ref{lemma:embeddability-spin-u(p,q)-spin}~\ref{lemma:embeddability-spin-u(p,q)-spin:u(p,q)}.

\ref{item:O(p,q)Sp(p,q)Sp(n,R)Ostar-spin}.
Since $\im \tau \subset_{\Int} G_{N}^{s}$ holds, we have
\begin{equation}
\label{eq:tau-assumption}
\im \tau \subset_{\Int} G_{N}^{\dagger,s(\dagger)}
\quad \text{for each } \dagger\in\{-,\vee\}.
\end{equation}

We construct a representation
$\tau_{s}\colon Spin(n,1)\to GL(N,\C)$
by distinguishing two cases according to whether
$n\in 2\N \cup (\N(-)\cap\N(\vee))$ (Case~1)
or not (Case~2). 
By Lemma~\ref{lemma:embeddability-spin-u(p,q)-spin}~\ref{lemma:embeddability-spin-u(p,q)-spin:Gs}, it suffices to construct $\tau_{s}$ so that
\begin{align}
\label{eq:want-to-show}
\im \tau_{s} \subset_{\Int} G_{N}^{\dagger,s(\dagger)}
\quad \text{for each } \dagger\in\{-,\vee\}.
\end{align}

In Case~1, we have $\tau_{-}=\tau_{\vee}$ by definition, and we set
\[
\tau_{s}:=\tau_{-}=\tau_{\vee}.
\]
Then \eqref{eq:want-to-show} follows from \eqref{eq:tau-assumption} and
Lemma~\ref{lem:construction_pi_dagger}~\ref{item:tau_dagger-subset-g-dagger}.

In Case~2, we take $\star\in\{-,\vee\}$ such that $n\notin\N(\star)$, and set
\[
\tau_{s}:=\tau_{\star}.
\]
By using \eqref{eq:tau-assumption}, it follows from
Lemma~\ref{lem:construction_pi_dagger}~\ref{item:tau_dagger-subset-g-dagger}
and \ref{item:tau_dagger-subset-g-dagger*} that
\eqref{eq:want-to-show} holds for $\dagger= \star$ and $\dagger= \star*$, respectively.

Therefore, Theorem~\ref{thm:proper-hurwitz-radon} has been 
established.

\section{Classification of symmetric spaces with \texorpdfstring{$\VE$}{Property (VE)}}
\label{section:classification_VE}
The main goal of this section is to prove 
Theorem~\ref{theorem:classification_of_symmetric_space_with_very_even_condition}, 
namely, to give an infinitesimal classification of symmetric spaces $G/H$ with $\mathfrak{g}$ simple 
that have $\VE$ and admit a proper $SL(2,\R)$-action.

In Section~\ref{section:proper-symmetric}, we summarize the results of Teduka~\cite{Teduka2008PJA} and Okuda~\cite{Oku13}
on proper $SL(2,\R)$-actions on semisimple symmetric spaces $G/H$,
where properness is determined in terms of the
weighted Dynkin diagrams
associated with Lie group homomorphisms $SL(2,\R) \rightarrow G$. 
In Section~\ref{section:characterization-very-even}, 
we also characterize  very even $\mathfrak{sl}(2,\R)$-homomorphisms 
in terms of the associated
weighted Dynkin diagrams. 
In Section~\ref{section:proof-theorem-classification}, 
we apply the results in Sections~\ref{section:proper-symmetric} and~\ref{section:characterization-very-even} to prove 
Theorem~\ref{theorem:classification_of_symmetric_space_with_very_even_condition}.  
In Section~\ref{section:non-symmetric}, we present non-symmetric examples with $\VE$
(Proposition~\ref{prop:main_prop_for_non-symmetric}), based on Kobayashi's properness criterion (Fact~\ref{fact:properness_criterion}). 
Section~\ref{section:non-symmetric} can be read independently.

\subsection{Preliminary: Proper \texorpdfstring{$SL(2,\R)$}{SL(2,R)}-actions on symmetric spaces}
\label{section:proper-symmetric}
In this subsection, we briefly review the results of Teduka~\cite{Teduka2008PJA} and Okuda~\cite{Oku13}.
These results are obtained by combining Kobayashi's properness criterion (Fact~\ref{fact:properness_criterion})
and the Dynkin--Kostant classification for $\mathfrak{sl}(2,\C)$-triples in 
complex semisimple Lie algebras.

For this purpose, 
we begin by recalling the concepts of the weighted Dynkin diagram and the Satake diagram. 
In terms of these concepts, we will discuss interpretations of Fact~\ref{fact:properness_criterion}
for the $SL(2,\R)$-actions on semisimple symmetric spaces, 
by dividing the symmetric spaces into the following three cases: 
For the corresponding symmetric pair $(\mathfrak{g},\mathfrak{h})$,   
\begin{enumerate}[label=$(\alph*)$]
    \item \label{item:cpx_pair} 
    $\mathfrak{g}$ is complex simple and $\mathfrak{h}$ is complex (Fact \ref{fact:teduka});
    \item \label{item:cpx_real}
    $\mathfrak{g}$ is complex simple and $\mathfrak{h}$ is a real form of $\mathfrak{g}$ (Fact~\ref{fact:okuda-cpx/real});
    \item \label{item:abs_simple} 
    $\mathfrak{g}$ is absolutely simple (Fact \ref{fact:okuda}).
\end{enumerate}

\vspace{\baselineskip}

\noindent \textbf{Case \ref{item:cpx_pair} and 
weighted Dynkin diagrams.}
Let us first recall the notion of weighted Dynkin diagrams for a complex semisimple Lie algebra
$\mathfrak{g}_{\C}$.
We fix a Cartan subalgebra $\mathfrak{j}_{\C}$ of $\mathfrak{g}_{\C}$, and 
denote by $\Delta(\mathfrak{g}_{\C},\mathfrak{j}_{\C})$
the root system of $(\mathfrak{g}_{\C},\mathfrak{j}_{\C})$.
We then define a real form $\mathfrak{j}$ of $\mathfrak{j}_{\C}$ by 
\[
\mathfrak{j}:= \{A \in \mathfrak{j}_{\C} \mid \alpha(A) \in \R \text{ for all } \alpha \in \Delta(\mathfrak{g}_{\C}, \mathfrak{j}_{\C})\}.
\]
We fix a fundamental system $\Pi$ of $\Delta(\mathfrak{g}_{\C}, \mathfrak{j}_{\C})$,
and denote by $\mathfrak{j}_{+}$
the closed positive Weyl chamber in $\mathfrak{j}$. 
Consider the bijective $\R$-linear map
\[
\Psi\colon \mathfrak{j} \rightarrow \mathrm{Map}(\Pi, \R),\quad A \mapsto (\Psi_A \colon \alpha \mapsto \alpha(A)).
\]
We call $\Psi_{A}$ the \emph{weighted Dynkin diagram} corresponding to $A \in \mathfrak{j}$.

Next, we recall how to classify 
homomorphisms $\mathfrak{sl}(2, \C) \to \mathfrak{g}_{\C}$ in terms of the weighted Dynkin diagrams (see, e.g.,~\cite[Chap.~3]{CoMc93} for details).
Throughout this section, we use the notation 
\[
A_{0}:= \begin{pmatrix} 1 & 0 \\ 0 & -1 \end{pmatrix}
\in \mathfrak{sl}(2, \R).
\]

Let $\varphi\colon \mathfrak{sl}(2, \C) \to \mathfrak{g}_{\C}$ be a homomorphism. 
Then there exists an 
inner automorphism
$\tau$ of $\mathfrak{g}_{\C}$ such that 
$\tau(\varphi(A_{0})) \in \mathfrak{j}_{+}$. 
The element $\tau(\varphi(A_{0}))$ in $\mathfrak{j}_{+}$ is independent of the choice of $\tau$.
Thus we obtain a weighted Dynkin diagram by 
\[
F_{\varphi} := \Psi_{\tau(\varphi(A_{0}))}
\in \mathrm{Map}(\Pi, \R).
\]
It is known that $F_{\varphi}(\alpha)\in \{0,1,2\}$ for all $\alpha\in \Pi$, 
and that if $F_{\varphi_1}=F_{\varphi_2}$, then $\varphi_1$ and $\varphi_2$ are conjugate by an inner automorphism of 
$\mathfrak{g}_\C$.

The following fact describes how to determine the properness of the
$SL(2,\R)$-action on complex symmetric spaces $G_{\C}/H_{\C}$ 
in terms of the  weighted Dynkin diagram $F_{\varphi}$.

    \begin{fact}
    [{Teduka \cite[Thm.~5.1]{Teduka2008PJA}}]
    \label{fact:teduka}
    Let $G_{\C}/H_{\C}$ be a complex semisimple symmetric space with $G_{\C}$ connected, 
    $\sigma$ the corresponding involution of $\mathfrak{g}_{\C}$, 
    and $\bar{\sigma}$ the automorphism of the Dynkin diagram of $\mathfrak{g}_{\C}$ induced by $\sigma$.
    Then the following claims are equivalent for a non-trivial complex Lie algebra homomorphism
    $\varphi\colon \mathfrak{sl}(2,\C)\rightarrow \mathfrak{g}_{\C}$: 
    \begin{enumerate}
    \item 
    The action of $SL(2,\R)$ on $G_{\C}/H_{\C}$ via the lift
    $\tilde{\varphi}\colon SL(2,\C)\rightarrow G_{\C}$ is proper.
    \item The weighted Dynkin diagram $F_{\varphi}$
    is not $\bar{\sigma}$-stable. Namely, $F_{\varphi}\circ \bar{\sigma}\neq F_{\varphi}$. 
    \end{enumerate}
    \end{fact}

Here we recall the classification of Lie algebra homomorphisms 
from $\mathfrak{sl}(2,\C)$ into a complex classical Lie algebra 
$\mathfrak{g}_{\C}$, together with the description of the associated 
weighted Dynkin diagrams.  
We associate to a partition 
$[d_{1}^{r_{1}},\ldots,d_{k}^{r_{k}}]$ 
(with $d_{1}>\cdots> d_{k}\geq 1$ and $r_{1},\ldots,r_{k}\geq 1$) 
the direct sum of $d_i$-dimensional complex irreducible representations with multiplicities $r_i$ for $i=1,\ldots,k$.
The following fact summarizes the relevant parts of 
\cite[Lem.~3.6.4 and Chap.~5]{CoMc93} in a suitable form for our argument:
\begin{fact}
\label{fact:weighted-Dynkin-diagram-classical-type}
Let $\iota\colon \mathfrak{g}_\C\rightarrow \mathfrak{gl}(N,\C)$ 
be the standard representation of a complex classical Lie algebra $\mathfrak{g}_\C$, 
and let $\tau\colon \mathfrak{sl}(2,\C)\to \mathfrak{gl}(N,\C)$
be the representation
corresponding to the partition 
$[d_1^{r_1},\ldots,d_k^{r_k}]$.

In what follows, we characterize, in terms of the partition,
when $\tau(\mathfrak{sl}(2,\C))$ is contained in $\mathfrak{g}_\C$
up to $GL(N,\C)$-conjugacy.
Moreover, all the homomorphisms  
$\varphi\colon \mathfrak{sl}(2,\C)\to\mathfrak{g}_{\C}$ 
for which the $\mathfrak{sl}(2,\C)$-representation $\iota\circ \varphi$
is equivalent to $\tau$
are classified up to inner automorphisms by describing the weighted Dynkin diagrams $F_{\varphi}$ of such $\varphi$
in terms of the eigenvalues $h_1\ge h_2\ge\cdots\ge h_N$ of $\tau(A_{0})$.
     
\begin{itemize}
    \item 
    $\mathfrak{g}_{\C}=\mathfrak{sl}(N,\C)$:  
    Always $\tau(\mathfrak{sl}(2,\C))\subset\mathfrak{g}_{\C}$, and
    \[
    F_{\varphi}=\begin{xy}
    (0,0)="1"*{\circ}, (0,3)*{h_{1}-h_{2}},
    (20,0)="2'"*{\cdots}, 
    (40,0)="3"*{\circ}, (40,3)*{h_{N-2}-h_{N-1}},
    (65,0)="4"*{\circ}, (65,3)*{h_{N-1}-h_{N}}
    \ar@{-} "1"+(1,0);"2'"+(-3,0)
    \ar@{-} "2'"+(3,0);"3"+(-1,0)
    \ar@{-} "3"+(1,0);"4"+(-1,0)
    \end{xy}.
    \]
    \item
    $\mathfrak{g}_{\C}=\mathfrak{so}(2M+1,\C)$ ($N=2M+1$):  
    $\tau(\mathfrak{sl}(2,\C))\subset\mathfrak{g}_{\C}$ up to conjugacy
   if and only if $r_i$ is even whenever $d_i$ is even, and 
    \[
    F_{\varphi}=\begin{xy}
    (0,0)="1"*{\circ}, (0,3)*{h_{1}-h_{2}}, 
    (20,0)="2"*{\cdots},
    (40,0)="3"*{\circ}, (40,3)*{h_{M-1}-h_{M}},
    (60,0)="4"*{\circ}, (60,3)*{h_{M}},
    \ar@{-} "1"+(1,0);"2"+(-3,0)
    \ar@{-} "2"+(3,0);"3"+(-1,0)
    \ar@{=>} "3"+(1,0);"4"+(-1,0)
    \end{xy}.
    \]

    \item
    $\mathfrak{g}_{\C}=\mathfrak{sp}(M,\C)$ ($N=2M$):  
    $\tau(\mathfrak{sl}(2,\C))\subset\mathfrak{g}_{\C}$  up to conjugacy
   if and only if $r_i$ is even whenever $d_i$ is odd, and 
    \[
    F_{\varphi}=\begin{xy}
    (0,0)="1"*{\circ}, (0,3)*{h_{1}-h_{2}}, 
    (20,0)="2"*{\cdots},
    (40,0)="3"*{\circ}, (40,3)*{h_{M-1}-h_{M}},
    (60,0)="4"*{\circ}, (60,3)*{2h_{M}},
    \ar@{-} "1"+(1,0);"2"+(-3,0)
    \ar@{-} "2"+(3,0);"3"+(-1,0)
    \ar@{<=} "3"+(1,0);"4"+(-1,0)
    \end{xy}.
    \]

    \item
    $\mathfrak{g}_{\C}=\mathfrak{so}(2M,\C)$ ($N=2M$):  
    $\tau(\mathfrak{sl}(2,\C))\subset\mathfrak{g}_{\C}$ up to conjugacy
    if and only if $r_i$ is even whenever $d_i$ is even.  
    If some $d_{i}$ is odd, we have
    \[
    F_{\varphi}=\begin{xy}
    (12.5,0)="3"*{\circ}, (12.5,4)*{h_{1}-h_2}, 
    (32.5,0)="4"*{\cdots},  
    (52.5,0)="5"*{\circ},
    (50,4)*{h_{M-2}-h_{M-1}},
    (72.5,3)="6"*{\circ}, (72.5,7)*{h_{M-1}}, 
    (72.5,-3)="7"*{\circ}, (72.5,-7)*{h_{M-1}}
    \ar@{-} "3"+(0.8,0);"4"+(-3,0)
    \ar@{-} "4"+(3,0);"5"+(-0.8,0)
    \ar@{-} "5"+(0.5,0.5);"6"+(-0.5,-0.5)
    \ar@{-} "5"+(0.5,-0.5);"7"+(-0.5,0.5)
    \end{xy},
    \]
    where we note that $h_{M}=0$. 
    
    If all $d_{i}$ are even (the very even case), 
    there exist exactly two different homomorphisms 
    $\varphi\colon \mathfrak{sl}(2,\C)\to\mathfrak{g}_{\C}$, 
    up to inner automorphisms of $\mathfrak{g}_{\C}$, whose diagrams are
    \[
    F_{\varphi}=
    \begin{xy}
    (12.5,0)="3"*{\circ}, (12.5,4)*{h_{1}-h_2}, 
    (32.5,0)="4"*{\cdots},  
    (52.5,0)="5"*{\circ},
    (50,4)*{h_{M-2}-h_{M-1}},
    (72.5,3)="6"*{\circ}, (72.5,7)*{a}, 
    (72.5,-3)="7"*{\circ}, (72.5,-7)*{b}
    \ar@{-} "3"+(0.8,0);"4"+(-3,0)
    \ar@{-} "4"+(3,0);"5"+(-0.8,0)
    \ar@{-} "5"+(0.5,0.5);"6"+(-0.5,-0.5)
    \ar@{-} "5"+(0.5,-0.5);"7"+(-0.5,0.5)
    \end{xy}\ ,
    \]
    where $(a,b)=(2,0),(0,2)$. 
\end{itemize}
\end{fact}

\begin{remark}\label{fact:invariance_Titz_inv}
When $\mathfrak{g}_\C=\mathfrak{sl}(N,\C)$, 
the weighted Dynkin diagram $F_{\varphi}$
is invariant under the non-trivial graph automorphism. 
\end{remark}
    
\vspace{\baselineskip}
\noindent \textbf{Case \ref{item:cpx_real} and Satake diagrams.}
\label{section:satake}
Let $\mathfrak{a}$ be a maximal split abelian subalgebra of a real semisimple Lie algebra $\mathfrak{g}$. In what follows, we may and do assume (see, e.g., \cite[p.\ 23--24]{Warner_harmonic_analysis_I}) that a Cartan subalgebra $\mathfrak{j}_{\C}$ 
    of $\mathfrak{g}_{\C}:=\mathfrak{g}\otimes_{\R}\C$ contains $\mathfrak{a}$, 
    and that the restriction $\overline{\Pi}:=\Pi|_{\mathfrak{a}}\smallsetminus\{0\}$
    of a fundamental system $\Pi$ of the root system $\Delta(\mathfrak{g}_{\C},\mathfrak{j}_{\C})$
    is a fundamental system of the restricted root system of $(\mathfrak{g},\mathfrak{a})$.
We write $\mathfrak{a}_{+}$ for the closed positive Weyl chamber with respect to the fundamental system $\overline{\Pi}$.
Obviously, we have $\mathfrak{a}_{+} \subset \mathfrak{j}_{+}$.

Now we briefly recall the Satake diagram of 
the real semisimple Lie algebra $\mathfrak{g}$
(see Araki~\cite{Araki62} for details).
We denote by $\Pi_{0}$ the set of simple roots in $\Pi$
which vanish on $\mathfrak{a}$.
The \emph{Satake diagram} $S_{\mathfrak{g}}$ of $\mathfrak{g}$
consists of the following data:
the Dynkin diagram of $\mathfrak{g}_{\C}$ with nodes $\Pi$, black nodes $\Pi_{0}$, and arrows joining $\alpha, \beta \in \Pi\smallsetminus\Pi_{0}$ 
such that $\alpha|_\mathfrak{a}=\beta|_{\mathfrak{a}}$.

Okuda introduced a combinatorial condition between 
a weighted Dynkin diagram in $\mathfrak{g}_{\C}$ and the Satake diagram
of $\mathfrak{g}$:

\begin{definition}[{Okuda \cite[Def.~7.3]{Oku13}}]
    \label{definition:match}
    A weighted Dynkin diagram 
    \emph{matches} the Satake diagram $S_{\mathfrak{g}}$ of $\mathfrak{g}$ if all the weights on black nodes in $S_{\mathfrak{g}}$ are zero and 
    any pair of nodes joined by an arrow in $S_{\mathfrak{g}}$ has the same weights. 
\end{definition}

Then the following lemma holds:
\begin{lemma}[{Okuda \cite[Lem.~7.5]{Oku13}}]
    \label{lemma:okuda_match}
    For $A\in \mathfrak{j}$, 
    the weighted Dynkin diagram $\Psi_{A}$
    matches the Satake diagram of $\mathfrak{g}$
    if and only if $A\in \mathfrak{a}$.
\end{lemma}

The following fact describes how to determine the properness of the
$SL(2,\R)$-action on 
symmetric spaces
$G_{\C}/G$ in Case~\ref{item:cpx_real}, which is an immediate corollary of Okuda's theorem (Fact~\ref{fact:okuda}):

    \begin{fact}
    \label{fact:okuda-cpx/real}
    Let $G_{\C}$ be a connected complex semisimple Lie group and $G$ its closed Lie subgroup whose 
    Lie algebra $\mathfrak{g}$ is a real form of $\mathfrak{g}_{\C}$. Then the following claims are equivalent for a non-trivial complex Lie algebra homomorphism
    $\varphi\colon \mathfrak{sl}(2,\C)\rightarrow \mathfrak{g}_{\C}$: 
    \begin{enumerate}
    \item 
    \label{fact:okuda-cpx/real_item:proper}
    The action of $SL(2,\R)$ on $G_{\C}/G$ via the lift of $\varphi$ is proper.
    \item
    \label{fact:okuda-cpx/real_item:match}
    The weighted Dynkin diagram $F_{\varphi}$
    does not match the Satake diagram of $\mathfrak{g}$.
    \end{enumerate}
    \end{fact}   

\vspace{\baselineskip}

\noindent \textbf{Case \ref{item:abs_simple} and 
a pair of Satake diagrams 
associated to a semisimple symmetric pair.} \label{subsubsec:properness_caseC}
    Let $(\mathfrak{g},\mathfrak{h})$ be a semisimple symmetric pair, and $\sigma$ be the corresponding involution of $\mathfrak{g}$. 
    Then we put $\mathfrak{q}:=\mathfrak{g}^{-\sigma}
    =\{X\in \mathfrak{g} \mid \sigma(X) = -X\}$, and 
    define the \emph{$c$-dual} 
    \begin{equation}
        \label{eq:c-dual}
        \mathfrak{g}^{c}:= \mathfrak{h}+\sqrt{-1}\mathfrak{q}\subset \mathfrak{g}_{\C}.
    \end{equation}
    The Lie algebra $\mathfrak{g}^{c}$
    is also a real form of $\mathfrak{g}_{\C}$.
    Given a symmetric pair $(\mathfrak{g},\mathfrak{h})$ with $\mathfrak{g}$ simple, 
    the $c$-dual $\mathfrak{g}^{c}$ is 
    described in \cite[Sect.~1, (1,12)--(1.16)]{OshimaSekiguchi84}.
    We note that 
    the pair $(\mathfrak{g}^{c},\mathfrak{h})$ coincides 
    with the symmetric pair $(\mathfrak{g},\mathfrak{h})^{ada}$ 
    in the notation of \cite[Sect.~1]{OshimaSekiguchi84}.
    
    One can define \emph{simultaneously} the Satake diagrams 
    of $\mathfrak{g}$ and $\mathfrak{g}^{c}$ by applying the argument in Oshima--Sekiguchi~\cite[Sect.~2 and 3]{OshimaSekiguchi84} to the associated pair $(\mathfrak{g},\mathfrak{h})^a$. Take a Cartan involution $\theta$ of $\mathfrak{g}$
    with $\theta\sigma=\sigma\theta$, and consider the anti-linear extension $\bar{\theta}$ of $\theta$ 
    to $\mathfrak{g}_{\C}$. 
    We write 
    $\mathfrak{g}=\mathfrak{k}+\mathfrak{p}$, $\mathfrak{g}^{c}=\mathfrak{k}^{c}+\mathfrak{p}^{c}$, and 
$\mathfrak{h}=\mathfrak{k}_{\mathfrak{h}}+\mathfrak{p}_{\mathfrak{h}}$
for the decompositions with respect to the Cartan involutions $\theta$, $\bar{\theta}|_{\mathfrak{g}^{c}}$, and $\theta|_{\mathfrak{h}}$, respectively. 
    Take a maximal abelian subspace $\mathfrak{a}_{\mathfrak{h}}$ of $\mathfrak{p}_{\mathfrak{h}}$, 
    and extend $\mathfrak{a}_{\mathfrak{h}}$ to maximal abelian subspaces 
    $\mathfrak{a}$ and $\mathfrak{a}^{c}$ 
    of $\mathfrak{p}$ and $\mathfrak{p}^{c}$, respectively.
    Then we have $[\mathfrak{a},\mathfrak{a}^{c}] = 0$ (\cite[Lem.~2.4~(i)]{OshimaSekiguchi84}). 
    Hence, one can take 
    a Cartan subalgebra $\mathfrak{j}_{\C}$ of $\mathfrak{g}_{\C}$ containing both $\mathfrak{a}$ and $\mathfrak{a}^{c}$.
    Further, by the argument in \cite[Sect.~3]{OshimaSekiguchi84}, we may and do take a fundamental system $\Pi$ of the root system of $(\mathfrak{g}_{\C},\mathfrak{j}_{\C})$ 
    such that the restrictions 
    $\Pi|_{\mathfrak{a}}\smallsetminus\{0\}$ and 
    $\Pi|_{\mathfrak{a}^{c}}\smallsetminus\{0\}$
    are fundamental systems of the restricted root systems of $(\mathfrak{g},\mathfrak{a})$ and $(\mathfrak{g}^c,\mathfrak{a}^c)$, respectively.
    Thus the Satake diagrams $S_{\mathfrak{g}}$ and $S_{\mathfrak{g}^{c}}$ are simultaneously defined. 

    The following fact characterizes when a complex homomorphism
    $\mathfrak{sl}(2,\C)\to\mathfrak{g}_{\C}$ descends to a real homomorphism
    $\mathfrak{sl}(2,\R)\to\mathfrak{g}$, and 
    yields a criterion for the properness of $SL(2,\R)$-actions on symmetric spaces $G/H$: 
    \begin{fact}
    \label{fact:okuda}
    Let 
    $\varphi\colon \mathfrak{sl}(2,\C)\rightarrow \mathfrak{g}_{\C}$
    be a non-trivial homomorphism. 
    \begin{enumerate}[label=$(\roman*)$]
        \item $($\cite[Thm.~7.10]{Oku13}$)$
        \label{fact:okuda:item:real}
        We have $\varphi(\mathfrak{sl}(2,\R))\subset \mathfrak{g}$ 
        up to inner automorphisms of $\mathfrak{g}_{\C}$
        if and only if 
        the weighted Dynkin diagram $F_{\varphi}$ 
        matches the Satake diagram $S_{\mathfrak{g}}$.

        \item $($\cite[Thm.~10.1]{Oku13}$)$
        \label{fact:okuda:item:proper}
        Let $G/H$ be a semisimple symmetric space with Setting~\ref{setting:semisimple}. 
        We assume $\varphi(\mathfrak{sl}(2,\R))\subset \mathfrak{g}$. 
        Then the $SL(2,\R)$-action on $G/H$ via the lift of $\varphi|_{\mathfrak{sl}(2,\R)}$
        is proper if and only if $F_{\varphi}$ 
        does not match
        the Satake diagram $S_{\mathfrak{g}^{c}}$.
    \end{enumerate}
    \end{fact}
    
    \begin{remark}    
    In the above fact, it is not necessary to assume that $\mathfrak{g}$ is absolutely simple. Fact~\ref{fact:teduka} (Case~\ref{item:cpx_pair}) and
Fact~\ref{fact:okuda-cpx/real} (Case~\ref{item:cpx_real}) both follow from
Fact~\ref{fact:okuda}.
However, the forms of Facts~\ref{fact:teduka} and~\ref{fact:okuda-cpx/real}
are more convenient for our purposes, and hence we have separated the
discussion into Cases~\ref{item:cpx_pair}, \ref{item:cpx_real}, and~\ref{item:abs_simple}.
    \end{remark}

    For most symmetric pairs $(\mathfrak{g},\mathfrak{h})$ with $\mathfrak{g}$ simple,
the pair $(S_{\mathfrak{g}}, S_{\mathfrak{g}^{c}})$ is  determined
only by the list of the $c$-dual $\mathfrak{g}^{c}$
given in \cite[Sect.~1, (1.12)--(1.16)]{OshimaSekiguchi84}.
However, as illustrated by the example below, this is not always the case.
Even in such situations, one can easily compute the pair
$(S_{\mathfrak{g}}, S_{\mathfrak{g}^{c}})$ by the following lemma.
The authors are grateful to T.~Okuda for explaining this useful lemma. 

    \begin{lemma}
    \label{lemma:okuda_match_a_ac}
    The dimension of the real vector space of
    weighted Dynkin diagrams which 
    match both the Satake diagrams $S_{\mathfrak{g}}$
    and $S_{\mathfrak{g}^{c}}$ is equal to 
    the real rank of $\mathfrak{h}$.
    \end{lemma}

    \begin{proof}
    By Lemma~\ref{lemma:okuda_match}, 
    the real vector space of such 
    weighted Dynkin diagrams is isomorphic to
    $\mathfrak{a}\cap \mathfrak{a}^{c}=\mathfrak{a}_{\mathfrak{h}}$.
    This proves the assertion.
    \end{proof}

 For example, let $(\mathfrak{g},\mathfrak{h})=(\mathfrak{so}^*(4N), \mathfrak{so}^*(4N-2p)\oplus \mathfrak{so}^*(2p))$ ($N\geq 2$, $1\leq p\leq N$). Then the $c$-dual
 $\mathfrak{g}^c$ is isomorphic to $\mathfrak{so}^*(4N)$. 
 By Lemma~\ref{lemma:okuda_match_a_ac}, 
 we have
\begin{equation}
\label{eq:pair-satake-o*}
(S_\mathfrak{g}, S_{\mathfrak{g}^c})=
\begin{cases}
\left(
\begin{xy}
(0,0)="1"*{\bullet}, 
(5,0)="2"*{\circ},
(12.5,0)="3"*{\cdots},
(20,0)="4"*{\bullet},  
(25,0)="5"*{\circ}, 
(28,3)="6"*{\circ}, 
(28,-3)="7"*{\bullet}, 
\ar@{-} "1"+(1,0);"2"+(-1,0)
\ar@{-} "2"+(1,0);"3"+(-3,0)
\ar@{-} "3"+(3,0);"4"+(-1,0)
\ar@{-} "4"+(1,0);"5"+(-1,0)
\ar@{-} "5"+(0.5,0.5);"6"+(-0.5,-0.5)
\ar@{-} "5"+(0.5,-0.5);"7"+(-0.5,0.5)
\end{xy}, 
\begin{xy}
(0,0)="1"*{\bullet}, 
(5,0)="2"*{\circ},
(12.5,0)="3"*{\cdots},
(20,0)="4"*{\bullet},  
(25,0)="5"*{\circ}, 
(28,3)="6"*{\circ}, 
(28,-3)="7"*{\bullet}, 
\ar@{-} "1"+(1,0);"2"+(-1,0)
\ar@{-} "2"+(1,0);"3"+(-3,0)
\ar@{-} "3"+(3,0);"4"+(-1,0)
\ar@{-} "4"+(1,0);"5"+(-1,0)
\ar@{-} "5"+(0.5,0.5);"6"+(-0.5,-0.5)
\ar@{-} "5"+(0.5,-0.5);"7"+(-0.5,0.5)
\end{xy}\right) & \text{($p$: even)}, \\[2ex]
\left(
\begin{xy}
(0,0)="1"*{\bullet}, 
(5,0)="2"*{\circ},
(12.5,0)="3"*{\cdots},
(20,0)="4"*{\bullet},  
(25,0)="5"*{\circ}, 
(28,3)="6"*{\circ}, 
(28,-3)="7"*{\bullet}, 
\ar@{-} "1"+(1,0);"2"+(-1,0)
\ar@{-} "2"+(1,0);"3"+(-3,0)
\ar@{-} "3"+(3,0);"4"+(-1,0)
\ar@{-} "4"+(1,0);"5"+(-1,0)
\ar@{-} "5"+(0.5,0.5);"6"+(-0.5,-0.5)
\ar@{-} "5"+(0.5,-0.5);"7"+(-0.5,0.5)
\end{xy}, 
\begin{xy}
(0,0)="1"*{\bullet}, 
(5,0)="2"*{\circ},
(12.5,0)="3"*{\cdots},
(20,0)="4"*{\bullet},  
(25,0)="5"*{\circ}, 
(28,3)="6"*{\bullet}, 
(28,-3)="7"*{\circ}, 
\ar@{-} "1"+(1,0);"2"+(-1,0)
\ar@{-} "2"+(1,0);"3"+(-3,0)
\ar@{-} "3"+(3,0);"4"+(-1,0)
\ar@{-} "4"+(1,0);"5"+(-1,0)
\ar@{-} "5"+(0.5,0.5);"6"+(-0.5,-0.5)
\ar@{-} "5"+(0.5,-0.5);"7"+(-0.5,0.5)
\end{xy}\right) & \text{($p$: odd)}.
\end{cases}
\end{equation}

\subsection{Characterization of very even homomorphisms in terms of the weighted Dynkin diagrams}
\label{section:characterization-very-even}
In this subsection, we 
characterize very even $\mathfrak{sl}(2,\R)$-homomorphisms in terms of the associated weighted Dynkin diagrams (Proposition~\ref{prop:characterization_of_no_odd_dimensional}). 
For this purpose, it is convenient to introduce, in the same manner as 
Definitions~\ref{definition:classical-pair} and~\ref{def:iota-very-even} for real homomorphisms,
the notions of a classical pair and of being very even for complex homomorphisms:
\begin{definition}
    Let $(\mathfrak{g}_{i,\C}, \iota_i)$ $(i=1,2)$ be pairs of complex Lie algebras 
    and their complex representations. 
    We say that the pairs $(\mathfrak{g}_{i,\C}, \iota_i)$ are 
    \emph{equivalent} 
    if there exists a complex Lie algebra isomorphism 
    $\varphi \colon \mathfrak{g}_{1,\C}\rightarrow \mathfrak{g}_{2,\C}$ such that 
    $\iota_1$ is equivalent to $\iota_2\circ \varphi$ 
    as a representation of $\mathfrak{g}_{1,\C}$. 

    A pair of a complex Lie algebra $\mathfrak{g}_{\C}$
    and its complex representation $\iota$
    is said to be \emph{complex classical} 
    if the pair $(\mathfrak{g}_{\C},\iota)$ is equivalent 
    to a pair of a classical  complex Lie algebra and its standard representation.

    For a complex classical pair
    $(\mathfrak{g}_{\C},\iota)$, 
    we say that a complex Lie algebra homomorphism 
    $\varphi \colon \mathfrak{sl}(2,\C) \rightarrow \mathfrak{g}_{\C}$ 
    is \emph{very even for $\iota$}
    if the $\mathfrak{sl}(2,\C)$-representation $\iota \circ \varphi$ 
    decomposes as a direct sum of irreducible representations of even dimension.
\end{definition}

The following obvious lemma ensures that the property of an $\mathfrak{sl}(2,\R)$-homomorphism being very even can be verified by its 
$\C$-linear extension.
\begin{lemma}
    \label{lemma:very-even-complex-real}
    \begin{enumerate}[label=$(\roman*)$]
        \item
        \label{lemma:very-even-complex-real:item-(b)}
        For a complex classical pair $(\mathfrak{g}_{\C},\iota)$,  
        a homomorphism
        $\varphi\colon \mathfrak{sl}(2,\R)\rightarrow \mathfrak{g}_{\C}$
        is very even for $\iota$ if and only if
        its $\C$-linear extension 
        $\varphi_{\C}\colon \mathfrak{sl}(2,\C)\rightarrow \mathfrak{g}_{\C}$
        is very even for $\iota$.

        \item
        \label{lemma:very-even-complex-real:item-(c)}
        For a classical pair $(\mathfrak{g},\iota)$ with $\mathfrak{g}$ absolutely simple,  
        a homomorphism
        $\varphi\colon \mathfrak{sl}(2,\R)\rightarrow \mathfrak{g}$
        is very even for $\iota\colon \mathfrak{g}\rightarrow \mathfrak{gl}(N,\C)$ if and only if
        the $\C$-linear extension 
        $\varphi_{\C}\colon \mathfrak{sl}(2,\C)\rightarrow \mathfrak{g}_{\C}$
        is very even for the $\C$-linear extension $\iota_{\C}\colon\mathfrak{g}_{\C}\rightarrow \mathfrak{gl}(N,\C)$.
    \end{enumerate}
\end{lemma}

For $\mathfrak{sl}(2,\C)$-homomorphisms, the property of being \emph{very even} 
can be checked in terms of weighted Dynkin diagrams as follows:

\begin{proposition}\label{prop:characterization_of_no_odd_dimensional}
Let $\mathfrak{g}_\C$ be 
either $\mathfrak{sl}(2N,\C)$, $\mathfrak{sp}(N,\C)$, or $\mathfrak{so}(2N,\C)$, 
and 
$\iota$ the standard representation of $\mathfrak{g}_\C$. 
Then  a complex Lie algebra homomorphism
$\varphi\colon \mathfrak{sl}(2,\C)\to \mathfrak{g}_{\C}$
is very even for $\iota$
if and only if the weights $a_{i}$ of 
    the weighted Dynkin diagram $F_{\varphi}$ satisfy 
    \begin{equation}
    \label{eq:very-even-weight}
    \begin{cases}
    a_N\neq 0 & \text{if } \mathfrak{g}_\C=\mathfrak{sl}(2N,\C) \text{ or } \mathfrak{sp}(N,\C), \\
    a_{N-1}\neq a_{N} & \text{if } \mathfrak{g}_\C = \mathfrak{so}(2N,\C). 
    \end{cases}
    \end{equation}
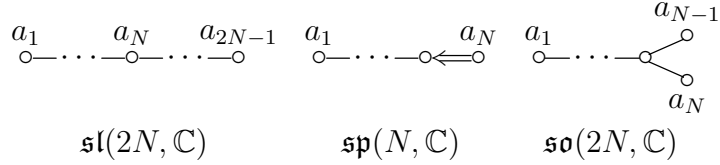
    \[
\renewcommand{\arraystretch}{1.5}
\begin{array}{ccc}
\begin{xy}
(0,0)="1"*{\circ}, (0,3)*{a_1},
(7,0)="2"*{\cdots},
(14,0)="3"*{\circ}, (14,3)*{a_N}, 
(21,0)="4"*{\cdots},
(28,0)="5"*{\circ}, (28,3)*{a_{2N-1}}
\ar@{-} "1"+(1,0);"2"+(-3,0)
\ar@{-} "2"+(3,0);"3"+(-1,0)
\ar@{-} "3"+(1,0);"4"+(-3,0)
\ar@{-} "4"+(3,0);"5"+(-1,0)
\end{xy}
     & 
\begin{xy}
(0,0)="1"*{\circ}, (0,3)*{a_1}, 
(7,0)="2"*{\cdots},
(14,0)="3"*{\circ}, 
(21,0)="4"*{\circ}, (21,3)*{a_{N}}
\ar@{-} "1"+(1,0);"2"+(-3,0)
\ar@{-} "2"+(3,0);"3"+(-1,0)
\ar@{<=} "3"+(1,0);"4"+(-1,0)
\end{xy}
     & 
\begin{xy}
(12.5,0)="3"*{\circ}, (12.5,3)*{a_{1}}, 
(19.5,0)="4"*{\cdots},  
(26.5,0)="5"*{\circ},
(32,3)="6"*{\circ}, (32,6)*{a_{N-1}}, 
(32,-3)="7"*{\circ}, (32,-6)*{a_{N}}
\ar@{-} "3"+(0.8,0);"4"+(-3,0)
\ar@{-} "4"+(3,0);"5"+(-0.8,0)
\ar@{-} "5"+(0.5,0.5);"6"+(-0.5,-0.5)
\ar@{-} "5"+(0.5,-0.5);"7"+(-0.5,0.5)
\end{xy}\\
    \mathfrak{sl}(2N,\C) & \mathfrak{sp}(N,\C) & \mathfrak{so}(2N,\C)\ \ 
\end{array}
\renewcommand{\arraystretch}{1}
\]
\end{proposition}

    In Proposition~\ref{prop:characterization_of_no_odd_dimensional},
we excluded the cases $\mathfrak{g}_{\C}=\mathfrak{sl}(2N+1,\C)$ and
$\mathfrak{so}(2N+1,\C)$, because these classical Lie algebras obviously admit no 
very even homomorphisms for their standard representations.

\begin{proof}
    We apply Fact~\ref{fact:weighted-Dynkin-diagram-classical-type} to prove the assertion.
    Assume that the representation $\iota\circ\varphi$
    corresponds to the partition$[d_{1}^{r_{1}},\ldots,d_{k}^{r_{k}}]$, and 
let 
$h_{1}\geq h_{2}\geq \cdots$
denote the eigenvalues of $\iota(\varphi(A_{0}))$.
By definition, we note that $\varphi$ is very even for $\iota$ if and only if
all of $d_{1},\ldots,d_{k}$ are even.

Let $\mathfrak{g}_{\C}=\mathfrak{sl}(2N,\C)$.
If $\varphi$ is very even for $\iota$, then we have $(h_{N},h_{N+1})=(1,-1)$, and therefore the middle weight
of the weighted Dynkin diagram $F_{\varphi}$ is equal to $2$. In particular,
it is non-zero.
Conversely, if $\varphi$ is not very even for $\iota$, then among
$d_{1},\ldots,d_{k}$ there are at least two odd integers counted with
multiplicity.
In this case, we have $(h_{N},h_{N+1})=(0,0)$, and hence the middle weight of
$F_{\varphi}$ is equal to $0$.
Hence the assertion is  proved.

Let $\mathfrak{g}_{\C}=\mathfrak{sp}(N,\C)$. 
By definition, $\varphi$ is very even for $\iota$ if and only if
$h_{N}=1$.
Moreover,
this is equivalent to the condition that the rightmost weight of the
weighted Dynkin diagram $F_{\varphi}$ is non-zero.

Let $\mathfrak{g}_{\C}=\mathfrak{so}(2N,\C)$. 
The assertion
follows immediately from Fact~\ref{fact:weighted-Dynkin-diagram-classical-type}.
\end{proof}

In the proof of 
Theorem~\ref{theorem:classification_of_symmetric_space_with_very_even_condition}, 
we also make use of the following lemma to show that several symmetric spaces 
$G/H$ admitting proper $SL(2,\R)$-actions in Cases~\ref{item:classification_caseB} and~\ref{item:classification_caseC}
do not have $\VE$.
\begin{lemma}
    \label{lemma:non-very-even}
    Let $(\mathfrak{g}_{\C},\iota)$ be 
    as in Proposition~\ref{prop:characterization_of_no_odd_dimensional}, 
    $\varphi\colon \mathfrak{sl}(2,\C)\rightarrow \mathfrak{g}_{\C}$ a complex homomorphism, 
    and $\tau$ any representation of $\mathfrak{g}_{\C}$
    such that $(\mathfrak{g}_{\C},\tau)$ is equivalent to $(\mathfrak{g}_{\C},\iota)$.
    Then the following claims hold: 
    \begin{enumerate}
        \item \label{lemma:non-very-even_item:non_D_4}
        In the case $\mathfrak{g}_{\C}\neq \mathfrak{so}(8,\C)$,
        $\varphi$ is very even for 
        $\tau$ if and only if 
        it is very even for 
        $\iota$.

        \item \label{lemma:non-very-even_item:D_4}
        In the case 
        $\mathfrak{g}_{\C}=\mathfrak{so}(8,\C)$,
        $\varphi$ is not very even for $\tau$
        if $a_{1}=a_{3}=a_{4}$, where 
        $a_{1},a_{2},a_{3},a_{4}$ are the weights of $F_{\varphi}$.
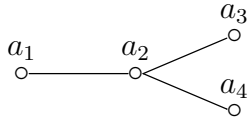
        \[
        \begin{xy}
        (30,0)="3"*{\circ},(30,3)*{a_1},
        (45,0)="4"*{\circ},"4"+(0,3)*{a_{2}},
        (58,5)="5"*{\circ},"5"+(0,3)*{a_{3}},
        (58,-5)="5'"*{\circ},"5'"+(0,3)*{a_{4}},
        \ar@{-} "3"+(1,0);"4"+(-1,0)
        \ar@{-} "4"+(1,0);"5"+(-1,-0.3)
        \ar@{-} "4"+(1,0);"5'"+(-1,0.3)
        \end{xy}
        \]
    \end{enumerate}
\end{lemma}

\begin{proof}
By our assumption, there exists a complex Lie algebra automorphism $\alpha\colon \mathfrak{g}_{\C}\rightarrow \mathfrak{g}_{\C}$
such that the representation $\iota\circ \alpha$ of $\mathfrak{g}_{\C}$ is 
equivalent to $\tau$.
Hence $\varphi$ is very even for $\tau$ if and only if
$\alpha\circ\varphi$ is very even for $\iota$.

Here we note that the weighted Dynkin diagrams
$F_{\alpha \circ \varphi}$ and $F_{\varphi}$ are mapped to each other
by an automorphism of the Dynkin diagram of
$\mathfrak{g}_{\C}$.
If $\mathfrak{g}_{\C} \neq \mathfrak{so}(8,\C)$, then the weight condition
\eqref{eq:very-even-weight} in
Proposition~\ref{prop:characterization_of_no_odd_dimensional}
is preserved under any automorphism of the Dynkin diagram.
Therefore,
$\alpha \circ \varphi$ is very even for $\iota$
if and only if $\varphi$ is very even for $\iota$ by Proposition~\ref{prop:characterization_of_no_odd_dimensional}.
This proves \eqref{lemma:non-very-even_item:non_D_4}.

On the other hand, for \eqref{lemma:non-very-even_item:D_4},
under the assumption $a_{1}=a_{3}=a_{4}$,
the two rightmost weights of $F_{\alpha \circ \varphi}$ coincide.
Hence, by Proposition~\ref{prop:characterization_of_no_odd_dimensional},
$\alpha \circ \varphi$ is not very even for $\iota$,
which completes the proof. 
\end{proof}

\subsection{Proof of 
\texorpdfstring{Theorem~\ref{theorem:classification_of_symmetric_space_with_very_even_condition}}{Theorem A}}
\label{section:proof-theorem-classification}
We prove Theorem~\ref{theorem:classification_of_symmetric_space_with_very_even_condition} by dividing symmetric spaces $G/H$ with $\mathfrak{g}$ simple into the three cases as in Section~\ref{section:proper-symmetric}:
\begin{enumerate}[label=$(\alph*)$]
    \item \label{item:classification_caseA}(Proposition~\ref{prop:classification_case_A})
    Both $\mathfrak{g}$ and $\mathfrak{h}$ are complex,
    \item \label{item:classification_caseB} (Proposition~\ref{prop:classification_case_B})
    $\mathfrak{g}$ is complex and $\mathfrak{h}$ is a real form of $\mathfrak{g}$, 
    \item \label{item:classification_caseC}(Proposition~\ref{prop:classification_case_C})
    $\mathfrak{g}$ is absolutely simple.
\end{enumerate} 
We shall prove these propositions by using the combinatorial 
criteria for properness of $SL(2,\R)$-actions in Section~\ref{section:proper-symmetric}.  

In the following, we utilize the classification result of irreducible semisimple symmetric spaces that admit a proper $SL(2,\R)$-action.
The complex case was first classified by Teduka~\cite{Teduka2008PJA}, and the general case was subsequently established by Okuda~\cite{Oku13}.

Here we note that there are some minor misprints in Okuda's classification list (\cite[Table~6]{Oku13}). While four of these misprints have been updated in Okuda~\cite[Table~A]{Okuda17}, it should be pointed out that the symmetric pair $(\mathfrak{so}(k,k), \mathfrak{so}(k,\mathbb{C}) \oplus \mathfrak{so}(2))$ (listed as the 14th from the top in \cite[Table~A]{Okuda17}) 
should be $(\mathfrak{so}(k,k), \mathfrak{so}(k,\mathbb{C}))$.

\vspace{\baselineskip}
\noindent \textbf{Proof of Theorem~\ref{theorem:classification_of_symmetric_space_with_very_even_condition} for Case \ref{item:classification_caseA}.}\label{subsubsec:classification_case_A}
We prove the following: 
\begin{proposition}
\label{prop:classification_case_A}
\begin{enumerate}
\item 
\label{item:case(a)-VET}
For $N\geq 3$ and $0\leq p<N/2$, \[
X_{N,p}:=SO(2N,\C)/(SO(2p+1,\C)\times SO(2N-2p-1,\C)) 
\]
has $\VET$ for the standard representation of $\mathfrak{so}(2N,\C)$. 
\item 
\label{item:case(a)-VE-locally-isom}
Let $G_\C/H_\C$ be a complex symmetric space with $G_\C$ connected and simple. 
If $G_\C/H_\C$ has $\VE$ and admits a proper $SL(2,\R)$-action, then it is locally isomorphic to $X_{N,p}$. 
\end{enumerate}
\end{proposition}

\begin{proof}
\eqref{item:case(a)-VET}.
We apply Fact~\ref{fact:teduka} to $X_{N,p}$, noting 
that the involution of $\mathfrak{so}(2N,\C)$ corresponding to the symmetric space $X_{N,p}$
induces the graph automorphism of the Dynkin diagram given by
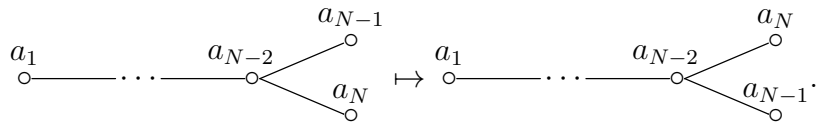
\[
\begin{xy}
(15,0)="1"*{\circ}, "1"+(0,3)*{a_1},
(30,0)="2"*{\cdots},
(45,0)="3"*{\circ}, "3"+(-1.5,3)*{a_{N-2}},
(58,5)="4"*{\circ},"4"+(0,3)*{a_{N-1}},
(58,-5)="5"*{\circ},"5"+(0,3)*{a_{N}},
\ar@{-} "1"+(1,0);"2"+(-3,0)
\ar@{-} "2"+(3,0);"3"+(-1,0)
\ar@{-} "3"+(1,0);"4"+(-1,-0.3)
\ar@{-} "3"+(1,0);"5"+(-1,0.3)
\end{xy}
\mapsto
\begin{xy}
(15,0)="1"*{\circ}, "1"+(0,3)*{a_1},
(30,0)="2"*{\cdots},
(45,0)="3"*{\circ}, "3"+(-1.5,3)*{a_{N-2}},
(58,5)="4"*{\circ},"4"+(0,3)*{a_{N}},
(58,-5)="5"*{\circ},"5"+(0,3)*{a_{N-1}},
\ar@{-} "1"+(1,0);"2"+(-3,0)
\ar@{-} "2"+(3,0);"3"+(-1,0)
\ar@{-} "3"+(1,0);"4"+(-1,-0.3)
\ar@{-} "3"+(1,0);"5"+(-1,0.3)
\end{xy}.
\]
Let $\varphi\colon \mathfrak{sl}(2,\C)\to\mathfrak{so}(2N,\C)$ be a complex Lie algebra homomorphism,
and denote by $a_{1},\ldots,a_{N}$ the weights of its weighted
Dynkin diagram $F_{\varphi}$.
By Fact~\ref{fact:teduka}, the $SL(2,\R)$-action on $X_{N,p}$ via the lift
$\tilde{\varphi}$ is proper if and only if $a_{N-1}\neq a_{N}$.
By Proposition~\ref{prop:characterization_of_no_odd_dimensional},
this is equivalent to $\varphi$ being very even for the standard representation $\iota$ of $\mathfrak{so}(2N,\C)$.
Hence $X_{N,p}$ has $\VET$ for $\iota$ since $\iota$ lifts to a Lie group representation of $SO(2N,\C)$.
This proves 
\eqref{item:case(a)-VET}.

\eqref{item:case(a)-VE-locally-isom}. 
By a result of Teduka~\cite[Thm.~1.4]{Teduka2008PJA} (or Okuda~\cite[Thm.~1.4]{Oku13}),
$G_{\C}/H_{\C}$ admits a proper $SL(2,\R)$-action if and only if it is locally isomorphic to
$X_{2N,p}$ ($N\geq 3$).
Hence \eqref{item:case(a)-VE-locally-isom} follows from this result. 
\end{proof}

\vspace{\baselineskip}
\noindent \textbf{Proof of Theorem~\ref{theorem:classification_of_symmetric_space_with_very_even_condition} for Case \ref{item:classification_caseB}.}\label{subsubsec:classification_case_B}
We prove the following: 
\begin{proposition}
\label{prop:classification_case_B}
\begin{enumerate}
\item 
\label{item:case(b)-VET}
The symmetric spaces 
\begin{align*}
&SL(2N,\C)/SU(N+1,N-1)\quad (N\geq 1), \\
&SO(2N,\C)/SO(N+1,N-1)
\quad (N\geq 3)
\end{align*}
 have $\VET$ for the standard representation of $\mathfrak{sl}(2N,\C)$ and $\mathfrak{so}(2N,\C)$, respectively. 
\item 
\label{item:case(b)-VE-locally-isom}
Let $G_\C$ be a connected complex simple Lie group and $G$ its closed subgroup whose Lie algebra is a real form of $\mathfrak{g}_\C$. 
If $G_\C/G$ has $\VE$ and admits a proper $SL(2,\R)$-action, then it is locally isomorphic to one of the symmetric spaces in \eqref{item:case(b)-VET}. 
\end{enumerate}
\end{proposition}

To apply Fact~\ref{fact:okuda-cpx/real} in our proof,
we use the following lemma:
\begin{lemma}\label{lemma:very_even_not_match_caseB}
 Let $\mathfrak{g}$ be either $\mathfrak{su}(N+1,N-1)$ or $\mathfrak{so}(N+1,N-1)$ and $\varphi\colon \mathfrak{sl}(2,\C)\to \mathfrak{g}_\C$ be a Lie algebra homomorphism. 
 The following conditions are equivalent: 
 \begin{enumerate}[label=(\arabic*)]
     \item $\varphi|_{\mathfrak{sl}(2,\R)}$ is very even for the standard representation of $\mathfrak{g}_\C$.
     \item The weighted Dynkin diagram $F_{\varphi}$
     does not match the Satake diagram of $\mathfrak{g}$
     in the sense of Definition~\ref{definition:match}. 
 \end{enumerate}
\end{lemma}

\begin{proof}
Our assertion follows from 
Proposition~\ref{prop:characterization_of_no_odd_dimensional}, together with Remark~\ref{fact:invariance_Titz_inv} and 
Lemma~\ref{lemma:very-even-complex-real}~\ref{lemma:very-even-complex-real:item-(b)}. Recall that 
the Satake diagrams of $\mathfrak{su}(N+1,N-1)$ and $\mathfrak{so}(N+1,N-1)$ are respectively 
 
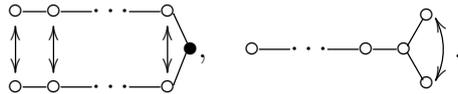
\[
 \begin{xy}
(0,5)="1"*{\circ}, 
(5,5)="2"*{\circ}, 
(12.5,5)="3"*{\cdots}, 
(20,5)="4"*{\circ}, 
(23,0)="5"*{\bullet},
(20,-5)="10"*{\circ}, 
(12.5,-5)="11"*{\cdots},
(5,-5)="12"*{\circ},
(0,-5)="13"*{\circ}, 
\ar@{-} "1"+(1,0);"2"+(-1,0)
\ar@{-} "2"+(1,0);"3"+(-2.5,0)
\ar@{-} "3"+(2.5,0);"4"+(-1,0)
\ar@{-} "4"+(0.8,-0.3);"5"+(-0.2,0.4)
\ar@{-} "10"+(0.8,0.3);"5"+(-0.2,0.4)
\ar@{-} "11"+(2.5,0);"10"+(-1,0)
\ar@{-} "12"+(1,0);"11"+(-2.5,0)
\ar@{-} "13"+(1,0);"12"+(-1,0)
\ar@{<->} "1"+(0,-2);"13"+(0,2)
\ar@{<->} "2"+(0,-2);"12"+(0,2)
\ar@{<->} "4"+(0,-2);"10"+(0,2)
\end{xy},\quad 
\begin{xy}
(5,0)="2"*{\circ},
(12.5,0)="3"*{\cdots},
(20,0)="4"*{\circ},  
(25,0)="5"*{\circ}, 
(28,4.5)="6"*{\circ}, 
(28,-4.5)="7"*{\circ},
\ar@{-} "2"+(1,0);"3"+(-2.5,0)
\ar@{-} "3"+(2.5,0);"4"+(-1,0)
\ar@{-} "4"+(1,0);"5"+(-1,0)
\ar@{-} "5"+(0.5,0.5);"6"+(-0.5,-0.5)
\ar@{-} "5"+(0.5,-0.5);"7"+(-0.5,0.5)
\ar@/^1mm/@{<->} "6"+(1.2,-0.4);"7"+(1.2,0.4)
\end{xy}\ .
\]
\end{proof}

We are ready to prove Proposition~\ref{prop:classification_case_B}. 
\begin{proof}[Proof of Proposition~\ref{prop:classification_case_B}]
\eqref{item:case(b)-VET}. 
By Fact~\ref{fact:okuda-cpx/real} 
and Lemma~\ref{lemma:very_even_not_match_caseB}, $G_\C/G$ has $\VE$ for $\iota$. 
Since $\iota$ lifts to a representation of $G_{\C}$,
the assertion is proved.

\eqref{item:case(b)-VE-locally-isom}. 
By Okuda's classification \cite[Thm.~1.4]{Oku13}, 
$G_{\C}/G$ with $\mathfrak{g}_{\C}$ simple admits a proper $SL(2,\R)$-action if and only if 
$\mathfrak{g}$ is isomorphic to 
one of the real Lie algebras 
in Table~\ref{tab:classification-caseB}.
Hence,
it suffices to show that $G_{\C}/G$ does not have $\VE$ in the case where $\mathfrak{g}$ is one of the bottom six real Lie algebras in Table~\ref{tab:classification-caseB}. 

\begin{table}[htbp]
\begin{center}
\begin{tabular}{cc}
real Lie algebra $\mathfrak{g}$ &  partition \\
\noalign{\hrule height 1.2pt}
$\mathfrak{su}(N+1,N-1)$ & $\VE$ \\
\hline
$\mathfrak{so}(N+1, N-1)$ & $\VE$ \\
\hline
$\mathfrak{sl}(N,\HA)$ & $[2N-1,1]$\\
\hline
\begin{tabular}{c}$\mathfrak{su}(p,q)$\\ ($p-q\geq 4$, $p+q=2N$) \end{tabular} & $[2N-1,1]$ \\
\hline
\begin{tabular}{c}$\mathfrak{so}(p,q)$\\ ($p-q\geq 3$, $p+q=2N+1$) \end{tabular} &$[2N+1]$ \\
\hline
\begin{tabular}{c}$\mathfrak{sp}(p,q)$ \\ ($p-q\geq0$, $p+q=N$)\end{tabular} & $[2N-2,1^2]$ \\
\hline
\begin{tabular}{c}$\mathfrak{so}(p,q)$\\ ($p-q\geq 4$, $p+q=2N$) \end{tabular} & $[2N-1,1]$\\
\hline
$\mathfrak{so}^*(2N)$ ($N\geq 4)$ & $[2N-1,1]$\\
\end{tabular}
\caption{Symmetric spaces $G_\C/G$ which admit a proper $SL(2,\R)$-action.}
\label{tab:classification-caseB}
\end{center}
\end{table}

By Fact~\ref{fact:weighted-Dynkin-diagram-classical-type}, 
for each partition in Table~\ref{tab:classification-caseB}, 
one can 
take a corresponding Lie algebra homomorphism
$\varphi\colon \mathfrak{sl}(2,\C)\rightarrow \mathfrak{g}_{\C}$.
Then we see that the weighted Dynkin diagram
$F_{\varphi}$ does not match the Satake diagram of $\mathfrak{g}$.
Hence, by Fact~\ref{fact:okuda-cpx/real},
the $SL(2,\R)$-action on $G_{\C}/G$ via the lift of  $\varphi|_{\mathfrak{sl}(2,\R)}$
is proper. 
On the other hand, one can check from Lemma~\ref{lemma:non-very-even} with 
Lemma~\ref{lemma:very-even-complex-real}~\ref{lemma:very-even-complex-real:item-(b)} that 
$\varphi|_{\mathfrak{sl}(2,\R)}$ is not very even for any representation $\iota$ 
of $\mathfrak{g}_{\C}$ such that $(\mathfrak{g}_{\C},\iota)$ is a complex classical pair. 
This proves the assertion.

For example, let us consider the case where  
$\mathfrak{g}=\mathfrak{so}^{*}(8)$, which corresponds to the last entry of Table~\ref{tab:classification-caseB} with $N=4$.
The weighted Dynkin diagram of the homomorphism
$\varphi\colon \mathfrak{sl}(2,\C)\to \mathfrak{g}_{\C}\simeq \mathfrak{so}(8,\C)$
corresponding to the partition $[7,1]$ is given by
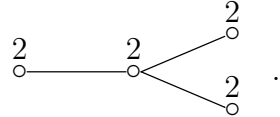
\[
\begin{xy}
(30,0)="3"*{\circ},(30,3)*{2},
(45,0)="4"*{\circ},"4"+(0,3)*{2},
(58,5)="5"*{\circ},"5"+(0,3)*{2},
(58,-5)="5'"*{\circ},"5'"+(0,3)*{2},
\ar@{-} "3"+(1,0);"4"+(-1,0)
\ar@{-} "4"+(1,0);"5"+(-1,-0.3)
\ar@{-} "4"+(1,0);"5'"+(-1,0.3)
\end{xy}\quad.
\]
This weighted Dynkin diagram does not match the Satake diagram of
$\mathfrak{so}^{*}(8)$.
Moreover, let 
$\tau\colon \mathfrak{so}(8,\C)\to \mathfrak{gl}(M,\C)$
be any representation such that $(\mathfrak{so}(8,\C),\tau)$
is a complex classical pair. 
Then we have $M=8$, and $\tau$ is either the standard representation
or the semispin representations.
By Lemma~\ref{lemma:non-very-even}~\eqref{lemma:non-very-even_item:D_4},
the homomorphism $\varphi$ is not very even for $\tau$ in either case.
\end{proof}

\vspace{\baselineskip}
\noindent \textbf{Proof of Theorem~\ref{theorem:classification_of_symmetric_space_with_very_even_condition} for Case \ref{item:classification_caseC}.}\label{subsubsec:classification_case_C}
We prove the following: 
\begin{proposition}\label{prop:classification_case_C}
\begin{enumerate}
\item 
\label{item:case(c)-VET}
Each symmetric space $G/H$ in Table~\ref{tab:sym_pair_case_c_group} has $\VET$ for the standard representation of $\mathfrak{g}$. 
\item
\label{item:case(c)-VE-locally-isom}
Let $G/H$ be a semisimple symmetric space with $\mathfrak{g}$ absolutely simple and assume Setting~\ref{setting:semisimple}. 
If $G/H$ has $\VE$ and admits a proper $SL(2,\R)$-action, then it is locally isomorphic to one of the symmetric spaces in Table~\ref{tab:sym_pair_case_c_group}. 
\end{enumerate}
\begin{table}[htbp]
\begin{scalebox}{0.85}{
\renewcommand{\arraystretch}{1.5}
\begin{tabular}{cc}
$G/H$  & $\mathfrak{g}^c$ \\
\noalign{\hrule height 1.2pt}
$SL(2N,\R)/SO(N+1,N-1)$  & $\mathfrak{su}(N+1,N-1)$ \\
\hline
$SL(2N,\HA)/Sp(N+1,N-1)$ & $\mathfrak{su}(2N+2,2N-2)$  \\
\hline
$SO^*(4N)/U(N+1,N-1)$ & 
$\mathfrak{so}(2N+2,2N-2)$
\\
\hline
\begin{tabular}{c}
$SO(N,N)/(SO(N-p, N-p-1)
\times SO(p,p+1))$\\ $(0\leq p< N/2)$ 
\end{tabular}
& $\mathfrak{so}(N+1,N-1)$ \\
\hline 
\begin{tabular}{c}
$SU(N,N)/
(S(U(N-p,N-p-1)\times U(p,p+1)))$\\ $(0\leq p<N/2)$
\end{tabular}& $\mathfrak{su}(N+1,N-1)$ \\
\hline
 \begin{tabular}{c}$Sp(N,N)/(Sp(N-p,N-p-1)\times 
Sp(p,p+1)) $ \\$(0\leq p<N/2)$
\end{tabular}
& $\mathfrak{sp}(N+1,N-1)$ \\
\hline
\begin{tabular}{c}
$SO^*(4N)/(SO^*(4N-4p-2)
\times SO^*(4p+2))$ \\$(0\leq p< N/2)$
\end{tabular} 
& $\mathfrak{so}^*(4N)$ \\
\end{tabular}
\renewcommand{\arraystretch}{1.5}}
\end{scalebox}
\caption{Symmetric spaces $G/H$ with $\mathfrak{g}$ absolutely simple which have $\VET$.}
\label{tab:sym_pair_case_c_group}
\end{table}
\end{proposition}

To apply Fact~\ref{fact:okuda} in our proof, 
we use the following: 

\begin{lemma}\label{lemma:very_even_and_Pair_of_Satake}
Let $(\mathfrak{g},\mathfrak{h})$ be one of the pairs corresponding to the symmetric spaces in 
Table~\ref{tab:sym_pair_case_c_group}, and 
let $(S_{\mathfrak{g}},S_{\mathfrak{g}^c})$ denote the pair 
of Satake diagrams. 
For a complex Lie algebra homomorphism $\varphi\colon \mathfrak{sl}(2,\C)\to \mathfrak{g}_{\C}$ 
whose weighted Dynkin diagram $F_{\varphi}$ matches $S_{\mathfrak{g}}$,   the following conditions are equivalent:
\begin{enumerate}[label=(\arabic*)]
    \item $\varphi$ is very even for the standard representation of $\mathfrak{g}_{\C}$.
    \item $F_{\varphi}$ does not match $S_{\mathfrak{g}^c}$.
\end{enumerate}

\end{lemma}
\begin{proof}
Draw explicitly the pair $(S_{\mathfrak{g}}, S_{\mathfrak{g}^{c}})$ of Satake diagrams for each symmetric pair $(\mathfrak{g},\mathfrak{h})$, as described in Section~\ref{subsubsec:properness_caseC}. 
For instance, 
in the case $(\mathfrak{g},\mathfrak{h})=(\mathfrak{sl}(2N,\HA), \mathfrak{sp}(N-1,N+1))$, we have 
\begin{equation*}
(S_{\mathfrak{g}},S_{\mathfrak{g}^{c}})
= \left(
\begin{xy}
(0,3.5)="1"*{\bullet}, 
(5,3.5)="2"*{\circ},
(12.5,3.5)="3"*{\cdots}, 
(20,3.5)="4"*{\circ}, 
(25,3.5)="4.5"*{\bullet},  
(28, 0)="5"*{\circ}, 
(25,-3.5)="5.5"*{\bullet},
(20,-3.5)="6"*{\circ},
(12.5,-3.5)="7"*{\cdots}, 
(5,-3.5)="8"*{\circ},
(0,-3.5)="9"*{\bullet},
\ar@{-} "1"+(0.8,0);"2"+(-0.8,0)
\ar@{-} "2"+(0.8,0);"3"+(-2.7,0)
\ar@{-} "3"+(2.7,0);"4"+(-0.8,0)
\ar@{-} "4"+(0.8,0);"4.5"+(-0.8,0)
\ar@{-} "4.5"+(0.7,-0.2);"5"+(-0.2,0.9)
\ar@{-} "5.5"+(0.8,0.2);"5"+(-0.2,-0.7)
\ar@{-} "6"+(0.8,0);"5.5"+(-0.8,0)
\ar@{-} "7"+(2.7,0);"6"+(-0.8,0)
\ar@{-} "8"+(0.8,0);"7"+(-2.7,0)
\ar@{-} "9"+(0.8,0);"8"+(-0.8,0)
\end{xy},
\begin{xy}
(0,3.5)="1"*{\circ}, 
(5,3.5)="2"*{\circ},
(12.5,3.5)="3"*{\cdots}, 
(20,3.5)="4"*{\circ}, 
(25,3.5)="4.5"*{\bullet},  
(28, 0)="5"*{\bullet}, 
(25,-3.5)="5.5"*{\bullet},
(20,-3.5)="6"*{\circ},
(12.5,-3.5)="7"*{\cdots}, 
(5,-3.5)="8"*{\circ},
(0,-3.5)="9"*{\circ},
\ar@{-} "1"+(0.8,0);"2"+(-0.8,0)
\ar@{-} "2"+(0.8,0);"3"+(-2.7,0)
\ar@{-} "3"+(2.7,0);"4"+(-0.8,0)
\ar@{-} "4"+(0.8,0);"4.5"+(-0.8,0)
\ar@{-} "4.5"+(0.7,-0.2);"5"+(-0.2,0.9)
\ar@{-} "5.5"+(0.8,0.2);"5"+(-0.2,-0.7)
\ar@{-} "6"+(0.8,0);"5.5"+(-0.8,0)
\ar@{-} "7"+(2.7,0);"6"+(-0.8,0)
\ar@{-} "8"+(0.8,0);"7"+(-2.7,0)
\ar@{-} "9"+(0.8,0);"8"+(-0.8,0)
\ar@{<->} "1"+(0,-1);"9"+(0,1)
\ar@{<->} "2"+(0,-1);"8"+(0,1)
\ar@{<->} "4"+(0,-1);"6"+(0,1)
\end{xy} 
\right).
\end{equation*}
In the case $(\mathfrak{g},\mathfrak{h})=(\mathfrak{so}^*(4N), \mathfrak{so}^*(4p+2)\oplus \mathfrak{so}^*(4N-4p-2))$ $(0\leq p< N/2)$, as in \eqref{eq:pair-satake-o*}, we have 
\begin{equation*}
(S_{\mathfrak{g}},S_{\mathfrak{g}^{c}})=
\left(
\begin{xy}
(0,0)="1"*{\bullet}, 
(5,0)="2"*{\circ},
(12.5,0)="3"*{\cdots},
(20,0)="4"*{\bullet},  
(25,0)="5"*{\circ}, 
(28,3)="6"*{\circ}, 
(28,-3)="7"*{\bullet}, 
\ar@{-} "1"+(1,0);"2"+(-1,0)
\ar@{-} "2"+(1,0);"3"+(-3,0)
\ar@{-} "3"+(3,0);"4"+(-1,0)
\ar@{-} "4"+(1,0);"5"+(-1,0)
\ar@{-} "5"+(0.5,0.5);"6"+(-0.5,-0.5)
\ar@{-} "5"+(0.5,-0.5);"7"+(-0.5,0.5)
\end{xy}, 
\begin{xy}
(0,0)="1"*{\bullet}, 
(5,0)="2"*{\circ},
(12.5,0)="3"*{\cdots},
(20,0)="4"*{\bullet},  
(25,0)="5"*{\circ}, 
(28,3)="6"*{\bullet}, 
(28,-3)="7"*{\circ}, 
\ar@{-} "1"+(1,0);"2"+(-1,0)
\ar@{-} "2"+(1,0);"3"+(-3,0)
\ar@{-} "3"+(3,0);"4"+(-1,0)
\ar@{-} "4"+(1,0);"5"+(-1,0)
\ar@{-} "5"+(0.5,0.5);"6"+(-0.5,-0.5)
\ar@{-} "5"+(0.5,-0.5);"7"+(-0.5,0.5)
\end{xy}\right).
\end{equation*}
Then our assertion is easily checked from Proposition~\ref{prop:characterization_of_no_odd_dimensional}, together with Remark~\ref{fact:invariance_Titz_inv}. 
\end{proof}

\begin{proof}[Proof of Proposition~\ref{prop:classification_case_C}]
\eqref{item:case(c)-VET}.
Let $G/H$ be one of the symmetric spaces in Table~\ref{tab:sym_pair_case_c_group}, 
$\iota$ the standard representation of $\mathfrak{g}$, 
and $\varphi\colon \mathfrak{sl}(2,\R)\to \mathfrak{g}$ a Lie algebra homomorphism. 
By Fact~\ref{fact:okuda}~\ref{fact:okuda:item:real}, 
the weighted Dynkin diagram $F_{\varphi_{\C}}$ of the $\C$-linear extension 
$\varphi_{\C}\colon \mathfrak{sl}(2,\C)\to \mathfrak{g}_{\C}$ 
matches $S_{\mathfrak{g}}$. 
Fact~\ref{fact:okuda}~\ref{fact:okuda:item:proper} then implies that the  
$SL(2,\R)$-action on $G/H$ via the lift of $\varphi$ is proper if and only if 
$F_{\varphi_{\C}}$ does not match $S_{\mathfrak{g}^{c}}$. 
By Lemma~\ref{lemma:very_even_and_Pair_of_Satake}, 
this is equivalent to $\varphi_{\C}$ being very even for the $\C$-linear extension $\iota_{\C}$,
and, by Lemma~\ref{lemma:very-even-complex-real}~\ref{lemma:very-even-complex-real:item-(c)}, this is equivalent to $\varphi$ being very even for $\iota$.
Thus $G/H$ has $\VE$ for $\iota$.
Since $\iota$ lifts to a Lie group representation of $G$ in each case, $G/H$ has $\VET$ for $\iota$.

\eqref{item:case(c)-VE-locally-isom}. 
It follows from Proposition~\ref{prop:characterization_of_no_odd_dimensional} that
if $G/H$ has $\VE$ and admits a proper $SL(2,\R)$-action, then 
$\mathfrak{g}_{\C}$ is of types $A$, $C$ or $D$.
By Okuda's classification \cite[Thm.~1.4]{Oku13}, 
a symmetric space $G/H$ with $\mathfrak{g}_{\C}$ of types $A$, $C$ or $D$ 
admits a proper $SL(2,\R)$-action if and only if $(\mathfrak{g},\mathfrak{h})$ is isomorphic to 
one of the symmetric pairs in 
Tables~\ref{tab:sym_pair_case_c_group} or \ref{longtab:non_very_even_symbol}.
Note that the pairs $(\mathfrak{so}(6,2), \mathfrak{u}(3,1))$ and 
$(\mathfrak{so}(3,3), \mathfrak{so}(3,\C))$, 
which are excluded in Table~\ref{longtab:non_very_even_symbol}, 
are isomorphic to 
$(\mathfrak{so}^*(8), \mathfrak{u}(3,1))$ and 
$(\mathfrak{sl}(4,\R), \mathfrak{so}(3,1))$ in Table~\ref{tab:sym_pair_case_c_group}, respectively.

Hence it suffices to show that 
$G/H$ does not have $\VE$ if $(\mathfrak{g},\mathfrak{h})$
    is isomorphic to 
    one of the symmetric pairs in Table~\ref{longtab:non_very_even_symbol}.
By Fact~\ref{fact:weighted-Dynkin-diagram-classical-type}, 
for each partition in Table~\ref{longtab:non_very_even_symbol}, 
take a corresponding Lie algebra homomorphism
$\varphi\colon \mathfrak{sl}(2,\C)\rightarrow \mathfrak{g}_{\C}$ and compute 
the weighted Dynkin diagram $F_{\varphi}$.
Then, one can check that the weighted Dynkin diagram $F_\varphi$ matches $S_\mathfrak{g}$, but does not match $S_{\mathfrak{g}^c}$. 
Hence, by Fact~\ref{fact:okuda}, 
we may assume that $\varphi(\mathfrak{sl}(2,\R))\subset \mathfrak{g}$ by replacing $\varphi$ with some $G_\C$-conjugate, 
and then the $SL(2,\R)$-action on $G/H$ via the lift of $\varphi$ is proper. 
On the other hand, by using Lemma~\ref{lemma:non-very-even} with Lemma~\ref{lemma:very-even-complex-real}~\ref{lemma:very-even-complex-real:item-(c)}, one can show that the Lie algebra homomorphism 
$\varphi|_{\mathfrak{sl}(2,\R)}\colon \mathfrak{sl}(2,\R)\rightarrow \mathfrak{g}$
is not very even for any  representation $\tau$ of $\mathfrak{g}$ such that $(\mathfrak{g},\tau)$ is a classical pair. Hence, $G/H$ does not have $\VE$. 
\end{proof}

This completes the proof of Theorem~\ref{theorem:classification_of_symmetric_space_with_very_even_condition}.

\subsection{Non-symmetric examples with 
\texorpdfstring{$\VE$}{Property (VE)}}
\label{section:non-symmetric}

In this subsection, we give homogeneous spaces that have Property~(VE) but are not locally isomorphic to any symmetric space:
\begin{proposition}\label{prop:main_prop_for_non-symmetric}
Let $N\geq 2$, $\D=\R,\C,\HA$, and  
$\F=\R,\C$. Suppose $G/H$ is one of the following homogeneous spaces:
\begin{enumerate}[label=(\roman*)]
\item\label{item:SL/SL} $SL(2N,\D)/\left(SL(p,\D)\times SL(q,\D)\right)$ $(p,q\text{ are odd, }p+q=2N)$;
\item\label{item:Sp/Sp}$Sp(N,\F)/
\prod_{k=1}^{l} Sp(p_{k},\F)$ 
$(\sum_{k=1}^{l} p_{k}=N-1)$.
\end{enumerate} 
Then, $G/H$ has $\VET$ for the standard representation of $\mathfrak{g}$. 
\end{proposition}

To prove the assertion, we apply Kobayashi's properness criterion (Fact~\ref{fact:properness_criterion}). 
In what follows, we use the notation in Section~\ref{subsec:properness_criterion}. Note that 
\[
A_0:=\begin{pmatrix}1 &0\\ 0 &-1\end{pmatrix}
\]
spans a maximal split abelian subalgebra of $\mathfrak{sl}(2,\R)$. 
Further, we can take a maximal split abelian subalgebra $\mathfrak{a}$ of $\mathfrak{g}$ as follows:  
\begin{align*}
\ref{item:SL/SL}\ 
\mathfrak{a}&:=\left\{\operatorname{diag}(a_1,\cdots, a_{2N})\ \middle|\ a_i\in \R,\ \sum_{i=1}^{2N}{a_i}=0 \right\}, \\ 
\ref{item:Sp/Sp}\ 
    \mathfrak{a}&:=\{\operatorname{diag}(a_1,\cdots, a_N, -a_N,\cdots, -a_1)\mid a_i\in \R \}. 
\end{align*}

\begin{lemma}\label{lem:key_for_non-symmetric}
In each case of \ref{item:SL/SL} and \ref{item:Sp/Sp} in Proposition~\ref{prop:main_prop_for_non-symmetric},
let $\varphi\colon \mathfrak{sl}(2,\R)\to \mathfrak{g}$ be a homomorphism with $\varphi(A_0)\in \mathfrak{a}$. 
Then $\varphi(A_0)\not \in W(\mathfrak{g},\mathfrak{a})\cdot\mathfrak{a}_{H}$
if and only if $\varphi$ is very even for the standard representation.
\end{lemma}
\begin{proof}
Let $\iota$ be the standard representation of $\mathfrak{g}$.

In Case~\ref{item:SL/SL},  
the set $W(\mathfrak{g},\mathfrak{a}) \cdot \mathfrak{a}_H$ consists of all elements $X \in \mathfrak{a}$ such that  
the sum of certain $p$ eigenvalues of $X$ equals zero.  
We note that  
all eigenvalues of $\varphi(A_0) \in \mathfrak{a}$ can be expressed as $c_1, -c_1, \dots, c_N, -c_N$ with $c_1 \geq c_2 \geq \cdots \geq c_N \geq 0$.

If $\varphi$ is very even for $\iota$,  
then all $c_i$ are odd integers.  
Consequently, the sum of any $p$ eigenvalues of $\varphi(A_0)$ is odd since $p$ is odd.  
This implies $\varphi(A_0) \not\in W(\mathfrak{g},\mathfrak{a}) \cdot \mathfrak{a}_H$. 

On the other hand, if $\varphi$ is 
not very even for $\iota$, then $c_N = 0$.  
In this case, 
the sum of the $p$ elements $\pm c_{1},\ldots,\pm c_{(p-1)/2},c_{N}$
is zero, which shows that $\varphi(A_0) \in W(\mathfrak{g},\mathfrak{a}) \cdot \mathfrak{a}_H$.  
This completes the proof for Case~\ref{item:SL/SL}.

In Case~\ref{item:Sp/Sp}, 
$W(\mathfrak{g},\mathfrak{a})\cdot\mathfrak{a}_{H}$
consists of all elements of $\mathfrak{a}$ whose eigenvalues are at least one zero. 
It follows from representation theory of $\mathfrak{sl}(2,\R)$ that
$\varphi$ is very even for $\iota$
if and only if no eigenvalues of $\varphi(A_{0})\in\mathfrak{a}$ are zero.
Hence, our assertion for Case~\ref{item:Sp/Sp} is also proved.
\end{proof}

We are ready to prove Proposition~\ref{prop:main_prop_for_non-symmetric}.
\begin{proof}[Proof of Proposition~\ref{prop:main_prop_for_non-symmetric}]
Let $G/H$ be a homogeneous space 
in Proposition~\ref{prop:main_prop_for_non-symmetric}
and $\iota$ the standard representation of $\mathfrak{g}$.
Since $\iota$ lifts to $G$, 
it suffices to show that $G/H$ has $\VE$ for $\iota$. 

Let $\varphi\colon \mathfrak{sl}(2,\R)\to \mathfrak{g}$ be a Lie algebra homomorphism. 
To prove the assertion, we may assume that
$\varphi(A_0)\in \mathfrak{a}$ after replacing $\varphi$ by a suitable
$G$-conjugate if necessary.
From Lemma~\ref{lem:key_for_non-symmetric} above, 
$\varphi$ is very even for $\iota$ if and only if $\varphi(A_0)\not \in W(\mathfrak{g},\mathfrak{a})\mathfrak{a}_{H}$. 
From Kobayashi's properness criterion (Fact~\ref{fact:properness_criterion}), this is equivalent to the condition that the $SL(2,\R)$-action on $G/H$ via the lift of $\varphi$ is proper. 
Thus, $G/H$ has $\VE$ for $\iota$. This proves the assertion.
\end{proof}

\begin{landscape}
\renewcommand{\arraystretch}{1.2}
\begin{longtable}{>{\centering}p{8.5em}ccc}
\caption{Symmetric spaces admitting a proper $SL(2,\R)$-action via a non-very even homomorphism.}
\label{longtab:non_very_even_symbol}
\endfirsthead
     \caption{(continued)}
    \endhead
$\mathfrak{g}$ & $\mathfrak{h}$ & $\mathfrak{g}^c$ & partition \\
\noalign{\hrule height 1.2pt}
  $\mathfrak{sl}(2N,\R)$ & $\mathfrak{sl}(N,\C)\oplus \mathfrak{so}(2)$ & $\mathfrak{sl}(N,\HA)$ &$[2N-1, 1 ]$\\
  \hline
  $\mathfrak{sl}(N,\R)$ & \begin{tabular}{c} $\mathfrak{so}(p,q)$ ($p+q=N$, $p>q+2$)\end{tabular} & $\mathfrak{su}(p,q)$  & $[N-1,1]$\\ 
  \hline
  $\mathfrak{sl}(N,\HA)$ & \begin{tabular}{c} $\mathfrak{sp}(p,q)$ ($p+q=N$, $p> q+2$)\end{tabular} & $\mathfrak{su}(2p,2q)$ &$[(N-1)^2,1^2]$\\
  \hline
  $\mathfrak{su}(2p,2q)$ & 
  $\mathfrak{sp}(p,q)$ 
  ($p>q$) 
  & $\mathfrak{sl}(p+q,\HA)$ &$[4q+1, 1^{2(p-q)-1}]$\\
  \hline
  $\mathfrak{su}(2N,2N)$ & $\mathfrak{sp}(N,N)$ & $\mathfrak{sl}(2N,\HA)$ &$[4N-1,1]$\\
  \hline
  $\mathfrak{su}(N,N)$ & $\mathfrak{so}^*(2N)$ & $\mathfrak{sl}(N,\HA)$  &$[2N-1,1]$\\ 
  \hline
  $\mathfrak{su}(p,q)$ ($p>q$) & \begin{tabular}{c}$\mathfrak{su}(p_1, q_1)\oplus \mathfrak{su}(p_2,q_2)\oplus \mathfrak{so}(2)$\\ ($p_1>q_1, q_2>p_2$, 
  $p_{1}+p_{2}=p$, $q_{1}+q_{2}=q$)
  \end{tabular}& $\mathfrak{su}(p_{1}+q_{2},p_{2}+q_{1})$ &$[2q+1, 1^{p-q-1}]$\\
  \hline
  $\mathfrak{su}(N,N)$  & \begin{tabular}{c}$\mathfrak{su}(p_1, q_1)\oplus \mathfrak{su}(p_2,q_2)\oplus \mathfrak{so}(2)$\\ ($p_1>q_1+1, q_2>p_2$, 
  $p_{1}+p_{2}=q_{1}+q_{2}=N$)
  \end{tabular} & $\mathfrak{su}(p_{1}+q_{2},p_{2}+q_{1})$ &$[2N-1, 1]$\\
    \hline
  $\mathfrak{sp}(N,\R)$ & \begin{tabular}{c}$\mathfrak{su}(p,q)\oplus \mathfrak{so}(2)$ ($p\geq q$, $p+q=N$)\end{tabular} & $\mathfrak{sp}(p,q)$ & $[2(N-1),1^2]$\\
    \hline
    $\mathfrak{sp}(2N,\R)$ & $\mathfrak{sp}(N,\C)$ & $\mathfrak{sp}(N,N)$  &$[4N-2, 1^2]$\\
    \hline
    \begin{tabular}{c}
    $\mathfrak{sp}(p,q)$ 
    $(p>q)$ \end{tabular}
    & \begin{tabular}{c}$\mathfrak{sp}(p_1,q_1)\oplus \mathfrak{sp}(p_2,q_2)$ \\ 
    ($p_1>q_1, q_2>p_2$, 
    $p_{1}+p_{2}=p$, $q_{1}+q_{2}=q$) \end{tabular} & $\mathfrak{sp}(p_{1}+q_{2},p_{2}+q_{1})$  & $[(2q+1)^2, 1^{2p-2q-2}]$\\
    \hline
    $\mathfrak{sp}(N,N)$ & \begin{tabular}{c}$\mathfrak{sp}(p_1,q_1)\oplus \mathfrak{sp}(p_2,q_2)$\\ ($p_1>q_1+1, q_2>p_2$, 
    $p_{1}+p_{2}=q_{1}+q_{2}=N$)
    \end{tabular} & $\mathfrak{sp}(p_{1}+q_{2},p_{2}+q_{1})$ &$[(2N-1)^2, 1^2]$\\
    \hline
    \begin{tabular}{c}
    $\mathfrak{so}(p,q)$ \\
    ($p+q$ is even, \\
    $p\geq q+2$, \\$(p,q)\neq (6,2)$)
    \end{tabular}& \begin{tabular}{c}$\mathfrak{so}(p_1,q_1)\oplus \mathfrak{so}(p_2,q_2)$\\ ($p_1>q_1$, $q_2>p_2$, 
    $p_{1}+p_{2}=p$, $q_{1}+q_{2}=q$) \\
    \end{tabular} & $\mathfrak{so}(p_{1}+q_{2},p_{2}+q_{1})$ &$[2q+1,1^{p-q-1}]$\\
    \hline
    $\mathfrak{so}(N,N)$ & \begin{tabular}{c}$\mathfrak{so}(p_1,q_1)\oplus \mathfrak{so}(p_2,q_2)$ \\ ($p_1>q_1+1$, $q_2>p_2$, 
    $p_{1}+p_{2}=q_{1}+q_{2}=N$)
    \end{tabular} & $\mathfrak{so}(p_{1}+q_{2},p_{2}+q_{1})$ &$[2N-1,1]$\\
    \hline
    $\mathfrak{so}(6,2)$ & $\mathfrak{so}(6,1)$ & $\mathfrak{so}(7,1)$  & $[3^2,1^2]$
    \\
    \hline
    \begin{tabular}{c}
    $\mathfrak{so}(2p,2q)$ \\
    ($p\geq q+1$, \\
    $(p,q)\neq (3,1)$) 
    \end{tabular} & $\mathfrak{su}(p,q)\oplus \mathfrak{so}(2)$ & $\mathfrak{so}^{*}(2(p+q))$ &$[4q+1, 1^{2p-2q-1}]$\\
    \hline
    $\mathfrak{so}(2N,2N)$ & $\mathfrak{su}(N,N)\oplus \mathfrak{so}(2)$ & $\mathfrak{so}^{*}(4N)$ & $[4N-1,1]$\\
    \hline
    $\mathfrak{so}^*(4N)$ & \begin{tabular}{c}$\mathfrak{su}(p,q)\oplus \mathfrak{so}(2)$\\ ($p+q=2N$, $p>q+2$) \end{tabular} 
    & $\mathfrak{so}(2p,2q)$  & $[2^{2N-2}, 1^4]$\\
    \hline
    $\mathfrak{so}^*(4N+2)$ & \begin{tabular}{c}
    $\mathfrak{su}(p,q)\oplus \mathfrak{so}(2)$ \\
    ($p+q=2N+1, p>q+1)$
    \end{tabular} 
    & $\mathfrak{so}(2p,2q)$ &$[(2N+1)^2]$\\
    \hline
    $\mathfrak{so}(N,N)$ ($N\geq 4$)& $\mathfrak{so}(N,\C)$ & $\mathfrak{so}^{*}(2N)$ &$[2N-1,1]$\\
    \hline
\end{longtable}
\renewcommand{\arraystretch}{1.5}
\end{landscape}

\appendix

\section{Proof of 
\texorpdfstring{Proposition \ref{prop:embedding-criterion}}{Proposition 5.9}}
\label{appendix:embedding}
The goal of this appendix is to prove the following:
\begin{proposition}
\label{prop:emb-criterion-variant}
    For any $g\in GL(N,\C)$, we have the following four inclusion relations
    up to $GL(N,\C)$-conjugacy:
    \begin{align*}
        (g^{-1} O(N,\C)g)\cap GL(N,\R) &\subset O(p,q)\ (\exists p,q\in\N \text{ s.t. }p+q=N), \\
        (g^{-1} Sp(N/2,\C)g)\cap GL(N,\R) &\subset Sp(N/2,\R), \\
        (g^{-1} O(N,\C)g)\cap GL(N/2,\HA) &\subset O^{*}(2(N/2)), \\
        (g^{-1} Sp(N/2,\C)g)\cap GL(N/2,\HA) &\subset Sp(p,q)\ (\exists p,q\in\N \text{ s.t. }p+q=N/2).
    \end{align*}
\end{proposition}

Proposition \ref{prop:embedding-criterion} is 
an immediate consequence of this proposition.

We prove simultaneously each of the above inclusion relations between classical Lie groups in terms of central simple algebras.
For this purpose we introduce a few notations.
Let $A$ be a finite-dimensional algebra over a commutative field $\F$.
Recall that an $\F$-linear map $\sigma\colon A\rightarrow A$ is called 
an \emph{anti-involution} on $A$ if $\sigma^{2}=\id$ and  $\sigma(xy)=\sigma(y)\sigma(x)$ for all $x,y\in A$.
Then we define an $\F$-algebraic group $G_{A}(\sigma)$ as 
    \[
    G_{A}(\sigma):=\{x\in A^{\times}\mid \sigma(x)x=1\}, 
    \]
where $A^{\times}$ denotes
the group of invertible elements in the $\F$-algebra $A$. 
We say that two $\F$-linear anti-involutions $\sigma_{1},\sigma_{2}\colon A\rightarrow A$ are \emph{equivalent} if there exists $a\in A^\times$ satisfying 
\[
\sigma_{2}(x)=a^{-1}\sigma_{1}(axa^{-1})a
\]
for all $x\in A$. Then we have $G_{A}(\sigma_{1})=aG_{A}(\sigma_{2})a^{-1}$. 
In particular, the two $\F$-algebraic groups $G_{A}(\sigma_{1}),G_{A}(\sigma_{2})$ are $\F$-isomorphic.

\begin{example}
\label{ex:classical}
Let $I_{N}$ be the identity matrix of size $N\in \N$.
We consider the $N\times N$ matrices 
\[
I_{p,q}:=\begin{pmatrix}
I_{p} & 0 \\
0 & -I_{q}
\end{pmatrix}
\text{ $(p+q=N)$ and }
J_{N}:=
\begin{pmatrix}
0 & -I_{N/2} \\
I_{N/2} & 0 
\end{pmatrix}\  \ (\text{when }N\in 2\N).
\]
Let $\D$ be one of $\R$, $\C$ or $\HA$, with $\F$ the center of $\D$.
The following $\F$-linear anti-involutions $\sigma$ on $A=M(N,\D)$
define classical Lie groups as below:
\begin{itemize}
    \item ($\D=\R$) 
    \begin{itemize}
        \item 
        For $\sigma(x) = I_{p,q}\!^{t}xI_{p,q}^{-1}$
        we have $G_{A}(\sigma) = O(p,q)$.
        \item
        For $\sigma(x) = J_{N}\!^{t}xJ_{N}^{-1}$
        we have $G_{A}(\sigma) = Sp(N/2,\R)$.
    \end{itemize}
    \item ($\D=\C$) 
    \begin{itemize}
        \item 
        For $\sigma(x) = \!^{t}x$ 
        we have $G_{A}(\sigma) = O(N,\C)$.
        \item
        For $\sigma(x) = J_{N}\!^{t}xJ_{N}^{-1}$
        we have $G_{A}(\sigma) = Sp(N/2,\C)$.
    \end{itemize}      
    \item ($\D=\HA:=\R+\R\Hai+\R\Haj+\R\Hak$) 
    \begin{itemize}
        \item
        For $\sigma(x) = \Haj x^{*}\Haj^{-1}$
        we have $G_{A}(\sigma) = O^{*}(2N)$.
        \item 
        For $\sigma(x) = I_{p,q}x^{*} I_{p,q}^{-1}$ 
        we have $G_{A}(\sigma) = Sp(p,q)$.
    \end{itemize}
\end{itemize}
\end{example}

The following lemma is elementary: 
\begin{lemma}
    \label{lem:classification-of-inv}
    Let $\D$ be one of $\R$, $\C$ or $\HA$, with $\F$ the center of $\D$.
    Any $\F$-linear anti-involution on $M(N,\D)$ is equivalent 
    to one of those in Example~\ref{ex:classical}.
\end{lemma}

By Lemma \ref{lem:classification-of-inv}, Proposition \ref{prop:emb-criterion-variant} immediately reduces to the following:
\begin{proposition}
    \label{prop:emb-critetion-algebra}
    Let $A$ be a central simple $\R$-algebra,
    and $\sigma$ a $\C$-linear anti-involution on the complexification
    $A_{\C}:=A\otimes_{\R}\C$. 
    Then there exists an $\R$-linear anti-involution 
    $\sigma'\colon A\rightarrow A$ satisfying the following two conditions:
    \begin{enumerate}
        \item 
        $\sigma$ is equivalent to  
        the anti-involution $\sigma'_{\C}=\sigma'\otimes \id$ on $A_{\C}=A\otimes_{\R}\C$;
        \item
        $G_{A_{\C}}(\sigma)\cap A^{\times}\subset G_{A}(\sigma')$.
    \end{enumerate}
\end{proposition}

The rest of this section is devoted to showing this proposition.
Let us give some lemmas for the proof.

\begin{lemma}
    \label{lem:cohomology-vanish}
    Let $\sigma\colon A_{\C}\rightarrow A_{\C}$ be a $\C$-linear anti-involution
    on a central simple $\C$-algebra $A_{\C}$. 
    If an invertible element $a$ in $A_{\C}$ satisfies $\sigma(a)=a$, 
    then there exists $b\in A_{\C}$ such that $a=\sigma(b)b$.
\end{lemma}

\begin{proof}
    By Wedderburn's theorem, it suffices to consider the case $A_{\C}=M(N,\C)$. 
    Furthermore, by Lemma \ref{lem:classification-of-inv},
    we may assume that $\sigma$ is
    defined by $\sigma(a)=\!^{t}a$ or $\sigma(a)=J_{N}\!^{t}aJ_{N}^{-1}$ ($N$ is even). 
    Then our claim follows immediately from the elementary fact that
    any invertible complex symmetric (resp.\ skew-symmetric) matrices can be written as $\!^{t}bb$
    (resp.\ $\!^{t}bJ_{N}b$) for some $b\in GL(N,\C)$.
\end{proof}

\begin{lemma}
    \label{lem:inv-real-complex}
    Let $A$ be a finite-dimensional $\R$-algebra.
    If an element $a+\sqrt{-1}b$ ($a,b\in A$) in $A_{\C}=A\otimes_{\R}\C$ is invertible, then there exists a real number $r$ such that $a+rb\in A$ is also invertible.
\end{lemma}

\begin{proof}
    We consider the $\C$-algebra homomorphism $\phi\colon A_{\C}\rightarrow \End_{\C}(A_{\C})$ defined by $\phi(x):=$(the left multiplication map by $x$ on $A_{\C}$).
    Notice that $x$ is invertible in $A_{\C}$  if and only if $\phi(x)$ is invertible in $\End_{\C}(A_{\C})$.
    
    We consider the polynomial function $f(t):=\det(\phi(a+tb))$ for $t\in \C$.
    Since $a+\sqrt{-1}b$ is invertible, we have $f(\sqrt{-1})\neq 0$.
    Hence $f(t)$ is a non-zero polynomial, and thus there exists
    a real number $r$ such that $f(r)\neq 0$. Then $a+rb\in A$ is also invertible, which proves our claim.
\end{proof}

We are ready to prove Proposition \ref{prop:emb-critetion-algebra}. 
\begin{proof}[Proof of Proposition \ref{prop:emb-critetion-algebra}]
    By Wedderburn's theorem, $A$ is isomorphic to $M(N,\R)$ or $M(N,\HA)$ as an $\R$-algebra. 
    Hence by the classification 
    of anti-involutions (Lemma~\ref{lem:classification-of-inv}), 
    we may and do take an $\R$-linear anti-involution $\tau$ on $A$ 
    such that the $\C$-linear anti-involution 
    $\tau_{\C}:=\tau\otimes \id$ on $A_{\C}$ is equivalent to $\sigma$.
    Namely, there exists an invertible element $g$ in $A_{\C}$ such that for all $x\in A_{\C}$ we have $\sigma(x)=g^{-1}\tau_{\C} (gxg^{-1})g$.
    
    Put $c:=\tau_{\C}(g)g\in A_{\C}$.  
    Then for all $x\in A_{\C}$ we have
    \[
    \sigma(x) = c^{-1}\tau_{\C}(x)c.
    \]
    Write $c=a+\sqrt{-1}b$ using $a,b\in A$. 
    By Lemma \ref{lem:inv-real-complex}, 
    there exists a real number $r$ such that $c':=a+rb\in A$ is invertible.
    Then we consider an $\R$-linear anti-automorphism $\sigma'$ on $A$
    defined by $\sigma'(x) := c'^{-1}\tau(x)c'$.
    
    Let us see that $\sigma'$ is an anti-involution
    and that its $\C$-linear extension $\sigma'_{\C}$ on $A_{\C}$
    is equivalent to $\sigma$.
    Notice $\tau_{\C}(c)=c$ since $c=\tau_{\C}(g)g$. 
    In particular, we have $\tau(a)=a$ and $\tau(b)=b$, 
    and thus $\tau(c')=c'$. Hence $\sigma'$ is an anti-involution. 
    Further, we see from Lemma \ref{lem:cohomology-vanish} that 
    there exists $g'\in A_{\C}$ such that $c'= \tau_{\C}(g')g'$.
    Hence $\sigma'_{\C}$ is equivalent to $\tau_{\C}$, 
    and also equivalent to $\sigma$.

    Now we prove $G_{A_{\C}}(\sigma)\cap A^{\times}\subset G_{A}(\sigma')$.
    Let $x\in G_{A_{\C}}(\sigma)\cap A^{\times}$.
    Then we have $\tau(x)cx=c$. In particular, 
    $\tau(x)ax=a$ and $\tau(x)bx=b$ since $x\in A$. 
    Hence we obtain $\tau(x)c'x=c'$, which means $x\in G_{A}(\sigma')$.
    Thus the $\R$-linear anti-involution $\sigma'$ satisfies the two desired conditions.
\end{proof}

\section{Topics related to spin groups associated with indefinite quadratic forms}
\label{appendix:topocs-clifford}

In this appendix, we briefly review some basic facts, including some less well-known ones, about Clifford algebras and spin groups over $\R$ associated with indefinite quadratic forms, which were used in Sections~\ref{section:proper->spin} and~\ref{section:proof_proper-HR}.
Although in this paper we only need the case of quadratic forms of signature $(n,1)$, we include the general signature for completeness.

\subsection{Clifford algebras, spin groups, and spin representations}
\label{section:clifford-spin}
In this subsection, we introduce some notation related to Clifford algebras
and spin groups. 
As a reference, see for example \cite[Chap.~5]{Varadarajan_supersymmetry}.

Let $V$ be an $n$-dimensional vector space 
over $\F=\R$ or $\C$ equipped with a non-degenerate symmetric bilinear form $Q\colon V\times V \rightarrow \F$. 
As usual, we put $Q(v):=Q(v,v)$, which defines a quadratic form on the $\F$-vector space $V$. 
We consider the tensor algebra $TV:=\bigoplus_{i\in\N}V^{\otimes i}$ 
and its two-sided ideal $I(Q)$ generated by 
$v\otimes v - Q(v)$ for all $v\in V$.
Then the quotient $\F$-algebra 
\[
C(V)\equiv C(V,Q):= TV/I(Q)
\]
is called the \emph{Clifford algebra} of the quadratic form $Q$.

Let $q\colon TV\rightarrow C(V)$ be the quotient homomorphism. 
If $\{e_{1},\ldots,e_{n}\}$ is an $\F$-basis of $V$,
then the elements 
$q(e_{i_{1}})\cdots q(e_{i_{k}})$
$(1\leq i_{1}<\ldots< i_{k}\leq n,\ 0\leq k\leq n)$
form an $\F$-basis of $C(V)$. 
In particular, $\dim_{\F}C(V)=2^{n}$ and 
the restriction of $q$ to the subspace $V$ of $TV$
is injective. Throughout this paper, 
we think of $V$ as a subspace of $C(V)$
via $q$ and the image $q(v)$ of $v\in V$ is also denoted by the same symbol $v$.

The $\F$-linear map
$V\rightarrow C(V)$ defined by $v\mapsto -v$
is uniquely extended to an involutive $\F$-algebra automorphism $(-)'$
of $C(V)$. 
We define an $\F$-subalgebra $C^{+}(V)$ of $C(V)$ by 
\[
C^{+}(V)\equiv C^{+}(V,Q):=\{x\in C(V,Q)\mid x'=x\},
\]
which is called the \emph{even Clifford algebra} of $Q$.
If $\{e_{1},\ldots,e_{n}\}$ is an $\F$-basis of $V$, then 
we have 
$(e_{i_{1}}\cdots e_{i_{k}})'=(-1)^{k}e_{i_{1}}\cdots e_{i_{k}}$
for $1\leq i_{1}<\ldots< i_{k}\leq n$
and $0\leq k\leq n$.
Hence $e_{i_{1}}\cdots e_{i_{k}}\in C^{+}(V)$
if and only if $k$ is even.
In particular, $\dim_{\F}C^{+}(V) = 2^{n-1}$.

The $\F$-linear map $V\rightarrow C(V)$ defined by 
$v\mapsto v$
is uniquely extended to an $\F$-linear \emph{anti}-involution $(-)^{*}$
of $C(V)$. 
Now we define the \emph{spin group} of $Q$ by 
\[
Spin(V)\equiv Spin(V,Q) :=\{x\in C^{+}(V)^{\times} \mid 
xx^{*} = 1 \text{ and } xVx^{-1}\subset V
\}.
\]
Note that the elements $\pm 1 \in C^{+}(V)$ 
belong to $Spin(V)$.

For $x\in Spin(V)$ and $v\in V$, 
we have $xvx^{-1}\in V$ by definition.
One can check that the $\F$-linear map $v\mapsto xvx^{-1}$ 
defines an element of the special orthogonal group $SO(V)$
with respect to $Q$. 
Thus we obtain a Lie group homomorphism
\[
\Ad\colon Spin(V)\rightarrow SO(V),\ \ x\mapsto (v\mapsto xvx^{-1}).
\]
As is well-known, the kernel of this homomorphism is $\{\pm 1\}$
and the image coincides with the identity component of 
the Lie group $SO(V)$.

Now we recall the spin representation
and semispin representations. 
It is known that there exists an isomorphism
\begin{equation}
    \label{eq:clifford-algebra}
    C^{+}(V)\otimes_{\F}\C
    \simeq 
    \begin{cases}
        M(2^{(n-1)/2},\C) & \text{($n$ is odd)}, \\
        M(2^{(n-2)/2},\C)^{\oplus 2}
        & \text{($n$ is even)},
    \end{cases}
\end{equation} 
of $\C$-algebras. Thus the inclusion 
$Spin(V)\hookrightarrow (C^{+}(V)\otimes_{\F}\C)^{\times}$
defines a homomorphism 
\begin{equation}
\label{def:spin}
S\colon Spin(V)\rightarrow
\begin{cases}
    GL(2^{(n-1)/2},\C) & \text{($n$ is odd)}, \\
    GL(2^{(n-2)/2},\C)^{\times 2}
    & \text{($n$ is even)},
\end{cases}
\end{equation}
which is called the \emph{spin representation} of $Spin(V)$.
Further, we put 
\begin{equation}
\label{def:semispin}
S_{i}:=\operatorname{pr}_{i}\circ S
\colon Spin(V)\rightarrow
    GL(2^{(n-2)/2},\C)
\ \ (i=1,2)
\end{equation}
if $n$ is even, where $\operatorname{pr}_{i}$ is the $i$th projection.
The homomorphisms $S_{1}$ and $S_{2}$ are called the \emph{semispin representations}
of $Spin(V)$.

Let us recall that the Lie algebra $\mathfrak{spin}(V)$ of $Spin(V)$ is realized as a Lie subalgebra of $C^{+}(V)$. Since $C^{+}(V)$ is an $\F$-algebra, we note that $C^{+}(V)$ is a Lie algebra over $\F$ with the bracket $[x,y] = xy - yx$. Further, via the exponential map $\exp\colon C^{+}(V) \rightarrow C^{+}(V)^\times$ ($x \mapsto \sum_{n=0}^{\infty} x^n/n!$), we think of $C^{+}(V)$ as the Lie algebra of the Lie group $C^{+}(V)^\times$. In particular, the Lie algebra $\mathfrak{spin}(V)$ of $Spin(V)$
has the following realization as a Lie subalgebra of $C^{+}(V)$:
\begin{equation}
\label{def:spin-lie}
\mathfrak{spin}(V)= 
\{x\in C^{+}(V,Q)\mid 
x+ x^{*} = 0 \text{ and } \ad(x)(V) \subset V
\},
\end{equation}
where $\ad(x)(v) = [x,v]$. If ${e_{1},\ldots,e_{n}}$ form an orthogonal basis of the $\F$-vector space $V$ with respect to $Q$, it can be easily verified that the $\F$-vector space 
$\mathfrak{spin}(V)$ is spanned by the elements $e_{i}e_{j}$ $(1\leq i<j\leq n)$.

Let $\R^{p,q}$ be the real vector space $\R^{p+q}$ equipped with the quadratic form $x_{1}^{2}+\ldots+ x_{p}^2-x_{p+1}^2-\ldots -x_{p+q}^2$. In this case, $C(\R^{p,q})$, $C^{+}(\R^{p,q})$, $Spin(\R^{p,q})$, and $\mathfrak{spin}(\R^{p,q})$ are denoted by $C(p,q)$, $C^{+}(p,q)$, $Spin(p,q)$, and $\mathfrak{spin}(p,q)$, respectively. When $q=0$, we write these as $C(p)$, $C^{+}(p)$, $Spin(p)$, and $\mathfrak{spin}(p)$, respectively.

It is known that the center of $Spin(p,q)$ with $p+q\geq 3$ is isomorphic to $\mathbb{Z}/2\mathbb{Z}$ 
when $p+q$ is odd or when both $p$ and $q$ are odd, 
and isomorphic to $\mathbb{Z}/4\mathbb{Z}$ when both $p$ and $q$ are even.  
In particular, the structure of the center of $Spin(n,1)$ was used in 
Section~\ref{section:proper->spin}:
\begin{lemma}
\label{lemma:spin-center}
For $n\geq 2$,
the center of $Spin(n,1)$ coincides with $\{\pm 1\}$.
\end{lemma}

Let us fix a Cartan involution of $\mathfrak{spin}(p,q)$. 
Take the standard basis $\{e_{1}^{+},\ldots,e_{p}^{+},e_{1}^{-},\ldots,e_{q}^{-}\}$ of $\mathbb{R}^{p+q}$. 
Then we define an involutive $\mathbb{R}$-algebra automorphism $\theta_{p,q}\colon C(p,q)\rightarrow C(p,q)$  by
$\theta_{p,q}(e^{+}_{i})=e^{+}_{i}$ and $\theta_{p,q}(e^{-}_{i})=-e^{-}_{i}$.
The $\mathbb{R}$-algebra automorphism $\theta_{p,q}$ preserves $\mathfrak{spin}(p,q)$,
and thus defines a Lie algebra automorphism of $\mathfrak{spin}(p,q)$. Further, it is easy to verify that the following diagram commutes:
\[
\xymatrix{
\mathfrak{spin}(p,q) \ar[r]^-{\ad}_{\simeq} \ar[d]_{\theta_{p,q}}  & 
\mathfrak{so}(p,q) \ar[d]^{-\trans(\cdot)}  \\ 
\mathfrak{spin}(p,q) \ar[r]^-{\ad}_{\simeq}   & 
\mathfrak{so}(p,q)
},
\]
where 
$\mathfrak{so}(p,q)=\{X\in M(p+q,\R)\mid 
\!^{t}XI_{p,q}+I_{p,q}X=0
\}$.
Thus $\theta_{p,q}$ defines a Cartan involution of $\mathfrak{spin}(p,q)$. Let
\begin{equation*}
\mathfrak{spin}(p,q)=\mathfrak{k}_{\mathfrak{spin}(p,q)}+\mathfrak{p}_{\mathfrak{spin}(p,q)}    
\end{equation*}
be the Cartan decomposition with respect to $\theta_{p,q}$. In the realization \eqref{def:spin-lie} of $\mathfrak{spin}(p,q)$, the $\R$-vector space $\mathfrak{p}_{\mathfrak{spin}(p,q)}$ is spanned by the elements $e_{i}^{+}e_{j}^{-}$. In the case $(p,q)=(n,1)$, 
the following holds. This lemma was effectively used in Section~\ref{section:clifford-hurwitz-radon}:

\begin{lemma}
    \label{lemma:c(n)-c(n,1)+}
    There exists an $\R$-algebra isomorphism 
    $\eta\colon C(n)\rightarrow C^{+}(n,1)$ such that $\eta(\R^{n})= \mathfrak{p}_{\mathfrak{spin}(n,1)}$. 
\end{lemma}

\begin{proof}
    We define an $\R$-algebra homomorphism
    $\eta\colon C(n)\rightarrow C^{+}(n,1)$ by $\eta(e_{i}^{+}):=e_{i}^{+}e_{1}^{-}$.
    It is immediate that $\eta$ is surjective,
    and thus $\eta$ is an isomorphism since $\dim_{\R}C(n) = \dim_{\R}C^{+}(n,1)$ holds. It is obvious from the construction
    of $\eta$
    that $\eta(\R^{n}) = \mathfrak{p}_{\mathfrak{spin}(n,1)}$.
\end{proof}

\subsection{Remarks on finite-dimensional representations of indefinite spin groups}
\label{appendix:spin-representation}
In this subsection, we prove, for general indefinite spin groups $Spin(p,q)$,
the facts (Facts~\ref{fact:index_for_spin1} and~\ref{fact:half-int-rep-dim})
used in Section~\ref{section:half-int-rep-spin-rep}
concerning finite-dimensional representations of $Spin(n,1)$.
The corresponding statements are given by
Propositions~\ref{proposition-dim-(p.q)} and~\ref{prop:index_of_spin_rep},
which correspond to Facts~\ref{fact:half-int-rep-dim}
and~\ref{fact:index_for_spin1}, respectively.

We begin by reviewing the necessary preliminaries.
For $p+q \geq 3$, the finite-dimensional complex irreducible representations
of $Spin(p,q)$ are in one-to-one correspondence, via the unitary trick,
with the finite-dimensional irreducible representations of $\mathfrak{so}(p+q,\C)$.
Moreover, there is a one-to-one correspondence between such representations and the
dominant integral weights of the root system of $\mathfrak{so}(p+q,\C)$, obtained by
taking the highest weight of each representation.

In what follows, we recall the root system $\Delta$ of $\mathfrak{so}(p+q,\C)$,
the choice of a fundamental system $\Pi$, and the action of the Weyl group $W$.
The notation follows Bourbaki~\cite[PLATES~II and~IV]{Bourbaki_Lie_4_to_6}.
We also review the description of the dominant integral weights corresponding to
irreducible representations $\pi\colon Spin(p,q)\to GL(N,\C)$ satisfying $\pi(-1)=-I_{N}$.
We put 
\begin{equation*}
\Lambda_{p,q}:=\{[\pi]\in \irr(Spin(p,q))\mid \pi(-1)=-\id\},
\end{equation*}
where $\irr(Spin(p,q))$ denotes 
the set of 
equivalence classes of irreducible finite-dimensional representations of $Spin(p,q)$. 
Here, $\Lambda_{n,1}$ is referred to as $\Lambda_n$ in
Section~\ref{section:half-int-rep-spin-rep}.

\underline{The case $p+q=2l+1$ $(l\in\N)$.}
Let $\varepsilon_{1},\ldots,\varepsilon_{l}$ be the standard basis of $\R^{l}$.
The root system $\Delta$ of $\mathfrak{so}(p+q,\C)$ and a fundamental system
$\Pi$ can be described as follows:
\begin{align*}
\Delta &= \{\pm \varepsilon_{i}\pm \varepsilon_{j}\mid 1\leq i<j\leq l\} \cup 
\{\pm \varepsilon_{i}\mid 1\leq i\leq l\}, \\
\Pi&= \{\varepsilon_{1}-\varepsilon_{2},\ \ldots,\ 
 \varepsilon_{l-1}-\varepsilon_{l},\ 
\varepsilon_{l}\}.
\end{align*}
The Weyl group $W = S_{l}\ltimes(\Z/2\Z)^{l}$ acts on $\R^{l}$ by all permutations of the components
and all sign changes.
Moreover, a vector $(\lambda_{1},\ldots,\lambda_{l})\in\R^{l}$ is dominant integral if and only if
\begin{equation*}
\lambda_1\geq \lambda_2\geq \cdots \geq \lambda_l \geq 0 
\quad \text{and} \quad 
\lambda_1,\ldots,\lambda_l \in \Z,
\end{equation*}
or
\begin{equation}
\label{eq:type-B-halfint}
\lambda_1\geq \lambda_2\geq \cdots \geq \lambda_l \geq 0
\quad \text{and} \quad 
\lambda_1,\ldots,\lambda_l \in \tfrac{1}{2}+\Z.
\end{equation}

\underline{The case $p+q=2l$ $(l\in\N)$.}
Let $\varepsilon_{1},\ldots,\varepsilon_{l}$ be the standard basis of $\R^{l}$.
Then the root system $\Delta$ of $\mathfrak{so}(p+q,\C)$ and a fundamental system
$\Pi$ can be described as follows:
\begin{align*}
\Delta &= \{\pm \varepsilon_i \pm \varepsilon_j \mid 1\leq i<j\leq l\}, \\
\Pi&=\{\varepsilon_1-\varepsilon_2,\ \ldots,\ 
\varepsilon_{l-1}-\varepsilon_l,\ 
\varepsilon_{l-1}+\varepsilon_l\}.
\end{align*}
The Weyl group $W=S_{l}\ltimes(\Z/2\Z)^{l-1}$ acts on $\R^{l}$ by all permutations of the components
and all even sign changes.
Moreover, a vector $(\lambda_{1},\ldots,\lambda_{l})\in\R^{l}$ is dominant integral if and only if
\[
\lambda_{1}\geq \cdots\geq \lambda_{l-1}\geq |\lambda_{l}|\geq 0
\quad \text{and} \quad 
\lambda_1,\ldots,\lambda_{l}\in \Z,
\]
or
\begin{equation}
\label{eq:type-D-halfint}
\lambda_{1}\geq \cdots\geq \lambda_{l-1}\geq |\lambda_{l}|\geq 0
\quad \text{and} \quad 
\lambda_1,\ldots,\lambda_{l}\in \tfrac{1}{2}+\Z.
\end{equation}

The dominant integral weights associated with the representations belonging to $\Lambda_{p,q}$
can be described as follows (see, e.g., \cite[Prop.~23.13~(iii)]{Fulton-Harris}):

\begin{lemma} 
    \label{lemma:half-int-rep-highest-weight}
The representations belonging to $\Lambda_{p,q}$ are in one-to-one correspondence
(via the highest weight theory and the unitary trick)
with the dominant integral weights $(\lambda_{1},\ldots,\lambda_{l})$
satisfying \eqref{eq:type-B-halfint} when $p+q=2l+1$,
and \eqref{eq:type-D-halfint} when $p+q=2l$.
\end{lemma}

The following fact is known 
about the dimensions of finite-dimensional representations
belonging to $\Lambda_{p,q}$;
Fact~\ref{fact:half-int-rep-dim} is a special case with $q=1$.
\begin{proposition}
    \label{proposition-dim-(p.q)}
   The dimension of any representation belonging to $\Lambda_{p,q}$
is divisible by $2^{\lceil\frac{p+q}{2}\rceil-1}$.
\end{proposition}
\begin{proof}
    Depending on the parity of $p+q$, we define a normal subgroup $W_{0}$  
of the Weyl group $W$ by  
\[
W_{0} =
\begin{cases}
(\mathbb{Z}/2\mathbb{Z})^{l} & \text{if } p+q = 2l+1, \\
(\mathbb{Z}/2\mathbb{Z})^{l-1} & \text{if } p+q = 2l.
\end{cases}
\]

Let $\pi$ be the $\mathfrak{so}(p+q,\C)$-module with highest weight $\lambda$. 
Then any weight $\mu$ of $\pi$ 
can be written in the form  
$\lambda - \sum_{\alpha \in \Pi} n_{\alpha} \alpha$,  
where $n_{\alpha} \in \mathbb{N}$.  
Therefore, by Lemma~\ref{lemma:half-int-rep-highest-weight},  
if $\lambda$ corresponds to $\Lambda_{p,q}$,  
then every weight $\mu$ of $\pi$ lies in $(1/2 + \mathbb{Z})^{l}$.

This observation implies that the above subgroup $W_{0}$ acts freely  
on the set of all weights of $\pi$.  
Hence, the dimension of the representations belonging to $\Lambda_{p,q}$ is divisible by  
$\# W_{0} = 2^{\lceil (p+q)/2 \rceil - 1}$.
\end{proof}

Next, we review some real representation-theoretic properties of finite-dimensional irreducible representations $\pi$ of $Spin(p,q)$.
We denote by $\bar{\pi}$, $\pi^{\vee}$, and $\pi^{*}$ the complex conjugate representation, the dual representation, and the conjugate dual representation of $\pi$, respectively.

Fix $\dagger \in \{-, \vee, *\}$.
For a representation $\pi_{\lambda}$ with dominant integral weight
$\lambda = (\lambda_{1}, \ldots, \lambda_{l})$, $\pi_{\lambda}^{\dagger}$ is also irreducible.
The dominant integral weight corresponding to $\pi_{\lambda}^{\dagger}$ is given by
\[
\begin{cases}
\lambda, & \text{if } p+q = 2l+1, \\[4pt]
\lambda \text{ or } (\lambda_{1}, \ldots, \lambda_{l-1}, -\lambda_{l}), & \text{if } p+q = 2l.
\end{cases}
\]
Here, which of $\lambda$ or $(\lambda_{1}, \ldots, \lambda_{l-1}, -\lambda_{l})$ occurs does not depend on the dominant integral weight $\lambda$, but depends on the signature $(p,q)$ and $\dagger\in \{-, \vee, *\}$.
For details, see, for example, \cite[\S7, Prop.~3]{Oni04} when $\dagger = \vee$, and \cite[\S8, Thm.~3]{Oni04} when $\dagger = -$.
In particular, the following statement follows from Lemma~\ref{lemma:half-int-rep-highest-weight}.
\begin{proposition}
    \label{prop:self-dual-conjugate}
    Fix $\dagger\in \{-,\vee,*\}$. 
    If $p+q$ is odd, then
    every $\pi \in \Lambda_{p,q}$ is equivalent to  $\pi^\dagger$. 
    If $p+q$ is even, then either every $\pi \in \Lambda_{p,q}$ is equivalent to  $\pi^\dagger$, or every $\pi \in \Lambda_{p,q}$ is not equivalent to  $\pi^\dagger$.
\end{proposition}

Next, we consider the case where $\dagger\in\{-,\vee\}$ and $\pi^{\dagger}\simeq \pi$ for $\pi\in \Lambda_{p,q}$,
and describe some property of  $\ind^{\dagger}\pi$ given in  \eqref{def:index}.
\begin{proposition}\label{proposition:index_of_spin_rep}
    Fix $\dagger\in \{-,\vee\}$
    and assume that 
    $\pi\in \Lambda_{p,q}$ is equivalent to
    $\pi^{\dagger}$.
    \begin{itemize}
        \item Suppose $p+q$ is odd.  
        Then 
        $\ind^{\dagger}\pi = \ind^{\dagger}S$.         
        \item Suppose $p+q$ is even.
        Then  $\ind^{\dagger}\pi = \ind^{\dagger}S_{1}$.
    \end{itemize}
    Here, $S$ and $S_{1},S_{2}$ denote the spin representation
    and semispin representations
    of 
    $Spin(p,q)$, respectively. 
\end{proposition}

For the proof of Proposition~\ref{proposition:index_of_spin_rep}, we use the following lemma concerning $\ind^{\dagger}\pi$, which goes back to É.~Cartan.
\begin{lemma}
\label{lemma:index_of_Cartan_product}
    Fix $\dagger\in\{-,\vee\}$.
Let $\lambda,\mu$ be dominant integral weights such that
$\pi_{\lambda}^{\dagger}\simeq \pi_{\lambda}$ and $\pi_{\mu}^{\dagger}\simeq \pi_{\mu}$.
Then we have $\pi_{\lambda+\mu}^{\dagger}\simeq \pi_{\lambda+\mu}$, and
\[
\ind^{\dagger}\pi_{\lambda+\mu}
= (\ind^{\dagger}\pi_{\lambda})(\ind^{\dagger}\pi_{\mu}).
\]
\end{lemma}
This lemma can be proved by observing that $\pi_{\lambda+\mu}$ occurs with multiplicity one
in the tensor product representation $\pi_{\lambda}\otimes \pi_{\mu}$.
For the case $\dagger=-$, see Iwahori~\cite[Lem.~10]{Iwahori-real-rep}.
The case $\dagger=\vee$ can be proved in the same way.

Lemma~\ref{lemma:index_of_Cartan_product} was used to reduce the computation of $\ind^{\dagger}\pi$
for irreducible representations to that for the fundamental representations.
We adopt the same idea here to prove Proposition~\ref{proposition:index_of_spin_rep}.
For this purpose, 
let us recall the description of the fundamental representations of $Spin(p,q)$.
In what follows, we denote by $V=\R^{p+q}$ the standard real representation of $Spin(p,q)$.

\underline{The case $p+q=2l+1$:}
The fundamental weights corresponding to the fundamental system $\Pi$
and the corresponding complex representations of $Spin(p,q)$ are given by
\begin{align*}
    \omega_{i}&=\varepsilon_{1}+\cdots+\varepsilon_{i},\quad
    \pi_{\omega_{i}}=(\wedge^{i}V)\otimes_{\R}\C\ \ (i=1,\ldots,l-1),\\
    \omega_{l}&=\tfrac{1}{2}(\varepsilon_{1}+\cdots+\varepsilon_{l}),\quad
    \pi_{\omega_{l}}=S\ \ (\text{the spin representation}).
\end{align*}
In particular, for each $\dagger\in\{-,\vee\}$ we have
\begin{equation}
    \label{eq:index-fundamental-B}
    \ind^{\dagger} \pi_{\omega_{i}}=1\quad (i=1,\ldots,l-1).
\end{equation}

\underline{The case $p+q=2l$:} 
The fundamental weights corresponding to the fundamental system $\Pi$
and the corresponding complex representations of $Spin(p,q)$ are given by
\begin{align*}
    \omega_{i}&=\varepsilon_{1}+\cdots+\varepsilon_{i},\quad
    \pi_{\omega_{i}}=(\wedge^{i}V)\otimes_{\R}\C\ \ (i=1,\ldots,l-2),\\
    \omega_{l-1}&=\tfrac{1}{2}(\varepsilon_{1}+\cdots+\varepsilon_{l-1}-\varepsilon_{l}),\quad
    \pi_{\omega_{l-1}}=S_{1},\\
    \omega_{l}&=\tfrac{1}{2}(\varepsilon_{1}+\cdots+\varepsilon_{l-1}+\varepsilon_{l}),\quad
    \pi_{\omega_{l}}=S_{2}.
\end{align*}
Since the definition of the semispin representations $S_{1}$ and $S_{2}$ involves
an ambiguity in labeling, it is more precise to say that $\pi_{\omega_{l-1}}$
and $\pi_{\omega_{l}}$ are denoted by $S_{1}$ and $S_{2}$, respectively.
Although not a fundamental representation, note that the representation
$\pi_{\omega_{l-1}+\omega_{l}}$ corresponding to $\omega_{l-1}+\omega_{l}$ is
$(\wedge^{l-1}V)\otimes_{\R}\C$.
From the above, for each $\dagger\in\{-,\vee\}$, 
\begin{equation}
    \label{eq:index-fundamental-D}
    \ind^{\dagger}\pi_{\omega_{i}}=1\quad (i=1,\ldots,l-2),\quad
    \ind^{\dagger}\pi_{\omega_{l-1}+\omega_{l}}=1.
\end{equation}

We are ready to show Proposition~\ref{proposition:index_of_spin_rep}. 
\begin{proof}[Proof of Proposition~\ref{proposition:index_of_spin_rep}]
Let $\lambda$ be the dominant integral weight corresponding to $\pi\in \Lambda_{p,q}$, and write
$\lambda=\sum_{k=1}^{l}a_{k}\omega_{k}$ with $a_{1},\ldots,a_{l}\in\N$.

When $p+q=2l+1$, by Lemma~\ref{lemma:half-int-rep-highest-weight}, note that $a_{l}$ is odd.
Hence, by Lemma~\ref{lemma:index_of_Cartan_product} and \eqref{eq:index-fundamental-B}, we obtain
\begin{equation*}
\ind^{\dagger}\pi=(\ind^{\dagger}\pi_{\omega_{l}})^{a_{l}}=\ind^{\dagger}S.
\end{equation*}

When $p+q=2l$, 
by Lemma~\ref{lemma:index_of_Cartan_product} and \eqref{eq:index-fundamental-D}, we have
$\ind^{\dagger}\pi_{\omega_{l-1}}=\ind^{\dagger}\pi_{\omega_{l}}$.
Moreover, by Lemma~\ref{lemma:half-int-rep-highest-weight}, note that
$a_{l-1}+a_{l}$ is odd.
Therefore, again by Lemma~\ref{lemma:index_of_Cartan_product} and
\eqref{eq:index-fundamental-D}, we have
\begin{equation*}
\ind^{\dagger}\pi=(\ind^{\dagger}\pi_{\omega_{l-1}})^{a_{l-1}+a_{l}}
=\ind^{\dagger}S_{1}.
\end{equation*}
This completes the proof. 
\end{proof}

Combining Propositions~\ref{prop:self-dual-conjugate}
and \ref{proposition:index_of_spin_rep}, we obtain the following. Fact~\ref{fact:index_for_spin1} is a special case with $q=1$.
\begin{proposition}\label{prop:index_of_spin_rep}
    Fix $\dagger\in \{-,\vee,*\}$.
    \begin{itemize}
        \item Suppose $p+q$ is odd.  
        Then $\pi^{\dagger}\simeq \pi$ for any $\pi\in \Lambda_{p,q}$.
        Further, if $\dagger\in \{-,\vee\}$, then 
        $\ind^{\dagger}\pi = \ind^{\dagger}S$.         
        \item Suppose $p+q$ is even.
        Then one of the following claims holds:
        \begin{enumerate}[label=$(\alph*)_{\dagger}$]
            \item 
            $\pi^{\dagger}\simeq \pi$
            for any $\pi\in \Lambda_{p,q}$. 
            Further, if $\dagger\in \{-,\vee\}$, then 
            $\ind^{\dagger}\pi = \ind^{\dagger}S_{1}$.  
            \item 
            $\pi^{\dagger}\not\simeq \pi$ 
            for any $\pi\in \Lambda_{p,q}$. 
            Further, $S_{1}^{\dagger} \simeq S_{2}$.
        \end{enumerate}
    \end{itemize}
\end{proposition}

\section*{Acknowledgements.}
We are grateful to Toshiyuki Kobayashi, Takeyoshi Kogiso, Taro Yoshino, Takayuki Okuda,
Masatoshi Kitagawa, and Ryosuke Nakahama  
for their helpful comments and discussion. 

The first named author is supported by the Special Postdoctoral
Researcher Program at RIKEN and JSPS under Grant-in Aid for Scientific Research 
(JP24K16929).
The second named author is supported by the Special Postdoctoral
Researcher Program at RIKEN.


\begin{thebibliography}{10}

\bibitem{Adams_sphere}
J.~F. Adams, \emph{Vector fields on spheres}, Ann. of Math. (2) \textbf{75}
  (1962), 603--632.

\bibitem{Adams_Lax_Phillips65}
J.~F. Adams, P.~D. Lax, and R.~S. Phillips, \emph{On matrices whose real linear
  combinations are non-singular}, Proc. Amer. Math. Soc. \textbf{16} (1965),
  318--322, Correction. Ibid. \textbf{17} (1966), 945--947.

\bibitem{Araki62}
S.~Araki, \emph{On root systems and an infinitesimal classification of
  irreducible symmetric spaces}, J. Math. Osaka City Univ. \textbf{13} (1962),
  1--34.

\bibitem{Yeung_71}
Y.~H. Au-yeung, \emph{On matrices whose nontrivial real linear combinations are
  nonsingular}, Proc. Amer. Math. Soc. \textbf{29} (1971), 17--22.

\bibitem{Benoist96}
Y.~Benoist, \emph{Actions propres sur les espaces homog\`enes r\'{e}ductifs},
  Ann. of Math. (2) \textbf{144} (1996), no.~2, 315--347.

\bibitem{Berger57}
M.~Berger, \emph{Les espaces sym\'{e}triques noncompacts}, Ann. Sci. \'{E}cole
  Norm. Sup. (3) \textbf{74} (1957), 85--177.

\bibitem{BochenskiOkuda16}
M.~Boche\'{n}ski, P.~Jastrz\k{e}bski, T.~Okuda, and A.~Tralle, \emph{Proper
  {$SL(2,\mathbb{R})$}-actions on homogeneous spaces}, Internat. J. Math.
  \textbf{27} (2016), no.~13, 1650106, 10.

\bibitem{BochenskiTralle24}
M.~Boche\'{n}ski and A.~Tralle, \emph{A solution of the problem of standard
  compact {C}lifford--{K}lein forms}, arXiv:2403.10539v3, 2024.

\bibitem{Bochenski-Tralle-transformation-group24}
\bysame, \emph{{S}tandard {C}ompact {C}lifford--{K}lein {F}orms and {L}ie
  {A}lgebra {D}ecompositions}, Transformation Groups (2024).

\bibitem{Bourbaki_Lie_4_to_6}
N.~Bourbaki, \emph{Lie groups and {L}ie algebras. {C}hapters 4--6}, Elements of
  Mathematics (Berlin), Springer-Verlag, Berlin, 2002, Translated from the 1968
  French original by Andrew Pressley.

\bibitem{CalabiMarkus1962}
E.~Calabi and L.~Markus, \emph{Relativistic space forms}, Ann. of Math. (2)
  \textbf{75} (1962), 63--76.

\bibitem{CoMc93}
D.~H. Collingwood and W.~M. McGovern, \emph{Nilpotent orbits in semisimple
  {L}ie algebras}, Van Nostrand Reinhold Mathematics Series, Van Nostrand
  Reinhold Co., New York, 1993.

\bibitem{Corlette_superrigidity}
K.~Corlette, \emph{Archimedean superrigidity and hyperbolic geometry}, Ann. of
  Math. (2) \textbf{135} (1992), no.~1, 165--182.

\bibitem{Culler-Shalen83}
M.~Culler and P.~B. Shalen, \emph{Varieties of group representations and
  splittings of {$3$}-manifolds}, Ann. of Math. (2) \textbf{117} (1983), no.~1,
  109--146.

\bibitem{Danciger-Gueritaud-Kassel-affine-2020}
J.~Danciger, F.~Gu\'eritaud, and F.~Kassel, \emph{Proper affine actions for
  right-angled {C}oxeter groups}, Duke Math. J. \textbf{169} (2020), no.~12,
  2231--2280.

\bibitem{Eckmann43}
B.~Eckmann, \emph{Gruppentheoretischer beweis des satzes von hurwitz-radon
  \"uber die komposition quadratischer formen}, Comment. Math. Helv.
  \textbf{15} (1943), 358--366.

\bibitem{Fulton-Harris}
W.~Fulton and J.~Harris, \emph{Representation theory}, Graduate Texts in
  Mathematics, vol. 129, Springer-Verlag, New York, 1991, A first course,
  Readings in Mathematics.

\bibitem{Ghys95}
\'{E}. Ghys, \emph{D\'{e}formations des structures complexes sur les espaces
  homog\`enes de {$\mathrm{SL}(2,\mathbf{C})$}}, J. Reine Angew. Math.
  \textbf{468} (1995), 113--138.

\bibitem{Goldmannonstandard}
W.~M. Goldman, \emph{Nonstandard {L}orentz space forms}, J. Differential Geom.
  \textbf{21} (1985), no.~2, 301--308.

\bibitem{GueritaudGuichardKasselWienhard2017anosov}
F.~Gu\'{e}ritaud, O.~Guichard, F.~Kassel, and A.~Wienhard, \emph{Anosov
  representations and proper actions}, Geom. Topol. \textbf{21} (2017), no.~1,
  485--584.

\bibitem{HelgasonDiffLieSymm2001}
S.~Helgason, \emph{Differential geometry, {L}ie groups, and symmetric spaces},
  Graduate Studies in Mathematics, vol.~34, American Mathematical Society,
  Providence, RI, 2001, Corrected reprint of the 1978 original.

\bibitem{Hurwitz22}
A.~Hurwitz, \emph{\"{U}ber die {K}omposition der quadratischen {F}ormen}, Math.
  Ann. \textbf{88} (1922), no.~1-2, 1--25.

\bibitem{Iwahori-real-rep}
N.~Iwahori, \emph{On real irreducible representations of {L}ie algebras},
  Nagoya Math. J. \textbf{14} (1959), 59--83.

\bibitem{Kaneyuki_Kozai_1985}
S.~Kaneyuki and M.~Kozai, \emph{Paracomplex structures and affine symmetric
  spaces}, Tokyo J. Math. \textbf{8} (1985), no.~1, 81--98.

\bibitem{kannaka24}
K.~Kannaka, \emph{Counting orbits of certain infinitely generated non-sharp
  discontinuous groups for the anti-de {S}itter space}, Selecta Math. (N.S.)
  \textbf{30} (2024), no.~1, 11.

\bibitem{KannakaKobayashi25}
K.~Kannaka and T.~Kobayashi, \emph{Zariski-dense deformations of standard
  discontinuous groups for pseudo-riemannian homogeneous spaces}, To appear in
  Zariski Dense Subgroups, Number Theory and Geometric Applications, (eds: G.
  Prasad, A. Rapinchuk, B. Sury, and A. Tralle), Infosys Science Foundation
  Series, Springer^^e2^^80^^93Nature, 2026.

\bibitem{KannakaOkudaTojo24}
K.~Kannaka, T.~Okuda, and K.~Tojo, \emph{Zariski-dense discontinuous surface
  groups for reductive symmetric spaces}, Symmetry in geometry and analysis.
  {V}ol. 1. {F}estschrift in honor of {T}oshiyuki {K}obayashi, Progr. Math.,
  vol. 357, Birkh\"auser/Springer, Singapore, 2025, pp.~321--354.

\bibitem{Kassel12}
F.~Kassel, \emph{Deformation of proper actions on reductive homogeneous
  spaces}, Math. Ann. \textbf{353} (2012), no.~2, 599--632.

\bibitem{KasselKobayashi16}
F.~Kassel and T.~Kobayashi, \emph{Poincar\'{e} series for non-{R}iemannian
  locally symmetric spaces}, Adv. Math. \textbf{287} (2016), 123--236.

\bibitem{KasselTholozan}
F.~Kassel and N.~Tholozan, \emph{Sharpness of proper and cocompact actions on
  reductive homogeneous spaces}, arXiv:2410.08179v1, 2024.

\bibitem{Kobayashi89}
T.~Kobayashi, \emph{Proper action on a homogeneous space of reductive type},
  Math. Ann. \textbf{285} (1989), no.~2, 249--263.

\bibitem{Kobayashi1992necessary}
\bysame, \emph{A necessary condition for the existence of compact
  {C}lifford--{K}lein forms of homogeneous spaces of reductive type}, Duke
  Math. J. \textbf{67} (1992), no.~3, 653--664.

\bibitem{Kobayashi96}
\bysame, \emph{Criterion for proper actions on homogeneous spaces of reductive
  groups}, J. Lie Theory \textbf{6} (1996), no.~2, 147--163.

\bibitem{Kobayashi98}
\bysame, \emph{Deformation of compact {C}lifford--{K}lein forms of
  indefinite-{R}iemannian homogeneous manifolds}, Math. Ann. \textbf{310}
  (1998), no.~3, 395--409.

\bibitem{Kobayashi-unlimit}
\bysame, \emph{Discontinuous groups for non-{R}iemannian homogeneous spaces},
  Mathematics unlimited---2001 and beyond, Springer, Berlin, 2001,
  pp.~723--747.

\bibitem{KobayashiYoshino05}
T.~Kobayashi and T.~Yoshino, \emph{Compact {C}lifford--{K}lein forms of
  symmetric spaces---revisited}, Pure Appl. Math. Q. \textbf{1} (2005), no.~3,
  Special Issue: In memory of Armand Borel. Part 2, 591--663.

\bibitem{Lam67}
K.~Y. Lam, \emph{Construction of nonsingular bilinear maps}, Topology
  \textbf{6} (1967), 423--426.

\bibitem{Margulis-freegroup-anounce}
G.~A. Margulis, \emph{Free completely discontinuous groups of affine
  transformations}, Dokl. Akad. Nauk SSSR \textbf{272} (1983), no.~4, 785--788.

\bibitem{Margulis-freegroup-1984}
\bysame, \emph{Complete affine locally flat manifolds with a free fundamental
  group}, vol. 134, 1984, Automorphic functions and number theory, II,
  pp.~190--205.

\bibitem{Margulis_discrete_subgroup}
\bysame, \emph{Discrete subgroups of semisimple {L}ie groups}, Ergebnisse der
  Mathematik und ihrer Grenzgebiete (3) [Results in Mathematics and Related
  Areas (3)], vol.~17, Springer-Verlag, Berlin, 1991.

\bibitem{Oku13}
T.~Okuda, \emph{Classification of semisimple symmetric spaces with proper
  {$SL(2,\mathbb{R})$}-actions}, J. Differential Geom. \textbf{94} (2013),
  no.~2, 301--342.

\bibitem{Okuda17}
\bysame, \emph{Abundance of nilpotent orbits in real semisimple {L}ie
  algebras}, J. Math. Sci. Univ. Tokyo \textbf{24} (2017), no.~3, 399--430.

\bibitem{Oni04}
A.~L. Onishchik, \emph{Lectures on real semisimple {L}ie algebras and their
  representations}, ESI Lectures in Mathematics and Physics, European
  Mathematical Society (EMS), Z\"{u}rich, 2004.

\bibitem{OshimaSekiguchi84}
T.~\={O}shima and J.~Sekiguchi, \emph{The restricted root system of a
  semisimple symmetric pair}, Group representations and systems of differential
  equations ({T}okyo, 1982), Adv. Stud. Pure Math., vol.~4, North-Holland,
  Amsterdam, 1984, pp.~433--497.

\bibitem{Radon22}
J.~Radon, \emph{Lineare {S}charen orthogonaler {M}atrizen}, Abh. Math. Sem.
  Univ. Hamburg \textbf{1} (1922), no.~1, 1--14.

\bibitem{Teduka2008}
K.~Teduka, \emph{Proper actions of {$\mathrm{SL}(2,\mathbb{R})$} on
  {$\mathrm{SL}(n,\mathbb{ R})$}---homogeneous spaces}, J. Math. Sci. Univ.
  Tokyo \textbf{15} (2008), no.~1, 1--13.

\bibitem{Teduka2008PJA}
\bysame, \emph{Proper actions of {$\mathrm{SL}(2,\mathbf{C})$} on irreducible
  complex symmetric spaces}, Proc. Japan Acad. Ser. A Math. Sci. \textbf{84}
  (2008), no.~7, 107--111.

\bibitem{Thurston-3fold}
W.~P. Thurston, \emph{Three-{D}imensional {G}eometry and {T}opology. {V}ol. 1},
  Princeton Mathematical Series, vol.~35, Princeton University Press,
  Princeton, NJ, 1997.

\bibitem{Tojo2019Classification}
K.~Tojo, \emph{Classification of irreducible symmetric spaces which admit
  standard compact {C}lifford--{K}lein forms}, Proc. Japan Acad. Ser. A Math.
  Sci. \textbf{95} (2019), no.~2, 11--15.

\bibitem{Varadarajan_supersymmetry}
V.~S. Varadarajan, \emph{Supersymmetry for mathematicians: an introduction},
  Courant Lecture Notes in Mathematics, vol.~11, New York University, Courant
  Institute of Mathematical Sciences, New York; American Mathematical Society,
  Providence, RI, 2004.

\bibitem{Warner_harmonic_analysis_I}
G.~Warner, \emph{Harmonic analysis on semi-simple {L}ie groups. {I}}, Die
  Grundlehren der mathematischen Wissenschaften, vol. Band 188,
  Springer-Verlag, New York-Heidelberg, 1972.

\bibitem{Yokota_g2}
I.~Yokota, \emph{Non-compact simple {L}ie group {$G\sb{2}'$} of type
  {$G\sb{2}$}}, J. Fac. Sci. Shinshu Univ. \textbf{12} (1977), no.~1, 45--52.

\bibitem{Zeghib98}
A.~Zeghib, \emph{On closed anti-de {S}itter spacetimes}, Math. Ann.
  \textbf{310} (1998), no.~4, 695--716.

\end{thebibliography}
\end{document}